\documentclass[11pt,oneside,reqno]{amsart}
\usepackage{geometry}
\usepackage[utf8]{inputenc}
\usepackage{amstext,latexsym,amsbsy,amsmath,amssymb,amsthm,mathtools,relsize,geometry,enumerate}
\usepackage{hyperref}
\hypersetup{
  colorlinks   = true, %Colours links instead of ugly boxes
  urlcolor     = blue, %Colour for external hyperlinks
  linkcolor    = red, %Colour of internal links
  citecolor   = red %Colour of citations
}

\usepackage{mathtools,enumerate,enumitem}

\usepackage[capitalize,nameinlink,noabbrev]{cleveref}

\usepackage{graphicx}
\usepackage{epstopdf}
\setkeys{Gin}{width=\linewidth,totalheight=\textheight,keepaspectratio}
\graphicspath{{./images/}}
\usepackage{comment}

\usepackage{multicol}
\usepackage{booktabs}
\usepackage{mathrsfs}
\usepackage{color}

\newcommand{\rbb}{\mathbb{R}}

\renewcommand{\L}{\mathcal{L}}

\newcommand{\W}{\mathcal{W}}

\newcommand{\Pcal}{\mathcal{P}}
\newcommand{\X}{\mathbf{X}}

\newcommand{\B}{\mathcal{B}}
\newcommand{\la}{\langle}
\newcommand{\ra}{\rangle}

\newcommand{\nt}{\notag}

\newcommand{\qhat}{\hat{q}}
\newcommand{\phat}{\hat{p}}

\newcommand{\qb}{\mathbf{q}}
\newcommand{\pb}{\mathbf{p}}

\renewcommand{\div}{\text{div}}
\newcommand{\mi}{\wedge}
\newcommand{\D}{\mathcal{D}}

\renewcommand{\d}{\text{d}}

\newcommand{\f}{\varphi}

\newcommand{\grad}{\nabla}

\newcommand{\E}{\mathbb{E}}

\renewcommand{\P}{\mathbb{P}}

\newcommand{\U}{U}
\newcommand{\G}{G}

\newcommand{\close}{\!\!\!}

\newcommand{\TV}{\textup{TV}}

\theoremstyle{plain}
\newtheorem{theorem}{Theorem}[section]

\newtheorem{lemma}[theorem]{Lemma}

\newtheorem{proposition}[theorem]{Proposition}

\theoremstyle{definition}
\newtheorem{definition}[theorem]{Definition}

\newtheorem{remark}[theorem]{Remark}

\numberwithin{equation}{section}

\title{Trend to equilibrium and Newtonian limit for the relativistic Langevin equation with singular potentials}

\author{Manh Hong Duong$^1$ and Hung Dang~Nguyen$^2$}

\address{$^1$ School of Mathematics, University of Birmingham, Birmingham, UK}

\address{$^2$ Department of Mathematics, University of Tennessee, Knoxville, Tennessee, USA}

\begin{document}

\begin{abstract}

We study a system of interacting particles in the presence of the relativistic kinetic energy, external confining potentials, singular repulsive forces as well as a random perturbation through an additive white noise. In comparison with the classical Langevin equations that are known to be exponentially attractive toward the unique statistically steady states, we find that the relativistic systems satisfy algebraic mixing rates of any order. This relies on the construction of Lyapunov functions adapting to previous literature developed for irregular potentials. We then explore the Newtonian limit as the speed of light tends to infinity and establish the validity of the approximation of the solutions by the Langevin equations on any finite time window.  

\end{abstract}

\maketitle

\section{Introduction} \label{sec:intro}

Given an integer $N\ge 1$, we consider a family of $N$ interacting particles evolving in $\rbb^d$ and governed by \textit{the relativistic underdamped Langevin dynamics}
\begin{align} \label{eqn:rLE:N-particle:original}
\d\, q_i(t) &= \grad K(p_i(t))\d t,\quad  i=1,\dots,N, \notag\\
\d\, p_i(t) & = -\gamma\grad K(p_i(t))\d t -\grad \U(q_i(t))\d t- \sum_{j\neq i}\grad \G\big(q_i(t)-q_j(t)\big) \d t +\sqrt{2\gamma} \,\d W_{i}(t).
\end{align}
In \eqref{eqn:rLE:N-particle:original}, the bivariate process $(q_i,p_i)$ represents the position and the momentum of the $i^{\text{th}}$ particle. Equation \eqref{eqn:rLE:N-particle:original} posits that the particle moves under the influence of four different forces, namely,  (i) a friction given by $-\gamma \nabla K(p_i)$, where
\begin{align} \label{form:K(p)}
K(p)=c\sqrt{m^2c^2+|p|^2},
\end{align}
is the relativistic kinetic energy, $c>0$ represents the speed of light, $m>0$ denotes the particle's mass, and $\gamma>0$ plays the role of a friction coefficient; (ii) an external force $-\nabla U(q_i)$ which is the derivative of a smooth confining potential $U:\rbb^d\to[1,\infty)$ having polynomial growths; (iii) an interacting force $\sum_{j\neq i}\nabla G (q_i-q_j)$ where the interaction potential $G:\rbb^d\setminus\{0\}\to\rbb$ being a singular function satisfying certain properties, and (iv) finally a random perturbation $\sqrt{2\gamma}\d W_i$, where $\{W_i\}_{i=1}^N$ is a collection of i.i.d standard Brownian motions in $\rbb^d$.

Historically, the single-particle system ($N=1$, $G=0$) was originally proposed in \cite{debbasch1997relativistic} as an Ornstein-Uhlenbeck process that is compatible with Einstein's theory of special relativity \cite{einstein1905enertia,einstein1905electrodynamics}. Since then many other relativistic stochastic models have also been introduced, including those in \cite{dunkel2005atheory,dunkel2005theory}, which are similar to the one in  \cite{debbasch1997relativistic} but with multiplicative noises. There is now a large body of literature devoted to studying relativistic stochastic models, see for instance \cite{alcantara2011relativistic,chevalier2008relativistic,
debbasch2004diffusion,debbasch1998diffusion,
duong2015formulation,felix2013newtonian,
haba2009relativistic,
haba2009relativisticII,haba2010energy,
pal2020stochastic} and the survey articles \cite{debbasch2007relativistic,dunkel2009relativistic} for a detailed exposition of the topics. Physically, whenever a relativistic model is introduced, it is desirable to achieve at least two fundamental properties: it possesses a thermodynamically correct equilibrium state, which is the relativistic Maxwell-Boltzmann
distribution and (ii) it recovers the classical  (non-relativistic) counterpart in the Newtonian limit, that is when the speed of light tends to infinity. Rigorously establishing these two properties is of special interest in mathematics over the last decades, see the aforementioned papers for the relativistic Langevin-type processes, as well as \cite{calogero2004newtonian,strain2010global,mai2022newtonian} for other relativistic kinetic models including the relativistic Boltzmann equation and the relativistic Euler-Poisson system.

Concerning the $N$-particle relativistic Langevin equation \eqref{eqn:rLE:N-particle:original}, we observe that as $c\to\infty$
\begin{align*}
\grad K(p)=\frac{cp}{\sqrt{m^2c^2+|p|^2}} \to \frac{p}{m}.
\end{align*}
Thus, upon letting $c\to\infty$ in the Newtonian (non-relativistic) limit, \eqref{eqn:rLE:N-particle:original} is formally reduced to a system of interacting particles governed by the classical Langevin dynamics
\begin{align} \label{eqn:LE:original}
\d\, q_i(t) &= \frac{1}{m}p_i(t) \d t, \notag\\
\d\, p_i(t) & = - \frac{\gamma}{m} p_i(t) \d t -\grad \U(q_i(t))\d t- \sum_{j\neq i}\grad \G\big(q_i(t)-q_j(t)\big) \d t +\sqrt{2\gamma} \,\d W_{i}(t).
\end{align}

On the one hand, we note that the large time asymptotic behavior of \eqref{eqn:LE:original} is relatively well-understood. More specifically, under suitable assumptions on $U$ and $G$, \eqref{eqn:LE:original} is exponentially attractive toward the unique Boltzmann-Gibbs invariant probability measure $\pi_{\text{LE}}$ given by
\begin{align*} 
\pi_{\text{LE}}(\d\qb,\d\pb)=\frac{1}{Z_{\text{LE}}}\exp\Big\{-\Big(\frac{1}{2m}|\pb|^2+\sum_{i=1}^N U(q_i)+\close\sum_{1\le i<j\le N}\close G(q_i-q_j)\Big) \Big\}\d \qb\d \pb,
\end{align*}
where $Z_{\text{LE}}$ is the normalization constant, $\qb=(q_1,\dots,q_N)$, and $\pb=(p_1,\dots,p_N)$ \cite{bolley2018dynamics, cooke2017geometric,herzog2017ergodicity,lu2019geometric,
pavliotis2014stochastic}. 

On the other hand, while it can be shown that the stationary solution of \eqref{eqn:rLE:N-particle:original} is the expected relativistic Maxwellian distribution
\begin{align} \label{form:pi_RLE}
\pi_{ \text{RLE} }(\d\qb,\d\pb)=\frac{1}{Z}\exp\Big\{-\Big(\sum_{i=1}^N K(p_i)+\sum_{i=1}^N U(q_i)+\close\sum_{1\le i<j\le N}\close G(q_i-q_j)\Big) \Big\}\d \qb\d \pb,
\end{align}
to the best of the authors' knowledge, the topic of ergodicity for \eqref{eqn:rLE:N-particle:original} in the presence of singular potentials has not been investigated, nor has the question of whether \eqref{eqn:rLE:N-particle:original} resembles \eqref{eqn:LE:original} with respect to the Newtonian limit been studied before. Addressing these two important questions is highly nontrivial due to the nonlinearities and singularities (and their interplay) of the relevant potentials.

 Our goal of the paper is thus two-fold: firstly, we seek sufficient conditions on both the confining potential $U$ and the singular potential $G$ so as to produce a mixing rate toward $\pi_{ \text{RLE} }$ for \eqref{eqn:rLE:N-particle:original}. Secondly, we demonstrate that in the Newtonian regime $c\to\infty$, \eqref{eqn:rLE:N-particle:original} is well approximated by \eqref{eqn:LE:original} on any finite time windows. In what follows, we shall summarize our main results, whose rigorous statements are deferred to Section \ref{sec:result}.

\subsection{Summary of main results} \label{sec:intro:result}

With regard to the ergodicity of \eqref{eqn:rLE:N-particle:original}, unlike the Langevin system \eqref{eqn:LE:original} which possesses a geometric mixing rate, we find that provided that the speed of light constant $c$ is sufficiently large, the dynamics \eqref{eqn:rLE:N-particle:original} converges toward $\pi_{ \text{RLE} }$ with at least algebraic rates of any order. This is summarized in the following result:
\begin{theorem} \label{thm:polynomial-mixing:N-particle:meta}
Under suitable assumptions on $U$ and $G$, let $(\qb_c,\pb_c)$ be the solutions of \eqref{eqn:rLE:N-particle:original}. Then, for all $c$ sufficiently large, bounded function $f$, and positive constant $r>0$, 
\begin{align} \label{ineq:polynomial-ergodicity:meta}
\Big|\E f(\qb_c(t),\pb_c(t)) - \int f(\qb,\pb)\pi_{ \textup{RLE} }(\textup{d}\pb,\textup{d}\pb ) \Big| \le \frac{C}{(1+t)^r},\quad t\ge 0,
\end{align}
where $C$ is a positive constant independent of $t$.

\end{theorem}
We refer the reader to Theorem \ref{thm:polynomial-mixing:N-particle} for a precise statement of the above result. We note that in the single-particle case, i.e., $N=1$,
\begin{align*} 
\d\, q(t) &= \grad K(p(t))\d t, \\
\d\, p(t) & = -\gamma \grad K(p(t)) \d t -\grad U(q(t))\d t- \grad G\big(q(t)\big) \d t +\sqrt{2\gamma} \,\d W(t),
\end{align*}
estimate \eqref{ineq:polynomial-ergodicity:meta} applies to a wide range of smooth polynomial potential $U$ and singular potential $G$, including the instances of the Lennard-Jones and the Coulomb functions. See Assumptions \nameref{cond:U} and \nameref{cond:G1}. However, in the multi-particle situation, i.e., $N\ge2$, owing to the complication between the kinetic relativistic energy $K(p)$ and the interaction potential $G(q)$, \eqref{ineq:polynomial-ergodicity:meta} requires a restriction on the behavior of $G(q)$ near the origin, cf. Assumption \nameref{cond:G2}.

Turning to the Newtonian limit as $c$ tends to infinity, we obtain the following result giving the validity of the approximation of \eqref{eqn:rLE:N-particle:original} by \eqref{eqn:LE:original} on any compact intervals.
\begin{theorem} \label{thm:newtonian-limit:meta}
Under suitable assumptions on $U$ and $G$, let $(\qb_c(t),\pb_c(t))$ and $(\qb(t),\pb(t))$ respectively be the solutions of \eqref{eqn:rLE:N-particle:original} and \eqref{eqn:LE:original}.

1. (Multi-particle system) For all $N\ge 2$, $T>0$, and $\xi>0$,
\begin{align} \label{lim:newtonian-limit:N-particle:prob:meta}
\P\Big( \sup_{t\in[0,T]}\Big[ |\qb_c(t)-\qb(t)|+|\pb_c(t)-\pb(t)|\Big]>\xi\Big) \rightarrow 0,\quad c\rightarrow \infty.
\end{align}

2. (Single-particle system) When $N=1$, for all $T>0$ and $n\ge 1$,
\begin{align} \label{lim:newtonian-limit:single-particle:L^p:meta}
\E \sup_{t\in[0,T]}\big[ |q_c(t)-q(t)|^n+|p_c(t)-p(t)|^n\big]\rightarrow 0,\quad c\rightarrow \infty.
\end{align}

\end{theorem}
In \eqref{lim:newtonian-limit:N-particle:prob:meta} and \eqref{lim:newtonian-limit:single-particle:L^p:meta}, we adopt the convention that $(\qb,\pb)$ and $(q,p)$ respectively refer to the multi-particle and single-particle processes. The rigorous version of Theorem \ref{thm:newtonian-limit:meta} is formulated in Theorem \ref{thm:newtonian-limit}. We note that while the limit \eqref{lim:newtonian-limit:single-particle:L^p:meta} in the single-particle case is a convergence in expectation, the analogue \eqref{lim:newtonian-limit:N-particle:prob:meta} for $N\ge 2$ is only a convergence in probability. This amounts to the fact that the case $N\ge 2$ induces further difficulties surrounding the interaction between the kinetic relativistic energy and the repulsive force. Nevertheless, unlike the mixing result \eqref{ineq:polynomial-ergodicity:meta} where we have to impose further assumption on the singularities, both limits \eqref{lim:newtonian-limit:N-particle:prob:meta} and \eqref{lim:newtonian-limit:single-particle:L^p:meta} are applicable to a large class of irregular potentials, of which, the Lennard-Jones and the Coulomb functions are members.

\subsection{Related literature} \label{sec:intro:literature}
It is important to note the physics background of equation \eqref{eqn:rLE:N-particle:original}, which is a special case of the more general relativistic model
\begin{align} \label{eqn:rLE:N-particle:general}
\d\, q_i(t) &= \nabla K(p_i)\d t, \notag \\
\d\, p_i(t) & = -\gamma D(p_i)\nabla K(p_i) \d t + \gamma\div(D(p_i))\d t  +\sqrt{2\gamma D(p_i)}\d W_{i}(t)\notag \\
& \qquad -\grad \U(q_i(t))\d t- \sum_{j\neq i}\grad \G\big(q_i(t)-q_j(t)\big) \d t,
\end{align}
where $D$ is the diffusion matrix. The specific structure of \eqref{eqn:rLE:N-particle:general}, particularly through the manifestation of the (relativistic) fluctuation-dissipation relation (the scaling of the matrix $\gamma D$ in the drift and the diffusion terms), guarantees that, among other properties, it possesses the relativistic Maxwellian distribution as a stationary solution.  In the literature, there are two popular choices of $D$, corresponding to considering either an additive noise or a multiplicative noise in the model above, namely, 
\[D(p)=I, ~\text{as in \cite{debbasch1997relativistic}}\]
\begin{align} \label{form:D(p)}
\text{or}\quad D(p)=\frac{mc}{\sqrt{m^2c^2+|p|^2}}\Big(I+\frac{p\otimes p}{m^2c^2}\Big), ~\text{as in \cite{dunkel2005atheory,dunkel2005theory}}.
\end{align}
%In the literature, there are several papers that study the convergence to the equilibrium and the Newtonian limit for the above relativistic Langevin dynamics and the associated (relativistic) Fokker-Planck equation for either an additive or a multiplicative noise, but in the setting of \textit{single-particle} ($N=1$) \textit{and without singularities} ($G=0$).

The topics of large-time asymptotic and the Newtonian limit for \eqref{eqn:rLE:N-particle:general} and the associated (relativistic) Fokker-Planck equation have been studied in several papers, though in the setting of \textit{single-particle} ($N=1$) \textit{and without singularities} ($G=0$).
In the case of an additive noise ($D=I$), these problems were formally discussed in the original paper \cite{debbasch1997relativistic}. Moreover, an $H$-theorem for the solution process is shown in \cite{barbachoux2001covariant}. A rigorous proof of the exponential convergence toward equilibrium for the spatially homogeneous process is established in \cite{angst2011trends} (this paper also covers the case of a variable $D$). Concerning the instance of a multiplicative noise, i.e., $D=D(p)$ given by \eqref{form:D(p)}, the well-posedness of the Fokker-Planck equation is established in \cite{alcantara2011relativistic} whereas in the absence of potentials $U$, the momentum processes are shown to exhibit a geometric mixing rate \cite{angst2011trends,
felix2013newtonian}. Similar asymptotic behavior is also obtained when the position $q$ is posed on a torus \cite{calogero2012exponential} making use of the hypocoercive method \cite{villani2009hypocoercivity}. In a more recent preprint~\cite{arnold2024trend}, the technique has been extended by means of Lyapunov functionals to successfully treat the case of smooth potential $U$. The Newtonian limit was also investigated rigorously in \cite{felix2013newtonian} where the solution is proven to be arbitrarily close to the Langevin dynamics in $L^1$ norm. 

Regarding $N\ge 2$, relativistic deterministically $N$-particle systems with singular interactions (i.e., \eqref{eqn:rLE:N-particle:original} with $\gamma=0$) have also been studied considerably, often in the context of the mean-field limits and their connections to the relativistic Vlasov systems. See for instance~\cite{chen2020combined,elskens2009vlasov,golse2012mean,
kiessling2008relativistic,lieb1988stability}.

The main novelty in our paper from literature is that we deal with a stochastic system of multiple particles in the presence of irregular potentials. The singularity and its nontrivial interplay with the kinetic energy require much more delicate estimates on the solutions as well as their relations with the classical (non-relativistic) Langevin dynamics. In what follows, we briefly review the methodology that we employ to establish the main results, whose detailed explanations will be carried out in Section \ref{sec:poly-mixing} and Section \ref{sec:newtonian-limit}.

\subsection{Methodology of the proofs} \label{sec:intro:proof}

With regard to the algebraic mixing rate, as captured through Theorem \ref{thm:polynomial-mixing:N-particle:meta}, we adopt the framework of \cite{hairer2009hot}, which in turn was built upon the technique developed in \cite{ bakry2008rate, douc2009subgeometric, fort2005subgeometric}. The proof of estimate \eqref{ineq:polynomial-ergodicity:meta} consists of three main ingredients, namely, the Hormander's condition ensuring the smoothness of the transition probabilities, the solvability of the associated control problem giving the possibility of returning to the center of the phase space, and a suitable Lyapunov function quantifying the convergence rate. See Definitions \ref{def:Hormander}, \ref{def:control-problem} and \ref{def:Lyapunov}, respectively. In particular, the first criterion is a consequence of the classical Hormander's Theorem \cite{hormander1967hypoelliptic} asserting that the phase space may be generated by the collection of vector fields jointly induced by the diffusion and the drifts. Since we are dealing with a finite-dimensional setting, this is a short computation on Lie brackets, cf. Lemma \ref{lem:Hormander}. The second condition on the associated control problem can be established by the Support Theorem \cite{stroock1972degenerate} showing that one can always find appropriate controls allowing for driving the dynamics to any bounded ball. In order to do so, we follow closely the approach of \cite{herzog2017ergodicity} dealing with the same issue for the Langevin dynamics \eqref{eqn:LE:original}. However, since \eqref{eqn:rLE:N-particle:original} is slightly more complicated than \eqref{eqn:LE:original} owing to the appearance of the relatisvistic forcing term, we have to modify the argument therein tailored to our settings, cf. Lemma \ref{lem:control-problem}. The last ingredient is the crucial Lyapunov functions, on which the result of Theorem \ref{thm:polynomial-mixing:N-particle:meta} mainly relies. Traditionally, an exponential convergence rate requires the construction of an energy-like function $V$ satisfying an inequality of the form
\begin{align} \label{ineq:Lyapunov:exponential-mixing}
\frac{\d}{\d t}\E \big[V(\qb_c(t),\pb_c(t))\big] \le  -c_1\E \big[V(\qb_c(t),\pb_c(t))\big]+c_2,\quad t\ge 0.
\end{align}
Although this Lyapunov effect is available for equation \eqref{eqn:LE:original} and related dynamics \cite{cooke2017geometric, duong2024asymptotic,herzog2017ergodicity, lu2019geometric}, the lack of strong dissipation in the $p$-direction as well as the impact of the singularities do not allow access to an estimate of \eqref{ineq:Lyapunov:exponential-mixing}-typed for \eqref{eqn:rLE:N-particle:original}. To circumvent the issue, we resort to the technique of \cite{hairer2009hot} by proving instead the existence of a function $V$ such that
\begin{align} \label{ineq:Lyapunov:polynomial-mixing}
\frac{\d}{\d t}\E\big[ V(\qb_c(t),\pb_c(t)) \big]\le  -c_1\E\big[ V(\qb_c(t),\pb_c(t))^\alpha\big]+c_2,\quad t\ge 0,
\end{align}
for a suitable constant $\alpha\in(0,1)$. See also Remark \ref{rem:polynomial-mixing}, part 2. While the Lyapunov function for the case $N=1$ resembles those of \cite{duong2024asymptotic,lu2019geometric}, the choice of $V$ for the multi-particle situation turns out to be much more challenging. This difficulty arises from the interaction between the relativistic energy and the singular potentials, producing badly behavior terms which are otherwise canceled in the case $N=1$. Nevertheless, their affect can be dominated by imposing further restriction on $G$, ultimately allowing for deducing \eqref{ineq:Lyapunov:polynomial-mixing}. We refer the reader to Section \ref{sec:result:assumption} for the precise assumptions on the nonlinearities and to Section \ref{sec:result:ergodicity} for the rigorous statement of \eqref{ineq:polynomial-ergodicity:meta}, whose proof is presented in Section \ref{sec:poly-mixing}.

The last main result of this article concerns the Newtonian limit as the speed of light $c$ tends to infinity. The proof of Theorem \ref{thm:newtonian-limit:meta} is motivated by the strategy of \cite{duong2024asymptotic,herzog2016small} developed for the small mass limits of second-order SDEs. The general argument essentially consists of two main steps: firstly, we truncate the nonlinearities in \eqref{eqn:rLE:N-particle:original} and \eqref{eqn:LE:original}, resulting in Lipschitz systems, whose convergence in expectations can be handled directly. Then, we remove the Lipschitz constraint by making use of suitable moment bounds on the dynamics. Particularly, with regard to the multi-particle case, instead of controlling the energy of \eqref{eqn:rLE:N-particle:original}, we turn to estimating \eqref{eqn:LE:original} through its Hamiltonian structure. Combining with a careful analysis on the relation between \eqref{eqn:rLE:N-particle:original} and \eqref{eqn:LE:original}, we are able to obtain the convergence in probability \eqref{lim:newtonian-limit:N-particle:prob:meta} while making use of auxiliary inequalities on the singularities. Turning to the simpler single-particle setting, the previous probability argument can be upgraded to a convergence in $L^p$ by facilitating critical moment bounds on \eqref{eqn:rLE:N-particle:original} that maintain the uniformity in the parameter $c$. More specifically, such an estimate has the form
\begin{align} \label{ineq:moment-bound:q_c,p_c}
\lim_{c\to\infty}\E \sup_{t\in[0,T]} \Big(\big[U(q_c(t))+G(q_c(t))\big]^2+\big[U(q_c(t))+G(q_c(t))\big]K(p_c(t))+|p_c(t)|^2\Big)<\infty.
\end{align}
We note that while the left-hand side of \eqref{ineq:moment-bound:q_c,p_c} is bounded when $N=1$, it may be arbitrarily large depending on $c$ when $N\ge 2$, stemming from the same issue on the interplay between the relativistic kinetic and potential energies that we face in the mixing problem. The reader is referred to Theorem \ref{thm:newtonian-limit} for the precise statement of \eqref{lim:newtonian-limit:N-particle:prob:meta}-\eqref{lim:newtonian-limit:single-particle:L^p:meta} and to Section \ref{sec:newtonian-limit} for their proofs.

Lastly, we remark that the results of this article are established for equation \eqref{eqn:rLE:N-particle:general} in the case of identity diffusion matrix $D=I$. This leads to the question whether a similar approach can be applied to the nonlinear $D(p)$ given by \eqref{form:D(p)}. In this regard, we anticipate the difficulties that we face in studying \eqref{eqn:rLE:N-particle:original} will carry over to the setting of \eqref{eqn:rLE:N-particle:general}-\eqref{form:D(p)}. While it is certainly interesting to explore the asymptotic analysis for \eqref{eqn:rLE:N-particle:general}-\eqref{form:D(p)}, we opt for leaving the problems as future work.

\subsection{Organization of the paper}\label{sec:intro:organization}
The rest of the paper is organized as follows: in Section \ref{sec:result}, we introduce the notations as well as the assumptions that we make on the nonlinearities. We also state the main results of the paper, including Theorem \ref{thm:polynomial-mixing:N-particle} on polynomial mixing of \eqref{eqn:rLE:N-particle:original} and Theorem \ref{thm:newtonian-limit} on the validity of the approximation of \eqref{eqn:rLE:N-particle:original} by \eqref{eqn:LE:original} in the Newtonian regime. In Section \ref{sec:poly-mixing}, we detail the construction of Lyapunov functions for \eqref{eqn:rLE:N-particle:original} and prove Theorem \ref{thm:polynomial-mixing:N-particle}. In Section \ref{sec:newtonian-limit}, we perform a series of estimates on auxiliary systems that we employ to deduce the convergence of \eqref{eqn:rLE:N-particle:original} toward \eqref{eqn:LE:original}. We also supply the detailed proof of Theorem \ref{thm:newtonian-limit} in this section. The paper concludes with Appendices \ref{sec:appendix} and \ref{sec:well-posed}. In Appendix \ref{sec:appendix}, we provide two useful estimates on singular potentials that are exploited to establish the main results, whereas in Appendix \ref{sec:well-posed}, we supply the argument for the well-posedness of the system.

\section{Assumptions and main results} \label{sec:result}

Throughout, we let $(\Omega, \mathcal{F}, (\mathcal{F}_t)_{t\geq 0},  \P)$ be a filtered probability space satisfying the usual conditions \cite{karatzas2012brownian} and $(W_{i}(t))$, $i=1,\dots, N$, be i.i.d standard $d$-dimensional Brownian Motions on $(\Omega, \mathcal{F},\P)$ adapted to the filtration $(\mathcal{F}_t)_{t\geq 0}$.

Since the mass $m$ and the friction constant $\gamma$ in \eqref{eqn:rLE:N-particle:original} do not affect the analysis, we set $m=\gamma=1$ for simplicity. Also, for notational convenience, we denote $\varepsilon:=1/c^2$ and recast \eqref{eqn:rLE:N-particle:original} as follows:
\begin{align} \label{eqn:rLE:N-particle:epsilon}
\d\, q_i(t) &= \frac{p_i(t)}{\sqrt{1+\varepsilon |p_i(t)|^2}}\d t,\qquad i=1,\dots,N, \notag\\
\d\, p_i(t) & = -\frac{p_i(t)}{\sqrt{1+\varepsilon |p_i(t)|^2}} \d t -\grad U(q_i(t))\d t- \sum_{j\neq i} \grad G\big(q_i(t)-q_j(t)\big) \d t +\sqrt{2} \,\d W_i(t).
\end{align}

In Section \ref{sec:result:assumption}, we detail various assumptions on the nonlinearities $U$ and $G$ that we will employ throughout the analysis. The well-posedness of \eqref{eqn:rLE:N-particle:epsilon} (and equivalently of \eqref{eqn:rLE:N-particle:original}) is briefly presented in Proposition \ref{prop:rLE:N-particle:well-posed}. In Section \ref{sec:result:ergodicity}, we state the first main result through Theorem~\ref{thm:polynomial-mixing:N-particle} giving the unique ergodicity, as well as the polynomial mixing rate in total variation distances. Particularly, while the ergodicity result for the single-particle system applies to a wide class of singular potential $G$, the case of multiple particles requires further restrictions on $G$. In Section \ref{sec:result:newtonian-limit}, we establish our second main result through Theorem \ref{thm:newtonian-limit} concerning the validity of the Newtonian limit of \eqref{eqn:rLE:N-particle:original} on any finite time window. More specifically, depending on the number $N$ of particles, the convergence is obtained either in probability $(N\ge2)$ or in $L^p$ $(N=1)$.

\subsection{Main assumptions} \label{sec:result:assumption} We start with the potential $U$ and impose the following conditions \cite{duong2024asymptotic,ottobre2011asymptotic,
pavliotis2014stochastic}.

\subsection*{(U)}  \label{cond:U}
\textup{(i)} There exist positive constants $a_1$ and $\lambda\ge 1$ such that $U\in C^\infty(\rbb^d;[0,\infty))$ satisfies for  all $q\in\rbb^d$
\begin{align}
\frac{1}{a_1}|q|^{\lambda+1}-a_1 &\le  |U(q)|\le a_1(1+|q|^{\lambda+1}),\label{cond:U:U(x)=O(x^lambda+1)}  \\
\text{and}\quad |\grad \U(q)| & \le a_1(1+|q|^{\lambda}).\label{cond:U:U'(x)=O(x^lambda)} 
\end{align}

\textup{(ii)} Furthermore,
\begin{equation} \label{cond:U:x.U'(x)>-x^(lambda+1)}
\la \grad\U(q), q\ra \ge a_2|q|^{\lambda+1}-a_3,\quad q\in\rbb^d,
\end{equation}
for some positive constants $a_2$ and $a_3$.

We note that Assumption \nameref{cond:U} requires that $U$ behave like a polynomial of the form $|q|^{\lambda+1}$, which is quite popular and can be found in a lot of previous works for SDEs \cite{glatt2020generalized,mattingly2002ergodicity} as well as for SPDEs \cite{hairer2002exponentialII,hairer2002exponentialI}. Concerning the singular potential $G$, we will make the following condition \cite{bolley2018dynamics,duong2024asymptotic,
herzog2017ergodicity,lu2019geometric}.

\subsection*{(G1)} \label{cond:G1}
  \textup{(i)} $G\in C^\infty(\rbb^d\setminus\{0\};\rbb)$ satisfies that $\G(q)\to \infty$ as $|q|\to 0$, $G$ is even and that $\grad G$ is odd. Furthermore, there exists a positive constant $\beta_1\ge 1$ such that for all $q\in\rbb^d\setminus\{0\}$
\begin{align} 
|\G(q)|&\le a_1\Big(1+|q|+\frac{1}{|q|^{\beta_1}}\Big) ,\label{cond:G:G<1/|x|^beta}\\
\text{and} \quad|\grad \G(q)| &\le a_1\Big(1+\frac{1}{|q|^{\beta_1}}\Big),\label{cond:G:grad.G(x)<1/|x|^beta}
\end{align}
where $a_1$ is the constant as in Assumption~\nameref{cond:U}.

\textup{(ii)} There exist constants $\beta_2\in[0,\beta_1)$, $a_4>0$, $a_5\in\rbb$ and $a_6>0$ such that 
\begin{equation} \label{cond:G:|grad.G(x)+q/|x|^beta_1|<1/|x|^beta_2}
\Big|\grad G(q) +a_4\frac{q}{|q|^{\beta_1+1}}+a_5\frac{q}{|q|^{\beta_2+1}}\Big| \le a_6, \quad q\in\rbb^d\setminus\{0\}.
\end{equation}

\begin{remark} \label{rem:G1} 
(i) We note that condition $\beta_1\ge 1$ is to ensure that $G$ is a singular function and that collisions do not occur in finite time. Furthermore, \eqref{cond:G:|grad.G(x)+q/|x|^beta_1|<1/|x|^beta_2} essentially states that $\grad G$ can be expressed as
\begin{align*}
\grad G(q)= -c\frac{q}{|q|^{\beta_1+1}}+\text{lower order terms},
\end{align*}
hence the requirement $\beta_2<\beta_1$. 

(ii) As mentioned elsewhere in \cite{duong2024asymptotic}, a routine calculation shows that both the repulsive Coulomb functions
\begin{align*} 
G(q)=\begin{cases}
-\log|q|,& d=2,\\
\frac{1}{|q|^{d-2}},& d\ge 3,
\end{cases}
\end{align*}
 and the Lennard-Jones functions 
\begin{align*}
G(q) = \frac{c_0}{|q|^{12}}-\frac{c_1}{|x|^6},
\end{align*} 
satisfy \eqref{cond:G:|grad.G(x)+q/|x|^beta_1|<1/|x|^beta_2}. In particular, a well-known example for the case $\beta_1=1$ is the log function $G(q)=-\log|q|$ whereas the case $\beta_1>1$ includes the instance $G(q)=|q|^{-\beta_1+1}$.

 Also, we note that \eqref{cond:G:|grad.G(x)+q/|x|^beta_1|<1/|x|^beta_2} implies
\begin{align} \label{cond:G:|grad.G(x)+q/|x|^beta_1|<1/|x|^beta_2:duong2024}
\Big|\grad G(q) +a_4\frac{q}{|q|^{\beta_1+1}}\Big| \le \frac{|a_5|}{|q|^{\beta_2}}+a_6, \quad q\in\rbb^d\setminus\{0\},
\end{align}
which is the same as \cite[Condition (2.7)]{duong2024asymptotic}. In fact, while estimate \eqref{cond:G:|grad.G(x)+q/|x|^beta_1|<1/|x|^beta_2:duong2024} will be directly employed to establish the ergodicity results in Section \ref{sec:poly-mixing}, condition \eqref{cond:G:|grad.G(x)+q/|x|^beta_1|<1/|x|^beta_2} will be invoked to study the Newtonian limit in Section \ref{sec:newtonian-limit}.

(iii) Without loss of generality, we may assume that when $N\ge 2$
\begin{align} \label{cond:U+G>1:N>1}
\sum_{i=1}^N U(q_i)+\sum_{1\le i<j\le N}\close G(q_i-q_j)\ge 0,\quad \frac{1}{2}\sum_{i=1}^N U(q_i)+\sum_{1\le i<j\le N}\close G(q_i-q_j)\ge 0,
\end{align}
and when $N=1$
\begin{align} \label{cond:U+G>1:N=1}
U(q)+G(q)\ge 0.
\end{align}
Otherwise, we may replace $U(q)$ by $U(q)+C$ for some sufficiently large constant $C$, which does not affect \eqref{eqn:rLE:N-particle:epsilon}.

 (iv) Lastly, we remark that \cite[Condition (2.2)]{duong2024asymptotic} and \cite[Condition (2.6)]{duong2024asymptotic} respectively provide an upper bound on $\grad^2 U$ and $\grad^2 G$, which are needed for the study of the small mass limit in \cite{duong2024asymptotic}. Since we do not explore such a limit in the present article, we do not impose requirements on the second derivatives of the nonlinearities. 

\end{remark}

Although Assumptions \nameref{cond:U} and \nameref{cond:G1} are sufficient to establish the majority of the main results in this paper, we will have to impose further restrictions on the behavior of $G$ near the origin when dealing with ergodicity in the setting of multiple particles. More precisely, we make the following condition.

\subsection*{(G2)} \label{cond:G2} Let $\beta_1$ and $\beta_2$ be the constants as in Assumption \nameref{cond:G1}. We assume that
\begin{align} \label{cond:G2:beta_1in(1,2]}
\beta_1\in(1,2],\quad \text{and}\quad \beta_2\in[0,\beta_1-1),
\end{align}
and that for all $q\in \rbb^d\setminus\{0\}$,
\begin{align}
    |\G(q)|&\le a_1\Big(1+\frac{1}{|q|^{\beta_1-1}}\Big) , \label{cond:G2:G<1/|x|^(beta-1)} 
\end{align}
where $a_1$ is the constant from Assumption \ref{cond:U}.

\begin{remark} \label{rem:G2} As it turns out, the ergodicity of \eqref{eqn:rLE:N-particle:epsilon} for the case $N\ge 2$ is more delicate due to the singular interaction between the particles together with the weak dissipation from the relativistic force. To circumvent the issue, we restrict to those functions $G$ satisfying Assumption \nameref{cond:G2}, which requires that $G$ be of Riesz type and grow at most like $1/|q|$ near the origin. In particular, as a consequence of \eqref{cond:G:|grad.G(x)+q/|x|^beta_1|<1/|x|^beta_2}, condition \eqref{cond:G2:beta_1in(1,2]} excludes the instances of log functions as well as the Lennard-Jones function.
\end{remark}

Having introduced sufficient conditions on the potentials,  we turn to the issue of well-posedness for \eqref{eqn:rLE:N-particle:epsilon}. The domain where the process $\qb^\varepsilon(t)$ evolves on is denoted by $\D$ and is defined as  \cite{bolley2018dynamics,herzog2017ergodicity,
lu2019geometric}
\begin{align} \label{form:D}
\D= \begin{cases}
    \{\qb=(q_1,\dots,q_N) \in (\rbb^d)^N: q_i\neq q_j \text{ if } i\neq j\},& d\ge 2,\\
    \{\qb=(q_1,\dots,q_N) \in \rbb^N: q_1<q_2<\dots<q_N\},& d=1.
\end{cases}
\end{align}
We note that in dimension one, the set $\{ (q_1,\dots,q_N) \in (\rbb)^N:q_i\neq q_j \text{ if } i\neq j\}\}$ is not path connected. This requires the system be restricted to one of the connected components, hence the choice of $\D$ when $d=1$ for notational convenience. Then, we define the phase space for the solution of \eqref{eqn:rLE:N-particle:epsilon} as follow:
\begin{align} \label{form:X}
\X=\D\times (\rbb^d)^N .
\end{align}

The first result of this paper is the following well-posedness result ensuring the existence and uniqueness of strong solutions to system \eqref{eqn:rLE:N-particle:epsilon}.

\begin{proposition} \label{prop:rLE:N-particle:well-posed}  Under Assumption \nameref{cond:U} and Assumption \nameref{cond:G1}, for every initial condition $X_0=(\qb_0,\pb_0)\in\X$, system \eqref{eqn:rLE:N-particle:epsilon} admits a unique strong solution $X^\varepsilon(t;X_0)=\big(\qb^\varepsilon(t),\pb^\varepsilon(t)\big)\in\X$.
\end{proposition}

The well-posedness of \eqref{eqn:rLE:N-particle:epsilon} is guaranteed by using the Hamiltonian structures in \eqref{form:H} and \eqref{form:H_N} ensuring the total energy is controlled over time. Alternatively, Proposition \ref{prop:rLE:N-particle:well-posed} can be established using the Lyapunov functions constructed in Lemma \ref{lem:Lyapunov:single-particle} and Lemma \ref{lem:Lyapunov:N-particle} below in Section \ref{sec:poly-mixing}. The explicit argument is relatively standard following \cite{veretennikov2024lyapunov} and can be found in Appendix \ref{sec:well-posed} where we supply the proof of Proposition \ref{prop:rLE:N-particle:well-posed}. See also \cite{glatt2020generalized,mattingly2002ergodicity,khasminskii2011stochastic} for a further detailed discussion on the well-posedness of SDEs.

As a consequence of the well-posedness, we can thus introduce the Markov transition probabilities of the solution $X^\varepsilon(t)$ by
\begin{align*}
P_t^\varepsilon(X_0,A):=\P(X^\varepsilon(t;X_0)\in A),
\end{align*}
which are well-defined for $t\ge 0$, initial condition $X_0\in\X$ and Borel sets $A\subset\X$. Letting $\B_b(\X)$ denote the set of bounded Borel measurable functions $f:\X \rightarrow \rbb$, the associated Markov semigroup $P_t^\varepsilon:\B_b(\X)\to\B_b(\X)$ is defined and denoted by
\begin{align*}
P_t^\varepsilon f(X_0)=\E[f(X^\varepsilon(t;X_0))], \,\, f\in \B_b(\X).
\end{align*}

\subsection{Polynomial mixing} \label{sec:result:ergodicity}
We now turn to the first main topic of the paper concerning the large-time asymptotics of equation~\eqref{eqn:rLE:N-particle:epsilon}. Letting $\Pcal r(\X)$ denote the space of probability measures on Borel subsets of $\X$, recall that $\mu\in \Pcal r (\X)$ is called {\bf invariant} for the semigroup $P_t^\varepsilon$ if for every $f\in \B_b(\X)$
\begin{align*}
\int_{\X} f(X) (P_t^\varepsilon)^*\mu(\d X)=\int_{\X} f(X)\mu(\d X),
\end{align*}
where $(P_t^\varepsilon)^*\mu$ denotes the measure obtained by the action of $P_t^\varepsilon$ on $\mu$, i.e., \cite{hairer2011yet}
\begin{align*}
(P_t^\varepsilon)^*\mu(A) = \int_{\X} P_t^\varepsilon(X,A)\mu (\d X).
\end{align*}
Next, let $\L_N$ denote the generator associated with \eqref{eqn:rLE:N-particle:epsilon}. One defines $\L_N$ for any $\f\in C^2(\X;\rbb)$ by
\begin{align} \label{form:L_N}
\L_N\varphi &= \sum_{i=1}^N \frac{p_i}{\sqrt{1+\varepsilon|p_i|^2}}\cdot\grad_{q_i}\varphi +\sum_{i=1}^N \Big( - \frac{p_i}{\sqrt{1+\varepsilon|p_i|^2}} -\grad U(q_i)\Big)\cdot\grad_{p_i}\varphi  \notag \\
&\qquad -\sum_{i\neq j}\grad G(q_i-q_j)\cdot\big[ \grad_{p_i}\varphi -\grad_{p_j}\varphi \big] +\sum_{i=1}^N \triangle_{p
_i}\varphi .
\end{align}
As mentioned in the introduction, it can be shown that the invariant probability measure for $P_t^\varepsilon$ is given by the following relativistic Maxwellian distribution 
\begin{align} \label{form:pi_N^epsilon}
\pi_N^\varepsilon(\d\qb,\d\pb)=\frac{1}{Z_N^\varepsilon}\exp\Big\{-\Big(\sum_{i=1}^N \frac{1}{\varepsilon}\sqrt{1+\varepsilon|p_i|^2} +\sum_{i=1}^N U(q_i)+\close\sum_{1\le i<j\le N}\close G(q_i-q_j)\Big) \Big\}\d \qb\d \pb,
\end{align}
where $Z_N^\varepsilon$ is the normalized constant. Indeed, it suffices to establish that $\pi_N^\varepsilon(X)$ is a solution of the stationary Fokker-Planck equation
\begin{align} \label{eqn:L^*.pi_N=0}
(\L_N)^*\pi_N^\varepsilon(X)=0,
\end{align}
where $(\L_N)^*$ denotes the dual of $\L_N$, i.e.,
\begin{align*}
\int_{\X} \L_Nf_1(X)\cdot f_2(X) \d X = \int_{\X} f_1(X)\cdot (\L_N)^*f_2(X)\d X,
\end{align*}
for any $f_1,f_2\in C^2_c(\X;\rbb)$. In the literature, such an explicit calculation was carried out in the settings of the Langevin dynamics and related systems \cite{glatt2020generalized,mattingly2002ergodicity,pavliotis2014stochastic}. The main distinction of those equations from \eqref{eqn:rLE:N-particle:epsilon} is the appearance of the viscous drag in lieu of the relativistic forcing. Nevertheless, the proof of identity \eqref{eqn:L^*.pi_N=0} can be simply adapted to that of \cite[Proposition 8.2]{pavliotis2014stochastic} tailored to our setting with the presence of the singularities and the relativistic term.

Concerning the unique ergodicity of $\pi_N^\varepsilon$, following the framework of \cite{hairer2009hot,hairer2011yet,mattingly2002ergodicity}, we will measure the mixing rate with respect to the total variation distance, denoted by $\W_{\TV}$. Recall that $\W_{\TV}$ in $\Pcal r (\X)$ is defined as
\begin{align*}
\W_{\TV}(\mu_1,\mu_2) = \inf \E \big[\mathbf{1}\{X_1\neq X_2\}\big],
\end{align*}
where the infimum is taken over all bivariate random variables $(X_1,X_2)$ such that $X_1\sim \mu_1$ and $X_2\sim\mu_2$.

We now state the first main result of the paper, establishing the unique ergodicity of $\pi_N^\varepsilon$ defined in \eqref{form:pi_N^epsilon} as well as the polynomial convergent rate toward $\pi_N^\varepsilon$.

\begin{theorem} \label{thm:polynomial-mixing:N-particle}
For all $N\ge 1$, suppose that Assumption \textup{\nameref{cond:U}} and Assumption \textup{\nameref{cond:G1}} hold. Suppose further that when $N\ge 2$, Assumption \textup{\nameref{cond:G2}} holds. Then, for every $\varepsilon\in(0,1)$ sufficiently small and for all $r\ge1$, there exists a function $V\in C^2(\X;[1,\infty))$ such that the following holds 
\begin{align} \label{ineq:polynomial-ergodicity}
\W_{\TV}\big(P_t^\varepsilon(X_0,\cdot), \pi^\varepsilon_N\big)\le \frac{C}{(1+t)^r}V(X_0),\quad t\ge 0,\, X_0\in\X,
\end{align}
for some positive constant $C$ independent of $t$ and $X_0$.
\end{theorem}

As mentioned in Section \ref{sec:intro:proof}, the proof of Theorem \ref{thm:polynomial-mixing:N-particle} relies on three main ingredients. The first of which is the construction of suitable Lyapunov functions, cf. Definition \ref{def:Lyapunov}, allowing for establishing moment bounds on the returning time to the center of the phase space. The second ingredient is the Hormander's condition, cf. Definition \ref{def:Hormander}, verified through H\"ormander's celebrated ``sums of squares", see Theorem 1.1 in \cite{hormander1967hypoelliptic}. In particular, this property ensures that transition probabilities $P_t^\varepsilon$ have smooth densities $p_t^\varepsilon(X,Y)$ with respect to Lebesgue measure. The last property needed to complete the proof of \eqref{ineq:polynomial-ergodicity} is the solvability of the associated control problem, cf. Definition \ref{def:control-problem}. In turn, this implies the so-called irreducibility property, guaranteeing the uniform possibility of returning provided the dynamics starts from a compact set. While the latter two properties are relatively not difficult to check, as they have been done in literature for the related Langevin systems, the former on the derivation of the Lyapunov function turns out to be the most challenging issue, owing to the lack of a strong dissipative mechanism from the $\pb$ process. All of this will be discussed in details in Section \ref{sec:poly-mixing} where we supply the proof of Theorem \ref{thm:polynomial-mixing:N-particle}.

\begin{remark} \label{rem:polynomial-mixing} 1. We note that the constant $C$ in Theorem \ref{thm:polynomial-mixing:N-particle} may depend on the parameter $\varepsilon$ and $N$, which heavily affect the construction of $V$ below in Section \ref{sec:poly-mixing}. In particular, $C$ may get arbitrarily large when either $\varepsilon\to 0$ or $N\to\infty$. The issue of uniform polynomial mixing with respect to either small $\varepsilon$ or large $N$ therefore remains an open problem.

2. In the literature of the classical Langevin dynamics and related systems, the more popular mixing results are usually exponential convergent rates \cite{arnold2024trend,calogero2012exponential,  duong2024asymptotic, herzog2017ergodicity, lu2019geometric, mattingly2002ergodicity}. This relies on finding a Lyapunov function $V$ satisfying the inequality
\begin{align*}
\L V\le -cV+C,
\end{align*}
where $\L$ is the infinitesimal generator associated with the dynamics. In our settings, particularly when $N\ge 2$, we have to navigate the difficulty arising from the relativistic forcing and the singular potentials to the extend that we are only able to construct a weaker Lyapunov function of the form
\begin{align} \label{ineq:LV<-V^alpha}
\L V\le -c V^\alpha+C,
\end{align}
for some positive power $\alpha\in(0,1)$. See Definition \ref{def:Lyapunov}. As a consequence, this would imply the polynomial mixing rate \eqref{ineq:polynomial-ergodicity}.
\end{remark}

\subsection{Newtonian limit} \label{sec:result:newtonian-limit}

With regard to the Newtonian regime as $c\to\infty$ in the original system \eqref{eqn:rLE:N-particle:original}, observe that this translates to the convergence of \eqref{eqn:rLE:N-particle:epsilon} toward the classical Langevin dynamics \eqref{eqn:LE:original} as $\varepsilon\to 0$. Before stating the main result of this subsection, we note that the well-posedness of \eqref{eqn:LE:original} in the presence of singularities is not an issue as this can be proven making use of its own Hamiltonian structure, see \cite{conrad2010construction,grothaus2015hypocoercivity, herzog2017ergodicity}. More interestingly, as mentioned in the introduction, thanks to the strong dissipation, \eqref{eqn:LE:original} is known to possess geometric ergodicity, i.e., the relaxation rate toward equilibrium is exponentially fast \cite{herzog2017ergodicity,lu2019geometric}.

We now state our second main result of the paper through Theorem \ref{thm:newtonian-limit}, establishing the validity of the approximation of \eqref{eqn:rLE:N-particle:epsilon} by \eqref{eqn:LE:original}.

\begin{theorem} \label{thm:newtonian-limit}
Suppose that Assumptions \nameref{cond:U} and \nameref{cond:G1} hold. For every initial condition $(\qb_0,\pb_0)\in \X$, let $X^\varepsilon(t)=\big(\qb^\varepsilon(t),\pb^\varepsilon(t)\big)$ and $X(t)=\big(\qb(t),\pb(t)\big)$, respectively, be the solutions of \eqref{eqn:rLE:N-particle:epsilon} and \eqref{eqn:LE:original} (with $m=\gamma=1$). Then, the followings hold:

1. (Multi-particle system) For all $N\ge 2$, $T>0$, and $\xi>0$,
\begin{align} \label{lim:newtonian-limit:N-particle:prob}
\P\Big( \sup_{t\in[0,T]}\big[ |\qb^\varepsilon(t)-\qb(t)|+|\pb^\varepsilon(t)-\pb(t)|\big]>\xi\Big) \rightarrow 0,\quad \varepsilon\rightarrow 0.
\end{align}

2. (Single-particle system) When $N=1$, for all $T>0$ and $n\ge 1$,
\begin{align} \label{lim:newtonian-limit:single-particle:L^p}
\E \sup_{t\in[0,T]}\big[ |q^\varepsilon(t)-q(t)|^n+|p^\varepsilon(t)-p(t)|^n\big]\rightarrow 0,\quad \varepsilon\rightarrow 0.
\end{align}
\end{theorem}

Owing to the presence of the singular potentials, we follow closely the framework of \cite[Section 4]{duong2024asymptotic} dealing with the same issue for the so-called small mass limit. Similar to the mixing results of Theorem \ref{thm:polynomial-mixing:N-particle}, the complication caused by the relativistic forcing jointly with the nonlinearities induces further difficulty when considering the case $N\ge2$. As a consequence, we have to treat this situation with a slightly different technique compared to the one employed for the single-particle scenario ($N=1$). More specifically, on the one hand, when $N\ge 2$, the argument of \eqref{lim:newtonian-limit:N-particle:prob} relies on moment bounds on the Langevin system \eqref{eqn:LE:original}. Adapting to the idea from \cite{herzog2016small}, this results in a convergence in probability on any finite-time window. On the other hand, when $N=1$, we are able to obtain uniform energy bounds on \eqref{eqn:rLE:N-particle:epsilon} that are independent of $\varepsilon$. In turn, they allow for upgrading the preceding probability argument so as to establish the convergence in $L^p$. We note that those estimates are relatively delicate making use of the fact that $N=1$ and otherwise remain open for the multi-particle case. All of this will be discussed in details in Section \ref{sec:newtonian-limit} where we provide the proof of Theorem \ref{thm:newtonian-limit}.

\section{Polynomial mixing of \eqref{eqn:rLE:N-particle:epsilon}} \label{sec:poly-mixing}

In this section, we provide the detailed proof of Theorem \ref{thm:polynomial-mixing:N-particle} giving the polynomial mixing rate of \eqref{eqn:rLE:N-particle:epsilon} toward the unique invariant probability measure $\pi_N^\varepsilon$ defined in \eqref{form:pi_N^epsilon}. As mentioned in Section \ref{sec:result:ergodicity}, the main ingredient is the construction of suitable Lyapunov functions, which will be presented in Sections \ref{sec:poly-mixing:Lyapunov:single-particle} and \ref{sec:poly-mixing:Lyapunov:N-particle} below. Then, together with verifying the Hormander's condition and solving the associated control problem, they allow us to extract the convergent rate \eqref{ineq:polynomial-ergodicity}, thereby concluding Theorem \ref{thm:polynomial-mixing:N-particle} in Section \ref{sec:poly-mixing:proof-of-thm}. For the reader's convenience, we first recall the definitions of these notions below. 

\begin{definition} \label{def:Lyapunov}
A function $V\in C^2(\X;[1,\infty))$ is called a \emph{Lyapunov} function for \eqref{eqn:rLE:N-particle:epsilon} if the followings hold:

(i) $V(\qb,\pb)\to \infty$ whenever $|(\qb,\pb)|+\sum_{1\le i<j\le N}|q_i-q_j|^{-1}\to \infty$ in $\X$; and

(ii) for all $(\qb,\pb)\in\X$, 
\begin{align} \label{ineq:Lyapunov}
\mathcal{L}_N V(\qb,\pb)\le -c  V(\qb,\pb)^\alpha+D, 
\end{align}
for some constants $c>0$, $D\ge 0$ and $\alpha\in(0,1)$ independent of $(\qb,\pb)$. In the above, $\L_N$ is the generator defined in \eqref{form:L_N}.

\end{definition}

To state the Hormander's condition, we note that $\L_N$ from \eqref{form:L_N} can be recast as
\begin{align*}
\L_N =  X_0 +\sum_{i=1}^N (X_i)^2,
\end{align*}
where $\{X_i\}_{i=0}^N$ denotes the family of vector fields given by
\begin{align} \label{form:X_i}
X_0& = \sum_{i=1}^N\frac{p_i}{\sqrt{1+\varepsilon|p_i|^2}}\partial_{q_i}+ \sum_{i=1}^N\Big( - \frac{p_i}{\sqrt{1+\varepsilon|p_i|^2}} -\grad U(q_i)  -\sum_{j\neq i}\grad G(q_i-q_j)\Big)\partial_{p_i}, \notag \\
X_i &= \partial_{p_i}.
\end{align}
Following \cite[Theorem 1.1]{hormander1967hypoelliptic} and \cite[Condition (159)]{bellet2006ergodic}, we have the following definition of Hormander's condition: 

\begin{definition} \label{def:Hormander}
The family of vector fields $\{X_i\}_{i=0}^N$ defined in \eqref{form:X_i} is said to satisfy the Hormander's condition if the Lie algebra generated by the family
\begin{align*}
\{X^i\}_{i=1}^N,\quad\{[X^i,X^j]\}_{i,j=0}^N,\quad \{[[X^i,X^j],X_k]\}_{i,j,k=0}^N,\quad\dots,
\end{align*}
has maximal rank at every point $X\in\X$.
\end{definition}

The last condition needed to establish Theorem \ref{thm:polynomial-mixing:N-particle} concerns the solvability of the control problem associated with \eqref{eqn:rLE:N-particle:epsilon}. Namely, consider the following system of ODEs:
\begin{align} \label{eqn:rLE:N-particle:epsilon:control-problem}
\d\, q_i(t) &= \frac{p_i(t) }{\sqrt{1+\varepsilon |p_i(t) |^2}}\d t,\qquad i=1,\dots,N, \notag\\
\d\, p_i(t) & = -\frac{p_i(t) }{\sqrt{1+\varepsilon |p_i(t) |^2}} \d t -\grad U(q_i(t))\d t- \sum_{j\neq i} \grad G\big(q_i(t)-q_j(t)\big) \d t +\sqrt{2} \,\d U_i(t),
\end{align}
where $U_i(\cdot)\in C^1(\rbb^d;\rbb)$ is a control path. We then have the following condition requiring that the element $X_1\in \X=\D\times
(\rbb^d)^N$ be reachable through \eqref{eqn:rLE:N-particle:epsilon:control-problem} from $X_0\in \X$.

\begin{definition} \label{def:control-problem}
Given any $X_0,X_1\in\X$, the control problem \eqref{eqn:rLE:N-particle:epsilon:control-problem} is said to solvable if there exist a time $T>0$ and paths $\{U_i\}_{i=1}^N$ such that the solution $Z(t)$ of \eqref{eqn:rLE:N-particle:epsilon:control-problem} satisfies
\begin{align*}
Z(0)= X_0,\quad \textup{and}\quad Z(T)=X_1.
\end{align*}
\end{definition}

As we will see later in Section \ref{sec:poly-mixing:proof-of-thm}, verifying the latter two properties are relatively not difficult as they follow standard arguments in literature \cite{hairer2009hot,herzog2017ergodicity,mattingly2002ergodicity,
ottobre2011asymptotic,pavliotis2014stochastic,bellet2006ergodic}. We now turn our attention to the Lyapunov property stated in Definition \ref{def:Lyapunov}, which is the main condition that the proof of Theorem \ref{thm:polynomial-mixing:N-particle} relies on. We start the procedure in Section \ref{sec:poly-mixing:Lyapunov:single-particle} and construct a Lyapunov function for the case of single-particle $(N=1)$ since it is simpler than that for the case of multi-particle $(N\ge 2)$.

\subsection{Lyapunov function for single-particle system} \label{sec:poly-mixing:Lyapunov:single-particle}

For the reader's convenience, when $N=1$, equation \eqref{eqn:rLE:N-particle:epsilon} is reduced to the following system:
\begin{align} \label{eqn:rLE:single-particle:epsilon}
\d\, q(t) &= \frac{p(t) }{\sqrt{1+\varepsilon |p(t) |^2}}\d t, \notag\\
\d\, p(t) & = -\frac{p(t) }{\sqrt{1+\varepsilon |p(t) |^2}} \d t -\grad U(q(t))\d t- \grad G\big(q(t)\big) \d t +\sqrt{2} \,\d W(t).
\end{align}
The generator $\L_1$ associated with \eqref{eqn:rLE:single-particle:epsilon} is given by
\begin{align} \label{form:L_1}
\L_1 = \frac{p}{\sqrt{1+\varepsilon |p|^2}}\cdot\grad_q +\Big[-\frac{p}{\sqrt{1+\varepsilon |p|^2}}-\grad U(q)-\grad G(q)\Big]\cdot\grad_p+ \triangle_p.
\end{align}
Before diving into the explicit construction of a Lyapunov function, we briefly discuss the heuristic argument to derive such a function. To this end, let us introduce the Hamiltonian $H$ defined as
\begin{align} \label{form:H}
H(q,p) = \varepsilon U(q)+\varepsilon G(q)+\sqrt{1+\varepsilon |p|^2}.
\end{align}
A routine calculation gives
\begin{align} \label{eqn:L_1.H}
\L_1 H = -\frac{\varepsilon|p|^2}{1+\varepsilon|p|^2}+\frac{d\varepsilon+(d-1)\varepsilon^2|p|^2}{(1+\varepsilon|p|^2)^{3/2}}.
\end{align}
Observe that the ansatz choice of $H$ does not satisfy the Lyapunov bound \eqref{ineq:LV<-V^alpha} so as to deduce a mixing rate. In order to achieve the desired Lyapunov effect, it is necessary to modify $H$ appropriately. To this end, on the one hand, when $|q|\to \infty$, we follow \cite{mattingly2002ergodicity} and employ the usual trick by combining $H$ with the cross term $\la q,p\ra$. On the other hand, when $|q|\to 0$, we adopt the framework of \cite{duong2024asymptotic, lu2019geometric} and perturb $H$ by the term $\la q,p\ra/|q|$. In other words, the next ansatz choice has the form
\begin{align} \label{form:V:Lu2019}
V= H+\varepsilon_1\la p,q\ra+\frac{\la q,p\ra}{|q|}.
\end{align}
However, one may see that this choice is not guaranteed to stay strictly positive, since the Hamiltonian comprises of $\sqrt{1+\varepsilon|p|^2}$ and does not dominate $\la q,p\ra$. Nor does  $\L_1 H$ itself produce any dissipation property as $|p|\to\infty$. While this is the main issue that we face in our settings, we note that it did not induce any challenge for the Langevin system \eqref{eqn:LE:original} since their Hamiltonian structure has the order of $|p|^2$ \cite{bolley2018dynamics, mattingly2002ergodicity,
herzog2017ergodicity,lu2019geometric}. In fact, the choice of \eqref{form:V:Lu2019} was successfully employed in the work of \cite{duong2024asymptotic,lu2019geometric} to establish an exponential convergent rate.

To circumvent the difficulty, we will have to increase the strength of the Hamiltonian, namely we will consider $H^2$ instead of $H$. As it turns out, not only is $H^2$ able to subsume the cross term $\la q,p\ra$ and $\la q,p\ra/|q|$, altogether, they also produce the desired Lyapunov effect. In summary, letting $\varepsilon_1,\kappa_1>0$ be given and be chosen later, a candidate function for the single-particle case is given below.
\begin{align} \label{form:V_1}
V_1(q,p) = H(q,p)^2+\varepsilon_1 \la q,p\ra - \frac{\la q,p\ra}{|q|}+\kappa_1.
\end{align}  
By tuning $\varepsilon_1$ small and $\kappa_1$ large enough, we will see that $V_1$ is a Lyapunov function for \eqref{eqn:rLE:single-particle:epsilon}. This is rigorously stated through Lemma \ref{lem:Lyapunov:single-particle} below.

\begin{lemma} \label{lem:Lyapunov:single-particle}
Suppose that Assumptions \textup{\nameref{cond:U}} and \textup{\nameref{cond:G1}} hold. For all $\varepsilon\in(0,1)$ sufficiently small, there exist positive constants $\varepsilon_1$ and $\kappa_1$ such that the function $V_1$ defined in \eqref{form:V_1} satisfies
\begin{align} \label{ineq:L_1.V_1^n}
\L_1\big[ V_1(q,p)^n\big] \le -c V_1(q,p)^{n-\frac{1}{2}}+C~ \text{for all}~ n\geq 1,
\end{align}
for some positive constants $c=c(n,\varepsilon)$and $C=C(n,\varepsilon)$ independent of $q$ and $p$. In the above, $\L_1$ is the generator given by \eqref{form:L_1}.
\end{lemma}
\begin{proof}
Recalling the function $H$ defined in \eqref{form:H}, from \eqref{eqn:L_1.H}, we have
\begin{align} \label{ineq:L_1.H^2}
\L_1 H^2 &= 2H\Big(-\frac{\varepsilon|p|^2}{1+\varepsilon|p|^2}+\frac{d\varepsilon+(d-1)\varepsilon^2|p|^2}{(1+\varepsilon|p|^2)^{3/2}}\Big)+2\frac{\varepsilon^2|p|^2}{1+\varepsilon|p|^2}   \notag \\
& \le 2H\Big(-\frac{\varepsilon|p|^2}{1+\varepsilon|p|^2}+\frac{d\varepsilon+(d-1)\varepsilon^2|p|^2}{(1+\varepsilon|p|^2)^{3/2}}\Big)+2\varepsilon  .
\end{align}

Next, with regard to the cross term $\varepsilon_1\la q,p\ra$ on the right-hand side of expression \eqref{form:V_1}, it holds that
\begin{align*}
\L_1 \la q,p\ra & = \frac{|p|^2}{\sqrt{1+\varepsilon|p|^2}}  - \frac{\la q,p\ra}{\sqrt{1+\varepsilon |p|^2}}-\la q,\grad U(q)\ra - \la q,\grad G(q)\ra.
\end{align*}
To estimate the potentials, we recall from condition \eqref{cond:U:x.U'(x)>-x^(lambda+1)} that
\begin{align*}
-\la q,\grad U(q)\ra \le -a_2|q|^{\lambda+1}+a_3,
\end{align*}
whereas condition \eqref{cond:G:grad.G(x)<1/|x|^beta} implies
\begin{align*}
\la q,\grad G(q)\ra \le a_1\Big( |q|+\frac{1}{|q|^{\beta_1-1}} \Big).
\end{align*}
It follows that
\begin{align} \label{ineq:L_1.<q,p>}
\L_1 \la q,p\ra & \le  \frac{|p|+|q|}{\sqrt{\varepsilon}}  - \frac{1}{2}a_2|q|^{\lambda+1}+\frac{a_1}{|q|^{\beta_1-1}}+C .
\end{align}
In the above, $C$ is a positive constant independent of $q,p$ and $\varepsilon$.

Turning to the singular term $\la q,p\ra/|q|$, a routine calculation gives
\begin{align}
\L_1\bigg(\!-\frac{\la q,p\ra}{|q|}\bigg) 
& = -\frac{|p|^2}{|q|\sqrt{1+\varepsilon|p|^2}}+\frac{|\la q,p\ra|^2}{|q|^3\sqrt{1+\varepsilon|p|^2}}+ \frac{\la q,p\ra}{|q|\sqrt{1+\varepsilon|p|^2}} \notag \\
&\qquad+ \frac{\la\grad U(q), q\ra}{|q|}+\frac{\la \grad \G(q),q\ra}{|q|}.\label{eqn:L_1<q,p>/|q|}
\end{align}
It is clear that
\begin{align*}
-\frac{|p|^2}{|q|\sqrt{1+\varepsilon|p|^2}}+\frac{|\la q,p\ra|^2}{|q|^3\sqrt{1+\varepsilon|p|^2}}\le 0,
\end{align*}
which is negligible. Also, using Cauchy-Schwarz inequality yields
\begin{align*}
\frac{\la q,p\ra}{|q|\sqrt{1+\varepsilon|p|^2}} \le \frac{1}{\sqrt{\varepsilon}}.
\end{align*}
With regard to the potentials, from condition~\eqref{cond:U:U'(x)=O(x^lambda)}, we readily have
\begin{align*}
 \frac{\la\grad U(q), q\ra}{|q|} \le |\grad \U(q)|\le a_1(1+|q|^\lambda).
\end{align*}
 Also, recalling \eqref{cond:G:|grad.G(x)+q/|x|^beta_1|<1/|x|^beta_2},  
\begin{align} \label{cond:G:<grad.G(q),q>/|q|}
\frac{\la \grad \G(q),q\ra}{|q|} = -\frac{a_4}{|q|^{\beta_1}}+ \frac{\la \grad \G(q)+a_4\frac{q}{|q|^{\beta_1+1}},q\ra}{|q|}& \le  -\frac{a_4}{|q|^{\beta_1}}+\frac{|a_5|}{|q|^{\beta_2}}+a_6\nt\\
&\le -\frac{a_4}{2|q|^{\beta_1}}+C.
\end{align}
In the last estimate above, we subsumed $|q|^{-\beta_2}$ into $-|q|^{-\beta_1}$ thanks to the fact that $\beta_2\in[0,\beta_1)$ by virtue of the condition \eqref{cond:G:|grad.G(x)+q/|x|^beta_1|<1/|x|^beta_2}. Altogether, we get
\begin{align} \label{ineq:L_1.<q,p>/|q|}
\L_1\bigg(\!-\frac{\la q,p\ra}{|q|}\bigg) 
&\le a_1 |q|^\lambda-\frac{a_4}{2|q|^{\beta_1}}+C.
\end{align}

Now, letting $\varepsilon_1>0$ be given and be chosen later, we collect \eqref{ineq:L_1.H^2}, \eqref{ineq:L_1.<q,p>} together with \eqref{ineq:L_1.<q,p>/|q|} to obtain the bound
\begin{align*}
& \L_1 \bigg( H^2+\varepsilon_1\la q,p\ra - \frac{\la q,p\ra}{|q|} \bigg)\\
& \le 2H\Big(-\frac{\varepsilon|p|^2}{1+\varepsilon|p|^2}+\frac{d\varepsilon+(d-1)\varepsilon^2|p|^2}{(1+\varepsilon|p|^2)^{3/2}}\Big) + \varepsilon_1\Big( \frac{|p|+|q|}{\sqrt{\varepsilon}}   - \frac{1}{2}a_2|q|^{\lambda+1}+\frac{a_1}{|q|^{\beta_1-1}}
\Big) \\
&\qquad+ a_1 |q|^\lambda-\frac{a_4}{2|q|^{\beta_1}}+C. 
\end{align*}
Since (recalling \eqref{cond:U+G>1:N=1})
\begin{align*}
H\ge \sqrt{1+\varepsilon|p|^2},\quad\text{and}\quad \frac{d\varepsilon+(d-1)\varepsilon^2|p|^2}{(1+\varepsilon|p|^2)^{3/2}}\le (2d-1)\varepsilon,
\end{align*}
we infer that
\begin{align*}
2H\Big(-\frac{\varepsilon|p|^2}{1+\varepsilon|p|^2}+\frac{d\varepsilon+(d-1)\varepsilon^2|p|^2}{(1+\varepsilon|p|^2)^{3/2}}\Big)& \le -\frac{2\varepsilon|p|^2}{\sqrt{1+\varepsilon|p|^2}}+2(2d-1)\varepsilon H\\
\le -2\sqrt{1+\varepsilon |p|^2}+2(2d-1)\varepsilon H +2.
\end{align*}
For simplicity, picking $\varepsilon_1=\varepsilon$ implies
\begin{align*}
& \varepsilon_1\Big( \frac{|p|+|q|}{\sqrt{\varepsilon}}   - \frac{1}{2}a_2|q|^{\lambda+1}+\frac{a_1}{|q|^{\beta_1-1}}
\Big) + a_1 |q|^\lambda-\frac{a_4}{2|q|^{\beta_1}}\\
&\qquad\le \sqrt{\varepsilon}|p|-\frac{1}{4} a_2\varepsilon |q|^{\lambda+1}-\frac{a_4}{4|q|^{\beta_1}}+C,
\end{align*}
where $C$ is now possibly very large as $\varepsilon$ is close to zero. As a consequence, we obtain the bound
\begin{align*}
\L_1 \bigg( H^2+\varepsilon\la q,p\ra - \frac{\la q,p\ra}{|q|} \bigg)& \le  -\sqrt{1+\varepsilon |p|^2}+2(2d-1)\varepsilon H-\frac{1}{4} a_2\varepsilon |q|^{\lambda+1}-\frac{a_4}{4|q|^{\beta_1}}+C . 
\end{align*}
In view of conditions \eqref{cond:U:U(x)=O(x^lambda+1)} and \eqref{cond:G:G<1/|x|^beta}, for $\varepsilon\in(0,1)$ small enough, we readily have
\begin{align*}
-\sqrt{1+\varepsilon |p|^2}+2(2d-1)\varepsilon H-\frac{1}{4} a_2\varepsilon |q|^{\lambda+1}-\frac{a_4}{4|q|^{\beta_1}} \le - cH+C,
\end{align*}
where $c>0$ is independent of $\varepsilon$. By shrinking $\varepsilon$ further to zero, we may subsume $2\varepsilon H$ into $-cH$ to produce the bound
\begin{align} \label{ineq:L_1(H^2+<q,p>-<q,p>/q)}
\L_1 \bigg( H^2+\varepsilon\la q,p\ra - \frac{\la q,p\ra}{|q|} \bigg) \le -cH+C.
\end{align}

Turning back to $V_1$ defined in \eqref{form:V_1}, by picking $\kappa_1$ sufficiently large and $\varepsilon$ sufficiently small, $V_1$ is guaranteed to be bounded from below by 1 and that the following holds 
\begin{align*}
cH^2-C \le V_1\le CH^2+C.
\end{align*}
This together with \eqref{ineq:L_1(H^2+<q,p>-<q,p>/q)} implies
\begin{align} \label{ineq:L_1V_1<-csqrt(V_1)+C}
\L_1 V_1\le -c \sqrt{V_1}+C,
\end{align}
which completes the base case $n=1$. 

Now for all $n\ge 2$, we apply It\^o's formula to $V_1^n$ to obtain the identity
\begin{align*}
\L_1(V_1^n) & = nV_1^{n-1}\L_1V_1+n(n-1)V_1^{n-2}\Big| 2H\frac{\varepsilon p}{\sqrt{1+\varepsilon |p|^2}}+\varepsilon q+\frac{q}{|q|}\Big|^2.
\end{align*} 
One the one hand, since $|q|^2$ is dominated by $U$ by virtue of condition \eqref{cond:U:U(x)=O(x^lambda+1)}, it is clear that
\begin{align*}
n(n-1)V_1^{n-2}\Big| 2H\frac{\varepsilon p}{\sqrt{1+\varepsilon |p|^2}}+\varepsilon q+\frac{q}{|q|}\Big|^2 \le cV_1^{n-1}.
\end{align*}
On the other hand, in light of estimate \eqref{ineq:L_1V_1<-csqrt(V_1)+C}, we readily have
\begin{align*}
 nV_1^{n-1}\L_1V_1 \le -c V_1^{n-\frac{1}{2}}+CV_1^{n-1}.
\end{align*}
Altogether, we invoke Young's inequality to deduce the bound
\begin{align*}
\L_1(V_1^n) \le -cV_1^{n-\frac{1}{2}}+C,
\end{align*}
which is the desired Lyapunov estimate \eqref{ineq:L_1.V_1^n}. The proof is thus complete.

\end{proof}

\subsection{Lyapunov functions for $N$-particle system} \label{sec:poly-mixing:Lyapunov:N-particle}

Similar to the previous subsection, we recall the Hamiltonian for the general $N$-particle system \eqref{eqn:rLE:N-particle:epsilon} is given by
\begin{align} \label{form:H_N}
H_N(\qb,\pb)= \varepsilon\sum_{i=1}^N U(q_i)+\varepsilon\sum_{1\le <j \le N}G(q_i-q_j)+\sum_{i=1}^N \sqrt{1+\varepsilon |p_i|^2}.
\end{align}
Motivated by the choice of $V_1$ in Lemma \ref{lem:Lyapunov:single-particle}, the first Lyapunov ansatz for \eqref{eqn:rLE:N-particle:epsilon} is of the form
\begin{align*}
V = H_N(\qb,\pb)^2 +\varepsilon\la \qb,\pb\ra-\sum_{1\le i<j\le N}\frac{\la q_i-q_j,p_i-p_j\ra}{|q_i-q_j|}+C.
\end{align*}
On the one hand, recalling the generator $\L_N$ given by \eqref{form:L_N}, the choice of $V$ above induces a ``bad" term arising from applying $\L_N$ to the cross term involving $|q_i-q_j|^{-1}$. More specifically, we have
\begin{align*}
\L_N\frac{\la q_i-q_j,p_i-p_j\ra}{|q_i-q_j|} \le c\frac{|p_i-p_j|}{|q_i-q_j|}+ \cdots.
\end{align*}
On the other hand, as demonstrated in the proof of Lemma \ref{lem:Lyapunov:single-particle}, a similar estimate shows that
$\L_N V$ has the order of $-|\pb|$, and thus does not dominate the ``bad" term $|p_i-p_j|/|q_i-q_j|$. In contrast, regarding the Langevin equation \eqref{eqn:LE:original}, this term does not appear when computing It\^o's formula on expression \eqref{form:V:Lu2019}, which therefore qualifies as a Lyapunov function for \eqref{eqn:LE:original}. 

In order to bypass the issue, we seek to reduce the singularity in $\la q_i-q_j,p_i-p_j\ra/|q_i-q_j|$  and increase the order of $-|\pb|$ in $\L_N V$. To this end, our modification of $V$ is as follows: letting $\varepsilon\in (0,1), \kappa_N,A_1,A_2>0$ be given and be chosen later, we introduce the function $V_N$ given by
%\begin{align} \label{form:V_N}
%V_N(\qb,\pb) = A_1 H_N(\qb,\pb)^{2+\gamma}+\varepsilon H_N(\qb,\pb)^\gamma\,\la \qb,\pb\ra -A_2\varepsilon^{1+\gamma}\sum_{i\neq j}\frac{\la q_i-q_j,p_i-p_j\ra}{|q_i-q_j|^\gamma}+\kappa_N.
%\end{align}
\begin{align} \label{form:V_N}
V_N(\qb,\pb) = A_1 H_N(\qb,\pb)^3+\varepsilon H_N(\qb,\pb)\,\la \qb,\pb\ra -A_2\varepsilon^{2}\sum_{i\neq j}\frac{\la q_i-q_j,p_i-p_j\ra}{|q_i-q_j|^{\beta_1-1}}+\kappa_N.
\end{align}
By tuning the parameters $\kappa_N,A_1,A_2$ carefully together with taking $\varepsilon$ sufficiently small, $V_N$ is indeed a Lyapunov for \eqref{eqn:rLE:N-particle:epsilon}, provided we  impose Assumption \nameref{cond:G2} in addition to \nameref{cond:U} and \nameref{cond:G1}. This is summarized in the following result.

\begin{lemma}  \label{lem:Lyapunov:N-particle} For all $N\ge2$, suppose that Assumptions \textup{\nameref{cond:U}}, \textup{\nameref{cond:G1}} and \textup{\nameref{cond:G2}} hold. Then, for all $\varepsilon\in(0,1)$ sufficiently small, the function $V_N$ defined in \eqref{form:V_N} is a Lyapunov function for $P_t^\varepsilon$. In particular, the following holds
\begin{align} \label{ineq:L_N.V_N^n}
\L_N\big[ V_N(\qb,\pb)^n \big] \le -c_nV_N(\qb,\pb)^{n-\frac{1}{3}}+C_n,\quad (\qb,\pb)\in \X, \, n\ge 1, 
\end{align}
for some positive constants $c_n$ and $C_n$, possibly dependent on $\varepsilon$ but independent of $(\qb,\pb)$.

\end{lemma}
\begin{proof}
Similar to the proof of Lemma \ref{lem:Lyapunov:single-particle}, in order to establish \eqref{ineq:L_N.V_N^n}, we shall proceed by induction on $n$. We start with the base case $n=1$ and recall $\L_N$ and $H_N$ from \eqref{form:L_N} and \eqref{form:H_N} respectively
\begin{align} \label{eqn:L_N.H_N}
\L_N H_N = -\sum_{i=1}^N \frac{\varepsilon|p_i|^2}{1+\varepsilon|p_i|^2}+\sum_{i=1}^N\frac{d\varepsilon+(d-1)\varepsilon^2|p_i|^2}{(1+\varepsilon|p_i|^2)^{3/2}}.
\end{align} 
It follows that 
\begin{align*}
\L_N H_N^{3} &= 3 H_N^{2} \bigg[-\sum_{i=1}^N \frac{\varepsilon|p_i|^2}{1+\varepsilon|p_i|^2}+\sum_{i=1}^N\frac{d\varepsilon+(d-1)\varepsilon^2|p_i|^2}{(1+\varepsilon|p_i|^2)^{3/2}}\bigg]\\
&\qquad +6H_N\sum_{i=1}^N\frac{\varepsilon^2|p_i|^2}{1+\varepsilon|p_i|^2}.
\end{align*}
Since
\begin{align} \label{ineq:d.epsilon}
\frac{d\varepsilon+(d-1)\varepsilon^2|p_i|^2}{(1+\varepsilon|p_i|^2)^{3/2}} \le (2d-1)\varepsilon,\quad \text{and}\quad \frac{\varepsilon^2|p_i|^2}{1+\varepsilon|p_i|^2} \le \varepsilon,
\end{align}
we deduce
\begin{align} \label{ineq:L_N.H_N^3}
\L_N H_N^3 &\le 3 H_N^2 \bigg[-\sum_{i=1}^N \frac{\varepsilon|p_i|^2}{1+\varepsilon|p_i|^2}+N(2d-1)\varepsilon\bigg]+6N\varepsilon H_N.
\end{align}

Next, with regard to $\varepsilon H_N\la \qb,\pb\ra$, a routine calculation together with identity \eqref{eqn:L_N.H_N} gives 
\begin{align} \label{eqn:L_N.H_N.<q,p>}
& \L_N \big(\varepsilon H_N\,\la\qb,\pb \ra\big) \notag \\
& = \varepsilon H_N \L_N\la \qb,\pb\ra+\varepsilon  \bigg[-\sum_{i=1}^N \frac{\varepsilon|p_i|^2}{1+\varepsilon|p_i|^2}+\sum_{i=1}^N\frac{d\varepsilon+(d-1)\varepsilon^2|p_i|^2}{(1+\varepsilon|p_i|^2)^{3/2}} \bigg] \la \qb,\pb\ra \notag \\
&\qquad+ 2\varepsilon^2\sum_{i=1}^N \frac{\la p_i,q_i\ra}{\sqrt{1+\varepsilon|p_i|^2}} \notag \\
&=I_1+I_2+I_3.
\end{align}
We proceed to estimate each term on the right hand side of \eqref{eqn:L_N.H_N.<q,p>}. Concerning the term $I_1$ involving $\L_N\la \qb,\pb\ra$, we have
\begin{align*}
\L_N\la \qb,\pb\ra & = \sum_{i=1}^N\frac{|p_i|^2}{\sqrt{1+\varepsilon|p_i|^2}}  - \sum_{i=1}^N\frac{\la q_i,p_i\ra}{\sqrt{1+\varepsilon |p_i|^2}}-\la q_i,\grad U(q_i)\ra\\
&\qquad - \sum_{ i\neq j}\la q_i-q_j,\grad G(q_i-q_j)\ra.
\end{align*}
We employ the same argument as in \eqref{ineq:L_1.<q,p>} while making use of conditions \eqref{cond:U:x.U'(x)>-x^(lambda+1)} and \eqref{cond:G:grad.G(x)<1/|x|^beta} to infer
\begin{align*}
&I_1 =\varepsilon H_N\L_N\la \qb,\pb\ra \\
& \le H_N \sum_{i=1}^N \big[ \sqrt{\varepsilon}|p_i|+\sqrt{\varepsilon}|q_i|-\varepsilon a_2|q_i|^{\lambda+1}\big]+H_N\sum_{i\neq j}\bigg[ \varepsilon a_1|q_i-q_j|+\varepsilon a_1 \frac{1}{|q_i-q_j|^{\beta_1-1}} \bigg]+C\varepsilon H_N\\
&\le  H_N \sum_{i=1}^N \big[ \sqrt{\varepsilon}|p_i|+\sqrt{\varepsilon}|q_i|+C\varepsilon|q_i|-\varepsilon a_2|q_i|^{\lambda+1}+C\big]+H_N\sum_{i\neq j}\varepsilon a_1 \frac{1}{|q_i-q_j|^{\beta_1-1}},
\end{align*}
for some positive constant $C$ independent of $\varepsilon$. In view of \eqref{cond:U+G>1:N>1}, it holds that
\begin{align*}
    H_N\ge \frac{1}{2}\varepsilon\sum_{i=1}^NU(q_i).
\end{align*}
This together with the choice $\lambda\ge 1$ by virtue of Assumption \nameref{cond:U} implies the bound
\begin{align*}
\sqrt{\varepsilon}|	q_i|\le C\varepsilon |q_i|^{\lambda+1}+C \le C H_N,
\end{align*}
We further deduce 
\begin{align*}
&I_1 =\varepsilon H_N\L_N\la \qb,\pb\ra \\
&\le  H_N \sum_{i=1}^N \Big[ \sqrt{\varepsilon}|p_i|-\frac{1}{2}\varepsilon a_2|q_i|^{\lambda+1}+C\Big]+H_N\sum_{i\neq j}\varepsilon a_1 \frac{1}{|q_i-q_j|^{\beta_1-1}} \\
&\le -\frac{1}{2}\varepsilon a_2 H_N \sum_{i=1}^N|q_i|^{\lambda+1}+CH_N\bigg[\sum_{i=1}^N \sqrt{\varepsilon}|p_i| + \sum_{1\le i<j\le N}\frac{\varepsilon}{|q_i-q_j|^{\beta_1-1}}\bigg]+CH_N.
\end{align*}
Likewise, 
\begin{align*}
I_2\le C H_N \sum_{i=1}^N \sqrt{\varepsilon}|p_i|,\quad \text{and}\quad
I_3\le C\varepsilon H_N.
\end{align*}
Altogether with \eqref{eqn:L_N.H_N.<q,p>}, we obtain the bound
\begin{align} \label{ineq:L_N.H_N.<q,p>}
&\L_N\big(\varepsilon H_N \la \qb,\pb\ra \big) \notag \\
 & \le -\frac{1}{2}\varepsilon a_2 H_N \sum_{i=1}^N|q_i|^{\lambda+1}+CH_N\bigg[\sum_{i=1}^N \sqrt{\varepsilon}|p_i| + \sum_{1\le i<j\le N}\frac{\varepsilon}{|q_i-q_j|^{\beta_1-1}}\bigg]+CH_N.
\end{align}
Turning to the term involving $\la q_i-q_j,p_i-p_j\ra/|q_i-q_j|^{\beta_1-1}$ on the right hand side of expression \eqref{form:V_N}. A routine calculation produces
\begin{align} \label{eqn:L_N.<q_i-q_j,p_i-p_j>/|q_i-q_j|^beta1}
&\L_N\bigg(- \sum_{1\le i<j\le N}\Big\la p_i-p_j,\frac{ q_i-q_j}{|q_i-q_j|^{\beta_1-1}}\Big\ra\bigg)\notag  \\
&= - \sum_{1\le i<j\le N} \bigg\la \frac{p_i}{\sqrt{1+\varepsilon|p_i|^2}}- \frac{p_j}{\sqrt{1+\varepsilon|p_j|^2}}, \frac{p_i-p_j}{|q_i-q_j|^{\beta_1-1}} \bigg\ra \notag \\
&\qquad +(\beta_1-1) \sum_{1\le i<j\le N}\la p_i-p_j,q_i-q_j\ra \bigg\la \frac{p_i}{\sqrt{1+\varepsilon|p_i|^2}}- \frac{p_j}{\sqrt{1+\varepsilon|p_j|^2}}, \frac{q_i-q_j}{|q_i-q_j|^{\beta_1+1} } \bigg\ra \notag \\
&\qquad + \sum_{1\le i<j\le N} \bigg\la \frac{q_i-q_j}{|q_i-q_j|^{\beta_1-1}}, \frac{p_i}{\sqrt{1+\varepsilon|p_i|^2}}- \frac{p_j}{\sqrt{1+\varepsilon|p_j|^2}} \bigg\ra\notag \\
&\qquad+\sum_{1\le i<j\le N} \frac{\la \grad U(q_i) -\grad U(q_j),q_i-q_j\ra}{|q_i-q_j|^{\beta_1-1}}        +\sum_{i=1}^N\Big\la  \sum_{j\neq i}\frac{q_i-q_j}{|q_i-q_j|^{\beta_1-1}}  ,\sum_{k\neq i}\grad G(q_i-q_k)    \Big\ra \notag \\
&= I_4+\dots+I_8. 
\end{align}
We proceed to estimate the above right-hand side while making use of Young's inequality. It is clear that
\begin{align*}
\varepsilon^2( I_4+I_5) \le C \sum_{1\le i<j\le N} \sqrt{\varepsilon}|p_i-p_j|\cdot \frac{\varepsilon}{|q_i-q_j|^{\beta_1-1}},
\end{align*}
whereas
\begin{align*}
\varepsilon^2 I_6 \le C\varepsilon^{\frac{1}{2}}\sum_{i=1}^N\varepsilon|q_i|^{2-\beta_1}.
\end{align*}
Concerning $I_7$, we invoke condition \eqref{cond:U:U'(x)=O(x^lambda)} and obtain
\begin{align*}
\sum_{1\le i<j\le N} \frac{\la \grad U(q_i) -\grad U(q_j),q_i-q_j\ra}{|q_i-q_j|^{\beta_1-1}}      
&\le C\sum_{i=1}^N|\grad U(q_i)|\sum_{i=1}^N |q_i|^{2-\beta_1}\le C\Big(1+\sum_{i=1}^N |q_i|^{\lambda+2-\beta_1}\Big),
\end{align*}
whence
\begin{align*}
\varepsilon^{2}I_7 \le \varepsilon
C\Big(1+\sum_{i=1}^N  \varepsilon|q_i|^{\lambda+2-\beta_1}\Big).
\end{align*}
With regard to $I_8$ involving $\grad G$, in light of condition \eqref{cond:G:|grad.G(x)+q/|x|^beta_1|<1/|x|^beta_2}, we recast $I_8$ as follows:
\begin{align*}
& \sum_{i=1}^N\Big\la  \sum_{j\neq i}\frac{q_i-q_j}{|q_i-q_j|^{\beta_1-1}}  ,\sum_{k\neq i}\grad G(q_i-q_k)    \Big\ra \\
&=-a_4\sum_{i=1}^N\Big\la  \sum_{j\neq i}\frac{q_i-q_j}{|q_i-q_j|^{\beta_1-1}}  ,\sum_{k\neq i}\frac{q_i-q_k}{|q_i-q_k|^{\beta_1+1}} \Big\ra\\
&\qquad+ \sum_{i=1}^N\Big\la  \sum_{j\neq i}\frac{q_i-q_j}{|q_i-q_j|^{\beta_1-1}}  ,\sum_{k\neq i}\grad G(q_i-q_k)+ a_4\frac{q_i-q_k}{|q_i-q_k|^{\beta_1+1}}   \Big\ra.
\end{align*}
On the one hand, by virtue of Assumption \nameref{cond:G2}, i.e., $\beta_1\in(1,2]$, we invoke Lemma~\ref{lem:|q_i-q_i|} to get
\begin{align*}
-a_4\sum_{i=1}^N\Big\la  \sum_{j\neq i}\frac{q_i-q_j}{|q_i-q_j|^{\beta_1-1}}  ,\sum_{k\neq i}\frac{q_i-q_k}{|q_i-q_k|^{\beta_1+1}} \Big\ra \le -2a_4 \sum_{1\le i<j\le N}\frac{1}{|q_i-q_j|^{2\beta_1-2}}.
\end{align*}
On the other hand, condition~\eqref{cond:G:|grad.G(x)+q/|x|^beta_1|<1/|x|^beta_2} implies the bound
\begin{align*}
&  \sum_{i=1}^N\Big\la  \sum_{j\neq i}\frac{q_i-q_j}{|q_i-q_j|^{\beta_1-1}}  ,\sum_{k\neq i}\grad G(q_i-q_k)+ a_4\frac{q_i-q_k}{|q_i-q_k|^{\beta_1+1}}   \Big\ra\\
&\qquad\le C\sum_{i=1}^N |q_i|^{2-\beta_1} \sum_{i=1}^N\sum_{\ell\neq i}\Big|\grad G(q_i-x_\ell)+ a_4\frac{q_i-q_\ell}{|q_i-q_\ell|^{\beta_1+1}}   \Big|\\
&\qquad\le C\sum_{i=1}^N |q_i|^{2-\beta_1}\Big[ \sum_{1\le i<j\le N}\frac{1}{|q_i-q_j|^{\beta_2}}+1\Big]\\
&\qquad\le C\sum_{i=1}^N |q_i|^{4-2\beta_1}+ C\sum_{1\le i<j\le N}\frac{1}{|q_i-q_j|^{2\beta_2}}+C.
\end{align*}
Since $\beta_2\le \beta_1-1$, cf. Assumption \nameref{cond:G2}, it is clear that $|q_i-q_j|^{2\beta_2}$ can be subsumed into $|q_i-q_j|^{2\beta_1-2}$. In other words, we have
\begin{align*}
\varepsilon^2 I_8\le -\frac{3}{2}a_4 \varepsilon^2 \sum_{1\le i<j\le N}\frac{1}{|q_i-q_j|^{ 2\beta_1-2 }} + C\varepsilon^2\sum_{i=1}^N |q_i|^{4-2\beta_1}+C.
\end{align*}
We collect the estimates on $I_4,\dots,I_8$ to obtain the bound
\begin{align*}
&\varepsilon^{2}\L_N\bigg(- \sum_{1\le i<j\le N}\Big\la p_i-p_j,\frac{ q_i-q_j}{|q_i-q_j|^{\beta_1}}\Big\ra\bigg)\notag \\
&\le C \sum_{1\le i<j\le N} \sqrt{\varepsilon}|p_i-p_j|\cdot \frac{\varepsilon}{|q_i-q_j|^{\beta_1-1}} + \varepsilon^{\frac{1}{2}}\sum_{i=1}^N\varepsilon|q_i|^{2-\beta_1}+C\\
&\qquad  +\sum_{i=1}^N\varepsilon^2|q_i|^{\lambda+ 2-\beta_1} -\frac{3}{2}a_4 \varepsilon^2 \sum_{1\le i<j\le N}\frac{1}{|q_i-q_j|^{2\beta_1-2}} + C\varepsilon^2\sum_{i=1}^N |q_i|^{4-2\beta_1}.
\end{align*}
Recalling from Assumption \nameref{cond:G2} and \nameref{cond:U} that $\beta_1\in(1,2]$ and that $\lambda\ge 1$, respectively, we note that the terms involving $|q_i|^{2-\beta_1}$, $|q_i|^{\lambda+ 2-\beta_1}$ and $|q_i|^{4-2\beta_1}$ can be controlled by $U(q_i)$. In other words,
\begin{align*}
\varepsilon^{\frac{3}{2}}|q_i|^{2-\beta_1} + \varepsilon^2|q_i|^{\lambda+ 2-\beta_1} +\varepsilon^2 |q_i|^{4-2\beta_1} \le C\varepsilon^{\frac{3}{2}} U(q_i) +C\le C\varepsilon^{\frac{1}{2}}H_N+C.
\end{align*}
Also, by Young's inequality, it holds that
\begin{align*}
 \sum_{1\le i<j\le N} \sqrt{\varepsilon}|p_i-p_j|\cdot \frac{\varepsilon}{|q_i-q_j|^{\beta_1-1}} &\le C\sum_{i=1}^N\varepsilon|p_i|^2+ \frac{1}{100}a_4\sum_{1\le i<j\le N} \frac{\varepsilon^2}{|q_i-q_j|^{2\beta_1-2}} \\
 &\le CH_N+ \frac{1}{100} a_4\sum_{1\le i<j\le N} \frac{\varepsilon^2}{|q_i-q_j|^{2\beta_1-2}}.
\end{align*}
It follows that
\begin{align} \label{ineq:L_N.<p_i-p_j,q_i-q_j>/|q_i-q_j|^beta1}
&\varepsilon^{2}\L_N\bigg(- \sum_{1\le i<j\le N}\Big\la p_i-p_j,\frac{ q_i-q_j}{|q_i-q_j|^{\beta_1-1}}\Big\ra\bigg)\notag \\
&\qquad\le CH_N\sum_{i=1}^N \sqrt{\varepsilon}|p_i| - a_4\sum_{1\le i<j\le N} \frac{\varepsilon^2}{|q_i-q_j|^{2\beta_1-2}}+C.
\end{align}
We emphasize that the positive constant $C=C(N,a_4)$ on the above right hand side does not depend on $\varepsilon$.

Now, letting $A_1$ and $A_2$ be given and be chosen later, from \eqref{ineq:L_N.H_N^3}, \eqref{ineq:L_N.H_N.<q,p>} and \eqref{ineq:L_N.<p_i-p_j,q_i-q_j>/|q_i-q_j|^beta1}, we get
\begin{align*}
& \L_N\bigg(A_1H_N^3+\varepsilon H_N\la \qb,\pb\ra-A_2\varepsilon^{2} \sum_{1\le i<j\le N}\Big\la p_i-p_j,\frac{ q_i-q_j}{|q_i-q_j|^{\beta_1-1}}\Big\ra\bigg)\\
&\le  3A_1 H_N^2 \bigg[-\sum_{i=1}^N \frac{\varepsilon|p_i|^2}{1+\varepsilon|p_i|^2}+N(2d-1)\varepsilon\bigg] -\frac{1}{2}\varepsilon a_2 H_N\sum_{i=1}^N |q_i|^{\lambda+1}\\
&\qquad-a_4A_2\sum_{1\le i<j\le N} \frac{\varepsilon^2}{|q_i-q_j|^{2\beta_1-2}}+I_9,
\end{align*}
where 
\begin{align*}
I_9 &= C A_1\varepsilon H_  N+CH_N\bigg[\sum_{i=1}^N \sqrt{\varepsilon}|p_i| + \sum_{1\le i<j\le N}\frac{\varepsilon}{|q_i-q_j|^{\beta_1-1}}\bigg]\\
&\qquad+CH_N+ CA_2H_N+CA_2+C.
\end{align*}
At this point, we claim that by tuning $A_1,A_2$ appropriately, the following holds
\begin{align} \label{ineq:L_N.H_N^3+H_N<q,p>+<q,p>/|q|^beta}
& \L_N\bigg(A_1H_N^3+\varepsilon H_N\la \qb,\pb\ra-A_2\varepsilon^{2} \sum_{1\le i<j\le N}\Big\la p_i-p_j,\frac{ q_i-q_j}{|q_i-q_j|^{\beta_1-1}}\Big\ra\bigg) \notag \\
&\quad\le -c_1 H_N^2 + C,
\end{align}
for all $\varepsilon$ sufficiently small. In order to prove \eqref{ineq:L_N.H_N^3+H_N<q,p>+<q,p>/|q|^beta}, we will consider two cases depending on the value of $\varepsilon\max\{|p_1|^2,\dots,|p_N|^2\}$.

Case 1: $\varepsilon\max\{|p_1|^2,\dots,|p_N|^2\}\ge 1$. In this case, on the one hand, for $\varepsilon$ sufficiently small,
\begin{align*}
3A_1 H_N^2 \bigg[-\sum_{i=1}^N \frac{\varepsilon|p_i|^2}{1+\varepsilon|p_i|^2}+N(2d-1)\varepsilon\bigg]  \le -\frac{3}{4}A_1H_N^2.
\end{align*}
On the other hand, the positive term $I_9$ satisfies
\begin{align*}
I_9&\le CA_1\varepsilon H_N+C H_N^2+C H_N+CA_2H_N+C+CA_2\\
&\le CH_N^2+D,
\end{align*}
for some positive constants $C$ and $D=D(\varepsilon,A_1,A_2)$. Since $C$ is independent of $\varepsilon,A_1,A_2$, we may take $A_1$ large enough so as to deduce
\begin{align*}
&\L_N\bigg(A_1H_N^3+\varepsilon H_N\la \qb,\pb\ra-A_2\varepsilon^{2} \sum_{1\le i<j\le N}\Big\la p_i-p_j,\frac{ q_i-q_j}{|q_i-q_j|^{\beta_1-1}}\Big\ra\bigg) \\
&\qquad\le -\frac{3}{4}A_1H_N^2+CH_N^2+D\\
&\qquad\le -\frac{1}{2}A_1 H_N^2+D.
\end{align*}
This produces \eqref{ineq:L_N.H_N^3+H_N<q,p>+<q,p>/|q|^beta} for Case 1.

Case 2: $\varepsilon\max\{|p_1|^2,\dots,|p_N|^2\}\le 1$. In this case, we note that \eqref{form:H_N} implies
\begin{align*}
     H_N\le   \sum_{i=1}^N\varepsilon U(q_i)+ \sum_{1\le i<  j\le N}\varepsilon G(q_i-q_j)+ N\sqrt{2}.
\end{align*}
This together with \eqref{cond:G2:G<1/|x|^(beta-1)} yields the bound
\begin{align*}
& CH_N\bigg[\sum_{i=1}^N \sqrt{\varepsilon}|p_i| + \sum_{1\le i<j\le N}\frac{\varepsilon}{|q_i-q_j|^{\beta_1-1}}\bigg] \\
&\le CH_N+C\sum_{i=1}^N \varepsilon U(q_i) \sum_{1\le i<j\le N}\frac{\varepsilon}{|q_i-q_j|^{\beta_1-1}} + C \sum_{1\le i<j\le N}\frac{\varepsilon^2}{|q_i-q_j|^{2\beta_1-2}}+C.
\end{align*}
Concerning the term involving $U$ on the above right-hand side, we employ \eqref{cond:U:U(x)=O(x^lambda+1)} and Young's inequality to infer
\begin{align*}
    &C\sum_{i=1}^N \varepsilon U(q_i) \sum_{1\le i<j\le N}\frac{\varepsilon}{|q_i-q_j|^{\beta_1-1}}\\
    &\le \frac{a_2}{100a_1} \Big[\sum_{i=1}^N \varepsilon U(q_i)\Big]^2 +   C  \sum_{1\le i<j\le N}\frac{\varepsilon^2}{|q_i-q_j|^{2\beta_1-2}}\\
    & \le \frac{1}{100}a_2\sum_{i=1}^N \varepsilon U(q_i)\sum_{i=1}^N \varepsilon |q_i|^{\lambda+1} + \frac{1}{100}a_2\sum_{i=1}^N \varepsilon U(q_i)+   C  \sum_{1\le i<j\le N}\frac{\varepsilon^2}{|q_i-q_j|^{2\beta_1-2}}.
\end{align*}
In view of \eqref{cond:U+G>1:N>1}, we have
\begin{align*}
\sum_{i=1}^N \varepsilon U(q_i) \le 2H_N,
\end{align*}
whence
\begin{align*}
    &C\sum_{i=1}^N \varepsilon U(q_i) \sum_{1\le i<j\le N}\frac{\varepsilon}{|q_i-q_j|^{\beta_1-1}}\\ 
    &\le \frac{1}{50}a_2H_N\sum_{i=1}^N \varepsilon |q_i|^{\lambda+1} + \frac{1}{50}a_2H_N+   C  \sum_{1\le i<j\le N}\frac{\varepsilon^2}{|q_i-q_j|^{2\beta_1-2}}.
\end{align*}
It follows that 
\begin{align*}
& CH_N\bigg[\sum_{i=1}^N \sqrt{\varepsilon}|p_i| + \sum_{1\le i<j\le N}\frac{\varepsilon}{|q_i-q_j|^{\beta_1-1}}\bigg] \\
&\le CH_N+\frac{1}{50}a_2H_N\sum_{i=1}^N \varepsilon |q_i|^{\lambda+1} + C \sum_{1\le i<j\le N}\frac{\varepsilon^2}{|q_i-q_j|^{2\beta_1-2}}+C.
\end{align*}
We emphasize that the above constant $C$ does not depend on $A_1$, $A_2$ and $\varepsilon$. Turning back to $I_9$, we have
\begin{align*}
I_9 &= C A_1\varepsilon H_  N+CH_N\bigg[\sum_{i=1}^N \sqrt{\varepsilon}|p_i| + \sum_{1\le i<j\le N}\frac{\varepsilon}{|q_i-q_j|^{\beta_1-1}}\bigg]\\
&\qquad+CH_N+ CA_2H_N+CA_2+C\\
&\le C(A_1\varepsilon+A_2+1)H_N +\frac{1}{50}a_2H_N\sum_{i=1}^N \varepsilon |q_i|^{\lambda+1} + C \sum_{1\le i<j\le N}\frac{\varepsilon^2}{|q_i-q_j|^{2\beta_1-2}} + D.
\end{align*} 
Also, 
\begin{align*}
 3A_1 H_N^2 \bigg[-\sum_{i=1}^N \frac{\varepsilon|p_i|^2}{1+\varepsilon|p_i|^2}+N(2d-1)\varepsilon\bigg]  \le CA_1\varepsilon H_N^2.
\end{align*}
Altogether, we get
\begin{align*}
&\L_N\bigg(A_1H_N^3+\varepsilon H_N\la \qb,\pb\ra-A_2\varepsilon^{2} \sum_{1\le i<j\le N}\Big\la p_i-p_j,\frac{ q_i-q_j}{|q_i-q_j|^{\beta_1-1}}\Big\ra\bigg) \\
&\le CA_1\varepsilon H_N^2+ C\sum_{1\le i<j\le N}\frac{\varepsilon^2}{|q_i-q_j|^{2\beta_1-2}}-\frac{1}{4}a_2 H_N \sum_{i=1}^N \varepsilon|q_i|^{\lambda+1}\\
&\qquad-a_4A_2 \sum_{1\le i<j\le N}\frac{\varepsilon^2}{|q_i-q_j|^{2\beta_1-2}} +D.
\end{align*}
Since $C$ does not depend on $A_2$, we may take $A_2$ sufficiently large to obtain
\begin{align*}
&\L_N\bigg(A_1H_N^3+\varepsilon H_N\la \qb,\pb\ra-A_2\varepsilon^{2} \sum_{1\le i<j\le N}\Big\la p_i-p_j,\frac{ q_i-q_j}{|q_i-q_j|^{\beta_1-1}}\Big\ra\bigg) \\
&\le CA_1\varepsilon H_N^2-\frac{1}{4}a_2 H_N \sum_{i=1}^N \varepsilon|q_i|^{\lambda+1}- \frac{1}{2}a_4A_2 \sum_{1\le i<j\le N}\frac{\varepsilon^2}{|q_i-q_j|^{2\beta_1-2}} +D.
\end{align*}
We invoke conditions \eqref{cond:U:U(x)=O(x^lambda+1)} and \eqref{cond:G:G<1/|x|^beta} once again to further deduce
\begin{align*}
&\L_N\bigg(A_1H_N^3+\varepsilon H_N\la \qb,\pb\ra-A_2\varepsilon^{2} \sum_{1\le i<j\le N}\Big\la p_i-p_j,\frac{ q_i-q_j}{|q_i-q_j|^{\beta_1-1}}\Big\ra\bigg) \\
&\le CA_1\varepsilon H_N^2 - CH_N^2+D.
\end{align*}
By shrinking $\varepsilon$  to zero if necessary, we establish \eqref{ineq:L_N.H_N^3+H_N<q,p>+<q,p>/|q|^beta}, which completes Case 2.

Turning back to the induction argument for \eqref{ineq:L_N.V_N^n}, it is clear that for $\kappa_N$ large enough (recalling $\beta_1\in(1, 2]$)
\begin{align*}
cH_N^3\le V_N\le CH_N^3+C  .
\end{align*}
From \eqref{ineq:L_N.H_N^3+H_N<q,p>+<q,p>/|q|^beta}, we immediately obtain
\begin{align*}
\L_N V_N \le - c V_N^{\frac{2}{3}}+C.
\end{align*}
In particular, this establishes \eqref{ineq:L_N.V_N^n} for the base case $n=1$.

Next, considering $n\ge 2$, by It\^o's formula, we have the identity
\begin{align*}
\d V_N &= \L_NV_N\d t+ \big(3A_1H_N^2 +  \varepsilon\la \qb,\pb\ra\big)\sum_{i=1}^N  \bigg\la \frac{ \varepsilon p_i}{\sqrt{1+\varepsilon|p_i|^2}},\sqrt{2}\d W_i\bigg\ra\\
&\qquad + \varepsilon H_N \sum_{i=1}^N \la q_i,\sqrt{2}\d W_i\ra - A_2\varepsilon^2 \sum_{i=1}^N \bigg\la \sum_{j\neq i}\frac{q_i-q_j}{|q_i-q_j|^{\beta_1-1}},\sqrt{2}\d W_i\bigg\ra,
\end{align*} 
whence
\begin{align*}
&\L_N V_N^n  = n V_N^{n-1} \L_N V_N \\
&+ \frac{1}{2}n(n-1) V_N^{n-2}\sum_{i=1}^N \bigg| 3A_1 H_N^2+\varepsilon\la \qb,\pb\ra+\varepsilon H_N q_i -A_2\varepsilon^2\sum_{j\neq i}\frac{q_i-q_j}{|q_i-q_j|^{\beta_1-1}}\bigg|^2.
\end{align*}
Since
\begin{align*}
 \bigg| 3A_1 H_N^2+\varepsilon\la \qb,\pb\ra+\varepsilon H_N q_i -A_2\varepsilon^2\sum_{j\neq i}\frac{q_i-q_j}{|q_i-q_j|^{\beta_1-1}}\bigg|^2\le CH_N^4+C\le  CV_N^{4/3},
\end{align*}
we employ the induction hypothesis to arrive at
\begin{align*}
\L_N V_N^n \le -c V_N^{n-\frac{1}{3}} + CV_N^{n-1}+ CV_N^{n-\frac{2}{3}} \le -c V_N^{n-\frac{1}{3}} + C.
\end{align*}
This produces \eqref{ineq:L_N.V_N^n} for the general case $n\ge 2$. The proof is thus finished.

\end{proof}

\subsection{Proof of Theorem \ref{thm:polynomial-mixing:N-particle}} \label{sec:poly-mixing:proof-of-thm}

As mentioned in Section \ref{sec:result:ergodicity}, the proof of Theorem \ref{thm:polynomial-mixing:N-particle} relies on three ingredients including the Lyapunov constructions, the Hormander's condition and the solvability of the control problem associated with \eqref{eqn:rLE:N-particle:epsilon}. The former property was already established in the previous two subsections \ref{sec:poly-mixing:Lyapunov:single-particle} and \ref{sec:poly-mixing:Lyapunov:N-particle}. Turning to the latter two properties, they are presented in the following auxiliary results.

\begin{lemma} \label{lem:Hormander}
System \eqref{eqn:rLE:N-particle:epsilon} satisfies the Hormander's condition as in Definition \ref{def:Hormander}.
\end{lemma}

\begin{lemma} \label{lem:control-problem}
The control problem \eqref{eqn:rLE:N-particle:epsilon:control-problem} satisfies the solvability as stated in Definition \ref{def:control-problem}.
\end{lemma}

For the sake of clarity, the proofs of Lemma \ref{lem:Hormander} and Lemma \ref{lem:control-problem} are deferred by a few lines. Assuming the above results, let us conclude Theorem \ref{thm:polynomial-mixing:N-particle} by verifying the hypothesis of \cite[Theorem 3.5]{hairer2009hot}. Since the argument is short, we include it here for completeness.

\begin{proof}[Proof of Theorem \ref{thm:polynomial-mixing:N-particle}]
First of all, the function $V$ constructed in Lemma \ref{lem:Lyapunov:single-particle} and Lemma \ref{lem:Lyapunov:N-particle} provide the Lyapunov condition required in \cite[Theorem 3.5]{hairer2009hot}. Also, the results of Lemma \ref{lem:Hormander} and Lemma \ref{lem:control-problem} respectively verify \cite[Assumption 2]{hairer2009hot} and \cite[Assumption 3]{hairer2009hot}. Altogether, we invoke \cite[Theorem 3.5]{hairer2009hot} to deduce the uniqueness of the invariant measure $\pi_N^\varepsilon$ as well as the polynomial convergent rate \eqref{ineq:polynomial-ergodicity}, as claimed.

\end{proof}

Turning to the auxiliary result in Lemma \ref{lem:Hormander}, we employ the same argument as in \cite[Page 1639]{ottobre2011asymptotic} to establish the Hormander's condition.

\begin{proof}[Proof of Lemma \ref{lem:Hormander}]
Recalling the family of vector fields $\{X_i\}_{i=0}^N$ given by \eqref{form:X_i}, observe that for $i=1,\dots,N$
\begin{align*}
[X_i,X_0] = \frac{\partial_{q_i}}{\sqrt{1+\varepsilon|p_i|^2}}-\frac{\varepsilon p_i\la p_i,\partial_{q_i}\ra}{(1+\varepsilon|p_i|^2)^{3/2}}-\frac{\partial_{p_i}}{\sqrt{1+\varepsilon|p_i|^2}}+\frac{\varepsilon p_i\la p_i,\partial_{p_i}\ra}{(1+\varepsilon|p_i|^2)^{3/2}}.
\end{align*}
We multiply both sides of the above equation with $p_i$ to produce the identity
\begin{align*}
    \frac{\la p_i,\partial_{q_i}\ra}{(1+\varepsilon|p_i|^2)^{3/2}} = \la p_i,[X_0,X_i]\ra +\frac{\la p_i,\partial_{p_i}\ra}{(1+\varepsilon|p_i|^2)^{3/2}}.
\end{align*}
It follows that 
\begin{align*}
    \frac{\partial_{q_i}}{\sqrt{1+\varepsilon|p_i|^2}}& = [X_i,X_0]+\frac{\partial_{p_i}}{\sqrt{1+\varepsilon|p_i|^2}}-\frac{\varepsilon p_i\la p_i,\partial_{p_i}\ra}{(1+\varepsilon|p_i|^2)^{3/2}}\\
    &\qquad+\varepsilon p_i\Big( \la p_i,[X_0,X_i]\ra +\frac{\la p_i,\partial_{p_i}\ra}{(1+\varepsilon|p_i|^2)^{3/2}} \Big).
\end{align*}
It follows that the collection $\{X_i\}_{i=1}^N\cup\{[X_i,X_0]\}_{i=1}^N$ has maximal rank at each point $X$ in $\X$. In other words, \eqref{eqn:rLE:N-particle:epsilon} satisfies the Hormander's condition, as claimed.

\end{proof}

Lastly, we provide the proof of Lemma \ref{lem:control-problem}, whose argument is adapted to the proof of \cite[Proposition 2.5]{herzog2017ergodicity} tailored to our settings, ensuring the solvability of the control problem \eqref{eqn:rLE:N-particle:epsilon:control-problem}. This together with Lemma \ref{lem:Hormander} and the Lyapunov constructions in Subsections \ref{sec:poly-mixing:Lyapunov:single-particle}-\ref{sec:poly-mixing:Lyapunov:N-particle} ultimately conclude Theorem \ref{thm:polynomial-mixing:N-particle}.

\begin{proof}[Proof of Lemma \ref{lem:control-problem}]

Let $X_j=(\qb_j,\pb_j) = (q_{j,1},\dots,q_{j,N})\times (p_{j,1},\dots,p_{j,N})\in \X=\D\times (\rbb^d)^N$, $j=0,1$, be given. Since $\D$ is path connected, there exists a time $T=T(\qb_0,\qb_1)$ sufficiently large and a path $\f_1=(\f_{1,i})_{i=1}^N\in C^\infty([0,T];\D)$ such that 
\begin{align*}
   \f_{1,i}(0)=q_{0,i},\quad \f_{1,i}(T)=q_{1,i} ,
\end{align*}
and that
\begin{align*}
\sup_{s\in[0,T]}|\f_{1,i} '(s)|\le \frac{1}{100},\quad i=1,\dots,N.
\end{align*}
Letting $\rho\in(0,1)$ be given and be chosen later, we introduce the function $g\in C^2([0,2\rho];\rbb)$ defined as
 \begin{align*}
g(s)= \begin{cases}
-\frac{(s-\rho)^{100}}{100\rho^{99}}+\frac{\rho}{100},& 0\le s\le \rho,\\
\frac{\rho}{100},& \rho\le s\le 2\rho.
\end{cases}
\end{align*}
Then, we consider the functions $\f_2:[0,2\rho]\to (\rbb^d)^N$ and $\f_3:[T-2\rho,T]\to (\rbb^d)^N$ respectively given by
\begin{align*}
\f_{2,i}(s)=q_{i,0}+g(s)\frac{p_{i,0}}{\sqrt{1+\varepsilon|p_{i,0}|^2}},
\end{align*}
and
\begin{align*}
\f_{3,i}(s)=q_{i,1}+g(T-s)\frac{p_{i,1}}{\sqrt{1+\varepsilon|p_{i,1}|^2}}.
\end{align*}
To produce a single path $Z=(\qb,\pb)$ satisfying the control problem \eqref{eqn:rLE:N-particle:epsilon:control-problem}, consider $\theta:\rbb\to\rbb$ a smooth cut-off function satisfying
\begin{align*} 
\theta(t) = \begin{cases} 
1,&  |t|\le 1,\\
\text{monotonicity},& 1\le |t|\le 2,\\
0,& |t|\ge 2,
\end{cases}
\end{align*}
and that 
\begin{align*}
\sup_{s\in\rbb}|\theta'(s)|\le 10.
\end{align*}
Given the above functions, a solution candidate $Z$ to the control problem \eqref{eqn:rLE:N-particle:epsilon:control-problem} is given by
\begin{align} 
q_i(s)&=\begin{cases} \f_{2,i}(s),& 0\le s\le \rho,\\
\theta\big(\frac{s}{\rho}\big)\f_{2,i}(s)+\Big[1-\theta\big(\frac{s}{\rho}\big)\Big]\f_{1,i}(s),& \rho\le s\le 2\rho,\\
\f_{1,i}(s),& 2\rho\le s\le T-2\rho,\\
\big[1-\theta\big(\frac{T-s}{\rho}\big)\big]\f_{1,i}(s)+\theta\big(\frac{T-s}{\rho}\big)\f_{3,i},& T-2\rho\le s\le T-\rho,\\
\f_{3,i},& T-\rho\le s\le T,
\end{cases}\label{form:control-problem:q_i.p_i}\\
\textup{and}\quad p_i(s)& = \frac{q_i'(s)}{\sqrt{1-\varepsilon|q_i'(s)|^2}},\quad 0\le s\le T.\notag
\end{align}

To see this, we first note that by taking $\rho$ sufficiently small, $\qb(s)$, $s\in[0,T]$, is guaranteed to remain in $\D$. Also, it is clear that
\begin{align*}
q_i(0)=q_{0,i},\quad q_{i}(T) = e_1,\quad p_i(0)=p_{0,i},\quad p_i(T)=p_{1,i}.
\end{align*}
Next, we claim that for all $\varepsilon<1/100$
\begin{align} \label{ineq:control-problem:sup_[0,T]]|q'|}
\sup_{s\in[0,T]}|q_i'(s)| < \frac{1}{\sqrt{\varepsilon}}.
\end{align}
Indeed, for each $s\in[0,\rho]$, by the choice of $g$,
\begin{align*}
|q_i'(s)| = |\f_{2,i}'(s)| \le |g'(s)|\frac{|p_{i,0}|}{\sqrt{1+\varepsilon|p_{i,0}|^2}} < \frac{1}{\sqrt{\varepsilon}}.
\end{align*}
Likewise, for each $s\in[T-\rho,T]$,
\begin{align*}
|q_i'(s)| = |\f_{3,i}'(s)| \le |g'(T-s)|\frac{|p_{i,1}|}{\sqrt{1+\varepsilon|p_{i,1}|^2}} < \frac{1}{\sqrt{\varepsilon}}.
\end{align*}
For $s\in[2\rho,T-2\rho]$,
\begin{align*}
|q_i'(s)|  = |\f_{1,i}'(s)|\le 1<\frac{1}{\sqrt{\varepsilon}},
\end{align*}
Concerning the interval $[\rho,2\rho]$, a calculation gives
\begin{align*}
q_i'(s)= \theta'\Big(\frac{s}{\rho}\Big) \frac{1}{\rho}\big[\f_{2,i}(s)-\f_{1,i}(s)\big]+\theta\Big(\frac{s}{\rho}\Big)\f_{2,i}'(s)+\Big[1-\theta\Big(\frac{s}{\rho}\Big)\Big]\f_{1,i}'(s).
\end{align*}
Since for $s\in[\rho,2\rho]$, $\f_{2,i}'(s)=0$, it is clear that
\begin{align*}
\theta\Big(\frac{s}{\rho}\Big)|\f_{2,i}'(s)|+\Big[1-\theta\Big(\frac{s}{\rho}\Big)\Big]|\f_{1,i}'(s)|\le \frac{1}{100},
\end{align*}
whereas (recalling $|\theta'|\le 10$)
\begin{align*}
\Big|\theta'\Big(\frac{s}{\rho}\Big) \frac{1}{\rho}\big[\f_{2,i}(s)-\f_{1,i}(s)\big]\Big|& \le \frac{10}{\rho}\big[|\f_{2,i}(s)-q_{0,i}|+|\f_{1,i}(s)-q_{0,i}|\big]\\
&\le \frac{1}{10}\cdot\frac{|p_{0,i}|}{\sqrt{1+\varepsilon|p_{0,i}|^2}}+10\cdot\frac{s}{\rho}\cdot \sup_{r\in[0,T]}|\f_{1,i}'(r)| \\
&\le \frac{1}{10}\cdot\frac{1}{\sqrt{\varepsilon}}+\frac{1}{5}.
\end{align*}
It follows that
\begin{align*}
\sup_{s\in[\rho,2\rho]}|q_i'(s)|\le \frac{1}{10}\cdot\frac{1}{\sqrt{\varepsilon}}+\frac{1}{5}+\frac{1}{100}<\frac{1}{\sqrt{\varepsilon}}.
\end{align*}
Similarly, for $s\in[T-2\rho,T-\rho]$, 
\begin{align*}
&|q_i'(s)| \\
&\le \Big|\theta'\Big(\frac{T-s}{\rho}\Big) \frac{1}{\rho}\big[\f_{1,i}(s)-q_{1,i}\big]\Big|+\Big|\theta'\Big(\frac{T-s}{\rho}\Big) \frac{1}{\rho}\big[\f_{3,i}(s)-q_{1,i}\big]\Big|+\Big[1-\theta\Big(\frac{T-s}{\rho}\Big)\Big]|\f_{1,i}'(s)|\\
&\le 10\frac{ T-s}{\rho}\sup_{r\in[0,T]}|\f_{1,i}'(r)|+\frac{1}{10}\cdot\frac{|p_{1,i}|}{\sqrt{1+\varepsilon|p_{1,i}|^2}}+\Big[1-\theta\Big(\frac{T-s}{\rho}\Big)\Big]\sup_{r\in[0,T]}|\f_{1,i}'(r)|\\
&\le \frac{1}{5}+\frac{1}{10\sqrt{\varepsilon}}+\frac{1}{100}<\frac{1}{\sqrt{\varepsilon}}.
\end{align*}
Altogether, we establish \eqref{ineq:control-problem:sup_[0,T]]|q'|}, as claimed. 

As a consequence of \eqref{ineq:control-problem:sup_[0,T]]|q'|}, $p_i(\cdot)$ defined in \eqref{form:control-problem:q_i.p_i} is a well-defined path in $\rbb^d$. Lastly, to verify $\{q_i(\cdot),p_i(\cdot)\}_{i=1}^N$ solves the control problem \eqref{eqn:rLE:N-particle:epsilon:control-problem}, we pick
\begin{align*}
 U_i(t)=\frac{1}{\sqrt{2}}\int_0^t \Big[\frac{p_i(s)}{\sqrt{1+\varepsilon|p_i(t)|^2}} +\grad U(q_i(t))+ \sum_{j\neq i }\grad G(q_i(t)-q_j(t)) +p_i'(t) \Big]\d t.
\end{align*}
This together with the identity of $p_i$ given by \eqref{form:control-problem:q_i.p_i} establishes that the family of paths $$\{q_i(t),p_i(t),U_i(t)\}_{i=1}^N,\quad t\in[0,T],$$ solves \eqref{eqn:rLE:N-particle:epsilon:control-problem}. The proof is thus finished.

\end{proof}

\section{Newtonian limit of \eqref{eqn:rLE:N-particle:epsilon}} \label{sec:newtonian-limit}

In this section, we establish Theorem \ref{thm:newtonian-limit} giving the  validity of the Newtonian limit $\varepsilon\to0$ for \eqref{eqn:rLE:N-particle:epsilon} toward the classical Langevin equation \eqref{eqn:LE:original} on any finite time window. Due to the complication caused by the relativistic forcing together with the singularities, we have to treat two cases $N\ge 2$ and $N=1$ separately. The main argument of Theorem \ref{thm:newtonian-limit} essentially consists of three steps as follows: 

Step 1: We assume that the nonlinearites in \eqref{eqn:rLE:N-particle:epsilon} and \eqref{eqn:LE:original} are Lipschitz function, and established a convergence in $L^p$ as $\varepsilon\to 0$. This is demonstrated in Proposition \ref{prop:newtonian-limit:Lipschitz} below.

Step 2: When $N\ge 2$, we remove the Lipschitz constraint in Step 1 by performing an analysis on the relation between \eqref{eqn:rLE:N-particle:epsilon} and \eqref{eqn:LE:original} while making use of moment bounds on \eqref{eqn:LE:original} on finite time windows. In particular, this produces the convergence in probability as $\varepsilon\to 0$. The explicit argument is carried out in Section \ref{sec:poly-mixing:Lyapunov:N-particle} where we supply the proof of Theorem \ref{thm:newtonian-limit}, part 1.

Step 3: When $N=1$, we provide a uniform moment bound on the $\varepsilon$-system \eqref{eqn:rLE:N-particle:epsilon} through suitable Lyapunov functions. In turn, this allows us to bypass the Lipschitz hypothesis from Step 1 and deduce a convergence in $L^p$, thereby concluding the proof of Theorem \ref{thm:newtonian-limit}, part 2. All of this will be presented in Section \ref{sec:newtonian-limit:proof-of-thm:part-2}.

\subsection{Newtonian limit for Lipschitz systems} \label{sec:newtornian-limit:Lipschitz}
In this subsection, we establish an analogue of Theorem \ref{thm:newtonian-limit} under the extra hypothesis that the nonlinear potentials are Lipschitz functions. For notational convenience, letting $K_i:\D\to\rbb^d$ be Lipschitz, consider the following relativistic system
\begin{align} \label{eqn:rLE:N-particle:epsilon:Lipschitz}
\d\, q_i(t) &= \frac{p_i(t)}{\sqrt{1+\varepsilon |p_i(t)|^2}}\d t, \notag\\
\d\, p_i(t) & = -\frac{p_i(t)}{\sqrt{1+\varepsilon |p_i(t)|^2}} \d t +K_i(\qb(t)) \d t +\sqrt{2} \,\d W_{i}(t).
\end{align}
and the corresponding Langevin system
\begin{align} \label{eqn:LE:N-particle:Lipschitz}
\d\, q_i(t) &= p_i(t)\d t, \notag\\
\d\, p_i(t) & = -p_i(t)\d t +K_i(\qb(t)) \d t +\sqrt{2} \,\d W_{i}(t).
\end{align}
We claim that \eqref{eqn:rLE:N-particle:epsilon:Lipschitz} are well-approximated by \eqref{eqn:LE:N-particle:Lipschitz} in the Newtonian regime $\varepsilon\to 0$. This is summarized in the following proposition.

\begin{proposition} \label{prop:newtonian-limit:Lipschitz}
Suppose that the functions $K_i$ as in \eqref{eqn:rLE:N-particle:epsilon:Lipschitz}-\eqref{eqn:LE:N-particle:Lipschitz} are Lipschitz in $\D$. For every initial condition $(\qb_0,\pb_0)\in \X$, let $X_\varepsilon^K(t)=\big(\qb^{\varepsilon,K}(t),\pb^{\varepsilon,K}(t)\big)$ and $X^K(t)=\big(\qb^K(t),\pb^K(t)\big)$, respectively, be the solutions of \eqref{eqn:rLE:N-particle:epsilon:Lipschitz} and \eqref{eqn:LE:N-particle:Lipschitz}. Then, for all $T>0$ and $n\ge 1$, the following holds
\begin{align} \label{lim:newtonian-limit:Lipschitz}
\E \sup_{t\in[0,T]}\big[ |\qb^{\varepsilon,K}(t)-\qb^K(t)|^{ 2n }+|\pb^{\varepsilon,K}(t)-\pb^K(t)|^{ 2n }\big] \le \varepsilon^n C,
\end{align}
for some positive constant $C=C(n,K,T,\qb_0,\pb_0,N)$ independent of $\varepsilon$.
\end{proposition}

In order to establish Proposition \ref{prop:newtonian-limit:Lipschitz}, we will employ the following auxiliary result providing a uniform moment bound on \eqref{eqn:rLE:N-particle:epsilon:Lipschitz} with respect to the parameter $\varepsilon$.

\begin{lemma} \label{lem:E.sup.|p_epsilon^K|^n<C}
Under the same hypothesis of Proposition \ref{prop:newtonian-limit:Lipschitz}, for every initial condition $(\qb_0,\pb_0)\in \X$, let $X^{\varepsilon,K}(t)=\big(\qb^{\varepsilon,K}(t),\pb^{\varepsilon,K}(t)\big)$ be the solution of \eqref{eqn:rLE:N-particle:epsilon:Lipschitz}. Then, for all $T>0$ and $n\ge 1$,
\begin{align} \label{ineq:E.sup.|p_epsilon^K|^n<C}
\E \sup_{t\in[0,T]}|\pb^{\varepsilon,K}(t)|^n \le C,
\end{align}
for some positive constant $C=C(n,T,\qb_0,\pb_0)$ independent of $\varepsilon$.
\end{lemma}

For the sake of clarity, we defer the proof of Lemma \ref{lem:E.sup.|p_epsilon^K|^n<C} to the end of this subsection. Assuming the result of Lemma \ref{lem:E.sup.|p_epsilon^K|^n<C}, let us prove Proposition \ref{prop:newtonian-limit:Lipschitz}.
\begin{proof}[Proof of Proposition \ref{prop:newtonian-limit:Lipschitz}]  For notational convenience, we will drop the superscript $K$ in $(\qb^{\varepsilon,K},\pb^{\varepsilon,K})$ and $(\qb^K,\pb^K)$ through out the proofs in this subsection.

Setting $\qhat_i = q_i^{\varepsilon}-q_i$, $\phat_i=p_i^{\varepsilon}-p_i$, and subtracting \eqref{eqn:rLE:N-particle:epsilon:Lipschitz} from \eqref{eqn:LE:N-particle:Lipschitz}, observe that $(\qhat_i(t),\phat_i(t))$ satisfies the following equation 
\begin{align*} 
\frac{\d}{\d t} \qhat_i &= \phat_i+ \frac{p_i^{\varepsilon}}{\sqrt{1+\varepsilon |p_i^\varepsilon|^2}}-p_i^\varepsilon, \notag\\
\frac{\d}{\d t} \phat_i & =-\phat_i -\frac{p_i^\varepsilon}{\sqrt{1+\varepsilon |p_i^\varepsilon|^2}}+p_i^\varepsilon  +K_i(\qb^\varepsilon)-K_i(\qb) ,
\end{align*}
with initial condition $\qhat_i(0)=\phat_i(0)=0$.
A routine calculation gives
\begin{align*}
\frac{1}{2}\frac{\d}{\d t}\big(|\qhat_i|^2+|\phat_i|^2\big)& = \la \qhat_i,\phat_i\ra-|\phat_i|^2+\Big\la \qhat_i,\frac{p_i^{\varepsilon}}{\sqrt{1+\varepsilon |p_i^\varepsilon|^2}}-p_i^\varepsilon\Big\ra\\
&\qquad -\Big\la \phat_i,\frac{p_i^{\varepsilon}}{\sqrt{1+\varepsilon |p_i^\varepsilon|^2}}-p_i^\varepsilon \Big\ra +\la \phat_i,K_i(\qb^\varepsilon)-K_i(\qb)\ra.
\end{align*}
Note that for all $p\in\rbb^d$, we have
\begin{align*}
\Big|\frac{p}{\sqrt{1+\varepsilon|p|^2}}-p\Big| & = \frac{\varepsilon|p|^3}{ \sqrt{1+\varepsilon|p|^2}   (1+\sqrt{1+\varepsilon|p|^2})}\le \sqrt{\varepsilon}|p|^2.
\end{align*}
Together with the fact that $K_i$ is Lipschitz, we employ Young's inequality to infer
\begin{align*}
\frac{1}{2}\frac{\d}{\d t}\big(|\qhat_i|^2+|\phat_i|^2\big)&\le C\big(|\hat{\qb}|^2+|\phat_i|^2+\varepsilon|p_i^\varepsilon|^4\big) .
\end{align*}
In other words, summing over $i=1,\dots,N$ yields
\begin{align*}
\frac{1}{2}\frac{\d}{\d t}\big(|\hat{\qb}|^2+|\hat{\pb}|^2\big)&\le C\big(|\hat{\qb}|^2+|\hat{\pb}|^2+\varepsilon|\pb^\varepsilon|^4\big).
\end{align*}
In the above, $C$ is a positive constant independent of $\varepsilon$. As a consequence, for all $n\ge 1$ and $T>0$, Gronwall's inequality implies the bound
\begin{align*}
\E \sup_{t\in[0,T]} |(\hat{\qb}(t),\hat{\pb}(t))|^{2n}\le C \varepsilon^n\E\sup_{t\in[0,T]} |\pb^\varepsilon(t)|^{4n}.
\end{align*}
In light of Lemma \ref{lem:E.sup.|p_epsilon^K|^n<C}, cf. \eqref{ineq:E.sup.|p_epsilon^K|^n<C}, the above estimate immediately produces \eqref{lim:newtonian-limit:Lipschitz}, as claimed.

\end{proof}

We now turn to Lemma \ref{lem:E.sup.|p_epsilon^K|^n<C}, whose result was employed to conclude Proposition \ref{prop:newtonian-limit:Lipschitz}.

\begin{proof}[Proof of Lemma \ref{lem:E.sup.|p_epsilon^K|^n<C}]
By It\^o's formula applied to \eqref{eqn:rLE:N-particle:epsilon:Lipschitz}, we have
\begin{align} \label{eqn:d.|q^(epsilon,K)|^2+|p^(epsilon,K)|^2}
\frac{1}{2}\d\big( |q_i^\varepsilon|^2+|p_i^\varepsilon|^2\big) & =\bigg[ \frac{\la q_i^\varepsilon,p_i^\varepsilon\ra }{\sqrt{1+\varepsilon|p_i^\varepsilon|^2}}-  \frac{|p_i^\varepsilon|^2}{\sqrt{1+\varepsilon|p_i^\varepsilon|^2}}+\la p_i^\varepsilon,K_i(\qb^\varepsilon)\ra+d\bigg]\d t+\sqrt{2}\la p_i^\varepsilon,\d W_i\ra.
\end{align}
We note that the following holds
\begin{align} \label{ineq:1+p^2/sqrt(1+epsilon.p^2)>p}
1+\frac{|p|^2}{\sqrt{1+\varepsilon|p|^2}}\ge |p|,\quad \varepsilon\in[0,1],\, p\in\rbb^d.
\end{align}
Indeed, \eqref{ineq:1+p^2/sqrt(1+epsilon.p^2)>p} is equivalent to
\begin{align*}
1+\varepsilon|p|^2 + 2|p|^2\sqrt{1+\varepsilon|p|^2}+|p|^4\ge |p|^2+\varepsilon|p|^4,
\end{align*}
which always holds true for all $\varepsilon\in[0,1]$ and $p\in\rbb^d$. Together with the hypothesis that $K_i$ is Lipschitz, from \eqref{eqn:d.|q^(epsilon,K)|^2+|p^(epsilon,K)|^2}, we infer the bound
\begin{align} \label{ineq:d.|q^(epsilon,K)|^2+|p^(epsilon,K)|^2}
\frac{1}{2}\d\big( |q_i^\varepsilon|^2+|p_i^\varepsilon|^2\big) &\le  C(|\qb^\varepsilon|^2+|\pb^\varepsilon|^2+1)\d t -|p_i^\varepsilon|^2\d t+ \sqrt{2}\la p_i^\varepsilon,\d W_i\ra,
\end{align}
for some positive constant $C=C(K_i)$ independent of $\varepsilon$. 

Next, for $n\ge 1$, we apply It\^o's formula once again to obtain
\begin{align*} 
\d\big( |q_i^\varepsilon|^2+|p_i^\varepsilon|^2\big)^n & =  n\big( |q_i^\varepsilon|^2+|p_i^\varepsilon|^2\big)^{n-1}\d \big( |q_i^\varepsilon|^2+|p_i^\varepsilon|^2\big) \\
&\qquad+ \frac{1}{2}n(n-1)\big( |q_i^\varepsilon|^2+|p_i^\varepsilon|^2\big)^{n-2}\big \la \d \big( |q_i^\varepsilon|^2+|p_i^\varepsilon|^2\big),\d \big( |q_i^\varepsilon|^2+|p_i^\varepsilon|^2\big)\big\ra. 
\end{align*}
On the one hand, from \eqref{eqn:d.|q^(epsilon,K)|^2+|p^(epsilon,K)|^2}, we have
\begin{align*}
    &\frac{1}{2}n(n-1)\big( |q_i^\varepsilon|^2+|p_i^\varepsilon|^2\big)^{n-2}\big \la \d \big( |q_i^\varepsilon|^2+|p_i^\varepsilon|^2\big),\d \big( |q_i^\varepsilon|^2+|p_i^\varepsilon|^2\big)\big\ra\\
    &= 4n(n-1)\big( |q_i^\varepsilon|^2+|p_i^\varepsilon|^2\big)^{n-2}|p_i^\varepsilon|^2\d t\\
    &\le 4n(n-1)\big( |q_i^\varepsilon|^2+|p_i^\varepsilon|^2\big)^{n-1}\d t.
\end{align*}
On the other hand, \eqref{ineq:d.|q^(epsilon,K)|^2+|p^(epsilon,K)|^2} implies
\begin{align*}
    & n\big( |q_i^\varepsilon|^2+|p_i^\varepsilon|^2\big)^{n-1}\d \big( |q_i^\varepsilon|^2+|p_i^\varepsilon|^2\big) \\
    & \le C(|\qb^\varepsilon|^2+|\pb^\varepsilon|^2)^n\d t+C\d t + 2\sqrt{2}n\big( |q_i^\varepsilon|^2+|p_i^\varepsilon|^2\big)^{n-1}\la p_i^\varepsilon,\d W_i\ra.
\end{align*}
So, we get
\begin{align} \label{ineq:d.|q^(epsilon,K)|^2n+|p^(epsilon,K)|^2n}
    &\d\big( |q_i^\varepsilon|^2+|p_i^\varepsilon|^2\big)^n \notag \\
    &\le C(|\qb^\varepsilon|^2+|\pb^\varepsilon|^2)^n\d t+C\d t + 2\sqrt{2}n\big( |q_i^\varepsilon|^2+|p_i^\varepsilon|^2\big)^{n-1}\la p_i^\varepsilon,\d W_i\ra ,
\end{align}
whence
\begin{align*}
    \E |(\qb^\varepsilon(t),\pb^\varepsilon(t))|^n \le |(\qb_0,\pb_0)|^n+C\int_0^t \E |(\qb^\varepsilon(s),\pb^\varepsilon(s))|^n\d s +Ct.
\end{align*}
As a consequence, we obtain by virtue of Gronwall's inequality 
\begin{align} \label{ineq:E.int_0^T|q^(epsilon,K)|^2n+|p^(epsilon,K)|^2n}
    \int_0^T \E |(\qb^\varepsilon(t),\pb^\varepsilon(t))|^n\d t \le C.
\end{align}
In the above, we emphasize that the positive constant $C(T,n,K_1\dots,K_N,\qb_0,\pb_0)$ does not depend on $\varepsilon$.

Turning back to \eqref{ineq:E.sup.|p_epsilon^K|^n<C}, we invoke \eqref{ineq:d.|q^(epsilon,K)|^2n+|p^(epsilon,K)|^2n} to infer
\begin{align*}
    \sup_{t\in[0,T]}\big( |q_i^\varepsilon(t)|^2+|p_i^\varepsilon|^2(t)\big)^n & \le \big( |q_i^\varepsilon(0)|^2+|p_i^\varepsilon(0)|^2\big)^n+\int_0^T  |(\qb^\varepsilon(t),\pb^\varepsilon(t)|^n\d t +CT\\
    &\qquad+ 2\sqrt{2}n\sup_{t\in[0,T]} \Big|\int_0^t \big( |q_i^\varepsilon(s)|^2+|p_i^\varepsilon(s)|^2\big)^{n-1}\la p_i^\varepsilon(s),\d W_i(s)\ra\Big|.
\end{align*}
In view of Burkholder's inequality, we get
\begin{align*}
    & \E\sup_{t\in[0,T]} \Big|\int_0^t \big( |q_i^\varepsilon(s)|^2+|p_i^\varepsilon(s)|^2\big)^{n-1}\la p_i^\varepsilon(s),\d W_i(s)\ra\Big|\\
    &\le C\E\int_0^T \big(|q_i^\varepsilon(s)|^2+|p_i^\varepsilon(s)|^2\big)^{2n-2}|p_i^\varepsilon(s)|^2\d s+C\\
    &\le C\E \int_0^T  |(\qb^\varepsilon(t),\pb^\varepsilon(t))|^{2n-1}\d t+C.
\end{align*}
This together with \eqref{ineq:E.int_0^T|q^(epsilon,K)|^2n+|p^(epsilon,K)|^2n} implies
\begin{align*}
    \E \sup_{t\in[0,T]}\big( |q_i^\varepsilon(t)|^2+|p_i^\varepsilon(t)|^2\big)^n \le C
\end{align*}
where $C=C(T,n,K_1\dots,K_N,\qb_0,\pb_0)$ is independent of $\varepsilon$. This establishes \eqref{ineq:E.sup.|p_epsilon^K|^n<C}, as claimed.

\end{proof}

\subsection{Proof of Theorem \ref{thm:newtonian-limit}, part 1} \label{sec:newtonian-limit:proof-of-thm:part-1}

In this subsection, we assume that $N\ge 2$ and establish Theorem \ref{thm:newtonian-limit}, part 1. As mentioned at the beginning of this section, we will employ the result of Proposition \ref{prop:newtonian-limit:Lipschitz} to produce the convergence in probability \eqref{lim:newtonian-limit:N-particle:prob} on any finite time window.

Adapting from the framework \cite{duong2024asymptotic,herzog2016small}, we start the procedure by presenting an auxiliary result concerning the Langevin equation \eqref{eqn:LE:original} and its moment bounds, which do not depend on $\varepsilon$.

\begin{lemma} \label{lem:LE:N-particle:moment-estimate}
For all $(\qb_0,\pb_0)\in \X$, $T>0$, $n\ge 1$, the following holds
\begin{align} \label{ineq:LE:moment-estimate:sup_[0,T]}
\E \sup_{t\in[0,T]}\bigg[\sum_{i=1}^N U(q_i(t))+\sum_{1\le i<j\le N} G(q_i(t)-q_j(t)) \bigg]^n\le C,
\end{align}
for some positive constant $C=C(n,T,\qb_0,\pb_0)$.
\end{lemma}

Since the proof of Lemma \ref{lem:LE:N-particle:moment-estimate} is similar to that of Lemma \ref{lem:E.sup.|p_epsilon^K|^n<C}, we defer the argument to the end of this subsection. Next, we shall truncate the nonlinearities in \eqref{eqn:rLE:N-particle:epsilon} and \eqref{eqn:LE:original} as follows: for $R>2$, let $\theta_R:\rbb\to\rbb$ be a smooth function satisfying
\begin{align} \label{form:theta_R}
\theta_R(t) = \begin{cases} 
1,&  |t|\le R,\\
\text{monotonicity},& R\le |t|\le R+1,\\
0,& |t|\ge R+1.
\end{cases}
\end{align}
With the above cut-off $\theta_R$, we consider the truncated version of \eqref{eqn:rLE:N-particle:epsilon} given by
\begin{align} \label{eqn:rLE:N-particle:epsilon:truncating}
\d\, q_i^{\varepsilon,R}(t) &= \frac{p_i^{\varepsilon,R} }{\sqrt{1+\varepsilon |p_i^{\varepsilon,R}|^2}}\d t, \notag\\
\d\, p_i^{\varepsilon,R}(t) & = -\frac{p_i^{\varepsilon,R}}{\sqrt{1+\varepsilon |p_i^{\varepsilon,R}|^2}} \d t-\theta_R(|q_i^{\varepsilon,R}(t)|) \grad \U(q_i^{\varepsilon,R}(t))\d t +\sqrt{2} \,\d W_{i}(t)\notag \\
&\qquad - \sum_{j\neq i}\theta_R\big(|q_i^{\varepsilon,R}(t)-q_j^{\varepsilon,R}(t)|^{-1}\big) \grad \G\big(q_i^{\varepsilon,R}(t)-q_j^{\varepsilon,R}(t)\big) \d t ,
\end{align}
as well as the following truncated version of \eqref{eqn:LE:original}
\begin{align} \label{eqn:LE:N-particle:truncating}
\d\, q_i^{R}(t) &= p_i^{R} \d t, \notag\\
\d\, p_i^{R}(t) & = -p_i^{R}\d t -\theta_R(|q_i^{R}(t)|) \grad \U(q_i^{R}(t))\d t +\sqrt{2} \,\d W_{i}(t) \notag \\
&\qquad - \sum_{j\neq i}\theta_R\big(|q_i^{R}(t)-q_j^{R}(t)|^{-1}\big) \grad \G\big(q_i^{R}(t)-q_j^{R}(t)\big) \d t.
\end{align}

Given the above truncated systems and assuming the result of Lemma \ref{lem:LE:N-particle:moment-estimate}, we now provide the proof of Theorem \ref{thm:newtonian-limit}, part 1 giving the convergence in probability for the $N-$particle system with respect to the Newtonian limit. The approach of which is drawn upon the technique of \cite[Theorem 2.10]{duong2024asymptotic} tailored to our settings. In particular, the argument will combine Lemma \ref{lem:LE:N-particle:moment-estimate} with Proposition \ref{prop:newtonian-limit:Lipschitz} as well as Lemma \ref{lem:U+G>|q|-log(q)} so as to establish the desired limit \eqref{lim:newtonian-limit:N-particle:prob}

\begin{proof}[Proof of Theorem \ref{thm:newtonian-limit}, part 1] 
For $R,\,\varepsilon>0$, we introduce the stopping times defined as
\begin{align}  \label{form:stopping-time:sigma^R}
\sigma^R = \inf_{t\geq 0}\Big\{\sum_{i=1}^N|q_i(t)|+\close\sum_{1\le i<j\le N}\close|q_i(t)-q_j(t)|^{-1}\geq R\Big\},
\end{align}
and
\begin{align} \label{form:stopping-time:sigma^R_varepsilon}
\sigma^R_\varepsilon = \inf_{t\geq 0}\Big\{\sum_{i=1}^N|q_i^\varepsilon(t)|+\close\sum_{1\le i<j\le N}\close |q_i^\varepsilon(t)-q_j^\varepsilon(t)|^{-1}\geq R\Big\}.
\end{align}
Fixing $T,\,\xi>0$, observe that
\begin{align} \label{ineq:P(|x-q|>xi)}
&\P\Big(\sup_{t\in[0,T]}\big( |\qb^\varepsilon(t)-\qb(t)|+ |\pb^\varepsilon(t)-\pb(t)|) >\xi\Big) \notag \\
&\leq \P\Big(\sup_{t\in[0,T]}|\qb^\varepsilon(t)-\qb(t)|+ |\pb^\varepsilon(t)-\pb(t)|>\xi,\sigma^R\mi\sigma^R_\varepsilon\geq T\Big) +\P\big(\sigma^R\mi\sigma^R_\varepsilon<T\big).
\end{align}
With regard to the first term on the above right-hand side, we note that
\begin{align*}
\P\Big(0\leq t\leq \sigma^R\mi\sigma^R_\varepsilon,\, \big(\qb^\varepsilon(t),\pb^\varepsilon(t)\big)=\big(\qb^{\varepsilon,R}(t),\pb^{\varepsilon,R}(t)\big),\, \big(\qb(t),\pb(t)\big)=\big(\qb^R(t),\pb^R(t)\big) \Big)=1,
\end{align*}
where $\big(\qb^{\varepsilon,R}(t),\pb^{\varepsilon,R}(t)\big)$ and $\big(\qb^R(t),\pb^R(t)\big)$ respectively are the solutions of the truncated equations \eqref{eqn:rLE:N-particle:epsilon:truncating} and \eqref{eqn:LE:N-particle:truncating}. In particular, being Lipschitz systems, they satisfy the hypothesis of Proposition \ref{prop:newtonian-limit:Lipschitz}. As a consequence, we invoke \eqref{lim:newtonian-limit:Lipschitz} while making use of Markov's inequality to infer
\begin{align} \label{ineq:P(|x-q|>xi,sigma>T)}
&\P\Big(\sup_{0\leq t\leq T} |\qb^{\varepsilon}(t)-\qb(t)| + |\pb^{\varepsilon}(t)-\pb(t)|>\xi,\sigma^R\mi\sigma^R_\varepsilon\geq T\Big) \nt \\
&\leq \P\Big(\sup_{0\leq t\leq T}|\qb^{\varepsilon,R}(t)-\qb^R(t)| + |\pb^{\varepsilon,R}(t)-\pb^R(t)|>\xi\Big)\le  \frac{\varepsilon}{\xi}\cdot C(T,R).
\end{align} 
Turning to the second term $\P(\sigma^R\mi\sigma^R_\varepsilon<T)$ on the right-hand side of \eqref{ineq:P(|x-q|>xi)}, we have the following implication
\begin{align*}
&\P\big(\sigma^R\mi\sigma^R_\varepsilon<T\big)\nt \\
&\leq \P\Big(\sup_{t\in[0,T]}|\qb^{\varepsilon,R}(t)-\qb^R(t)| + |\pb^{\varepsilon,R}(t)-\pb^R(t)|\leq \frac{\xi}{R},\sigma^R\mi\sigma^R_\varepsilon < T\Big) \notag\\
&\qquad +\P\Big(\sup_{t\in[0,T]}|\qb^{\varepsilon,R}(t)-\qb^R(t)| + |\pb^{\varepsilon,R}(t)-\pb^R(t)|> \frac{\xi}{R}\Big),
\end{align*}
whence
\begin{align}
&\P\big(\sigma^R\mi\sigma^R_\varepsilon<T\big)\nt \\
&\leq  \P\Big(\sup_{t\in[0,T]}|\qb^{\varepsilon,R}(t)-\qb^R(t)| + |\pb^{\varepsilon,R}(t)-\pb^R(t)|\leq\frac{\xi}{R},\sigma^R_\varepsilon < T\leq\sigma^R\Big)+\P(\sigma^R< T) \nt \\
&\qquad\qquad\qquad+\P\Big(\sup_{t\in[0,T]}|\qb^{\varepsilon,R}(t)-\qb^R(t)| + |\pb^{\varepsilon,R}(t)-\pb^R(t)|> \frac{\xi}{R}\Big)  \nt \\
&= I_1+I_2+I_3. \label{ineq:P(sigma<T):I_1+I_2+I_3}
\end{align}
We proceed to estimate each term on the above right-hand side and start with $I_3$. The same argument as in \eqref{ineq:P(|x-q|>xi,sigma>T)} making use of \eqref{lim:newtonian-limit:Lipschitz} implies the bound 
\begin{align}\label{ineq:P(sigma<T):I_3}
I_3=\P\Big(\sup_{t\in[0,T]}|\qb^{\varepsilon,R}(t)-\qb^R(t)| + |\pb^{\varepsilon,R}(t)-\pb^R(t)|> \frac{\xi}{R}\Big)\le  \frac{\varepsilon}{\xi}\cdot C(T,R).
\end{align}
Next, considering $I_2$, let $C_G$ and $c_G$ be the constants from Lemma \ref{lem:U+G>|q|-log(q)}. From the expression \eqref{form:stopping-time:sigma^R} together with the second inequality of \eqref{ineq:U+G>|q|-log(q)}, we get
\begin{align*}
\big\{  \sigma^R<T\big\}&=\Big\{\sup_{t\in[0,T]}\sum_{i=1}^N|q_i(t)|+\close\sum_{1\le i<j\le N}\close|q_i(t)-q_j(t)|^{-1}\geq R\Big\}\\
&\subseteq \Big\{  \sup_{t\in[0,T]}\sum_{i=1}^N|q_i(t)|\geq \frac{R}{N}\Big\}\bigcup_{1\le i<j\le N} \Big\{\sup_{t\in[0,T]} -c_G\log|q_i(t)-q_j(t)|\geq c_G\log\Big(\frac{R}{N^2}\Big)\Big\}.
\end{align*}
On the one hand, we apply \eqref{ineq:|q|-log|q|>0}, i.e.,
\begin{align*}
    \frac{1}{2}\sum_{i=1}^N|q_i(t)|-c_G\close\sum_{1\le i<j\le N}\close\log|q_i(t)-q_j(t)|\ge 0,
\end{align*}
to infer for all $R$ large enough
\begin{align*}
    &\Big\{  \sup_{t\in[0,T]}\sum_{i=1}^N|q_i(t)|\geq \frac{R}{N}\Big\}\\
    &\subseteq  \Big\{ \sup_{t\in[0,T]}\sum_{i=1}^N|q_i(t)|-c_G\close\sum_{1\le i<j\le N}\close\log|q_i(t)-q_j(t)| \geq \frac{1}{2}\frac{R}{N}\Big\}\\
    &\subseteq  \Big\{ \sup_{t\in[0,T]}\sum_{i=1}^N|q_i(t)|-c_G\close\sum_{1\le i<j\le N}\close\log|q_i(t)-q_j(t)| \geq c_G\log\Big(\frac{R}{N^2}\Big)\Big\}.
\end{align*}
On the other hand, we invoke the second inequality of \eqref{ineq:U+G>|q|-log(q)} to estimate
\begin{align*}
    &\Big\{\sup_{t\in[0,T]} -c_G\log|q_i(t)-q_j(t)|\geq c_G\log\Big(\frac{R}{N^2}\Big)\Big\}\\
    &\subseteq \Big\{\sup_{t\in[0,T]} \Big[\max_{1\le \ell<k\le N}-c_G\log|q_\ell(t)-q_k(t)|\Big]\geq c_G\log\Big(\frac{R}{N^2}\Big)\Big\}\\
    &\subseteq \Big\{\sup_{t\in[0,T]} \Big[c_G\sum_{k=1}^N|q_k(t)|+  \max_{1\le \ell<k\le N}-c_G\log|q_\ell(t)-q_k(t)|\Big]\geq c_G\log\Big(\frac{R}{N^2}\Big)\Big\}\\
    &\subseteq  \Big\{ \sup_{t\in[0,T]}\sum_{i=1}^N|q_i(t)|-c_G\close\sum_{1\le i<j\le N}\close\log|q_i(t)-q_j(t)| \geq c_G\log\Big(\frac{R}{N^2}\Big)\Big\}.
\end{align*}
So, we get
\begin{align} \label{ineq:sigma^R<T}
    \big\{  \sigma^R<T\big\} &\subseteq \Big\{ \sup_{t\in[0,T]}\sum_{i=1}^N|q_i(t)|-c_G\close\sum_{1\le i<j\le N}\close\log|q_i(t)-q_j(t)| \geq c_G\log\Big(\frac{R}{N^2}\Big)\Big\} \notag \\
&\subseteq \Big\{ \sup_{t\in[0,T]}\sum_{i=1}^N U(q_i)+\sum_{1\le i<j\le N}G(q_i-q_i) \geq \frac{c_G}{C_G}\log\Big(\frac{R}{N^2}\Big)\Big\}.
\end{align}
where the last inclusion follows from the first inequality of \eqref{ineq:U+G>|q|-log(q)}. In view of Lemma \ref{lem:LE:N-particle:moment-estimate} together with Markov's inequality, we infer the bound for $R$ large enough
\begin{align} \label{ineq:P(sigma<T):I_2}
I_2=\P(\sigma^R<T)\le \frac{C_G}{c_G}\cdot \frac{C(T,\qb_0,\pb_0)}{\log(R/N^2)}\le \frac{C(T)}{ \log R}.
\end{align}
Turning to $I_1$ on the right-hand side of \eqref{ineq:P(sigma<T):I_1+I_2+I_3}, we observe that 
\begin{align*}
&\Big\{\sup_{t\in[0,T]} |\qb^{\varepsilon,R}(t)-\qb^R(t)|+|\pb^{\varepsilon,R}(t)-\pb^R(t)| \leq\frac{\xi}{R},\sigma^R_\varepsilon < T\leq\sigma^R\Big\}\\
&= \Big\{\sup_{t\in[0,T]}|\qb^{\varepsilon,R}(t)-\qb(t)|+|\pb^{\varepsilon,R}(t)-\pb(t)| \leq\frac{\xi}{R},\sigma^R_\varepsilon < T\leq\sigma^R\Big\}\\
&\hspace{2cm}\bigcap\Big\{ \sup_{t\in[0,T]}\Big(\sum_{i=1}^N|q_i^{\varepsilon, R} (t)|+\close\sum_{1\le i<j\le N}\close|q_{i}^{\varepsilon, R}(t)-q_{j}^{\varepsilon, R}(t)|^{-1}\Big)\ge R\Big\}  \\
&\subseteq \Big\{\sup_{t\in[0,T]}|\qb^{\varepsilon, R}(t)-\qb(t)|\leq\frac{\xi}{R},\sup_{t\in[0,T]}\sum_{i=1}^N|q_i^{\varepsilon,R}(t)|\ge \frac{R}{N}\Big\}\\
&\qquad\qquad \bigcup_{1\le i<j\le N}\Big\{\sup_{t\in[0,T]}|\qb^{\varepsilon, R}(t)-\qb(t)|\leq\frac{\xi}{R},\sup_{t\in[0,T]}|q_i^{\varepsilon,R}(t)-q_j^{\varepsilon,R}(t)|^{-1}\ge \frac{R}{N^2}\Big\}.
\end{align*}
In other words, we obtain
\begin{align*}
&\Big\{\sup_{t\in[0,T]} |\qb^{\varepsilon,R}(t)-\qb^R(t)|+|\pb^{\varepsilon,R}(t)-\pb^R(t)| \leq\frac{\xi}{R},\sigma^R_\varepsilon < T\leq\sigma^R\Big\}\\
&\subseteq \Big\{\sup_{t\in[0,T]}|\qb^{\varepsilon, R}(t)-\qb(t)|\leq\frac{\xi}{R},\sup_{t\in[0,T]}\sum_{i=1}^N|q_i^{\varepsilon,R}(t)|\ge \frac{R}{N}\Big\}\\
&\qquad\qquad \bigcup_{1\le i<j\le N}\Big\{\sup_{t\in[0,T]}|\qb^{\varepsilon, R}(t)-\qb(t)|\leq\frac{\xi}{R},\inf_{t\in[0,T]}|q_i^{\varepsilon,R}(t)-q_j^{\varepsilon,R}(t)|\le \frac{N^2}{R}\Big\}\\
&=: B_1 \bigcup_{1\le i<j\le N} B_{ij}.
\end{align*}
On the one hand, we have
\begin{align*}
B_1&=\Big\{\sup_{t\in[0,T]}|\qb^{\varepsilon, R}(t)-\qb(t)|\leq\frac{\xi}{R},\sup_{t\in[0,T]}\sum_{i=1}^N|q_i^{\varepsilon,R}(t)|\ge \frac{R}{N}\Big\} \\
& \subseteq \Big\{\sup_{t\in[0,T]}\sum_{i=1}^N|q_i^{\varepsilon, R}(t)-q_i(t)|\leq\frac{\xi \sqrt{N} }{R},\sup_{t\in[0,T]}\sum_{i=1}^N|q_i^{\varepsilon,R}(t)|\ge \frac{R}{N}\Big\} \\
&\subseteq  \Big\{\sup_{t\in[0,T]}\sum_{i=1}^N|q_i(t)|\ge \frac{R}{N}- \frac{\xi\sqrt{N}}{R}\Big\}.
\end{align*}
Since $\xi$ and $N$ are fixed, for $R$ large enough, say, $\frac{R}{N}- \frac{\xi\sqrt{N}}{R}\ge \sqrt{R}$, it holds that
\begin{align*}
B_1& \subseteq  \Big\{\sup_{t\in[0,T]}\sum_{i=1}^N|q_i(t)|\ge \sqrt{R}\Big\}\\
&\subseteq \Big\{\sup_{t\in[0,T]}\sum_{i=1}^N|q_i(t)|-c_G\close\sum_{1\le i<j\le N}\close\log|q_i(t)-q_j(t)| \ge \frac{1}{2}\sqrt{R}\Big\},
\end{align*}
where the last inclusion follows from \eqref{ineq:|q|-log|q|>0}. We invoke the first inequality of \eqref{ineq:U+G>|q|-log(q)} from Lemma \ref{lem:U+G>|q|-log(q)} to deduce further
\begin{align*}
   B_1 &\subseteq \Big\{ \sup_{t\in[0,T]}\sum_{i=1}^NU(q_i(t))+\close\sum_{1\le i<j\le N}\close G(q_i(t)-q_j(t)) \geq \frac{c_G}{2C_G}\sqrt{R}\Big\}.
\end{align*} 
On the other hand, by triangle inequality,
\begin{align*}
\inf_{t\in[0,T]}|q_{i}(t)-q_{j}(t)| \le 2\sup_{t\in[0,T]}|\qb^{\varepsilon, R}(t)-\qb(t)|+\inf_{t\in[0,T]}|q_{i}^{\varepsilon, R}(t)-q_{j}^{\varepsilon, R}(t)|,
\end{align*}
implying
\begin{align*}
B_{ij}& = \Big\{\sup_{t\in[0,T]}|\qb^{\varepsilon, R}(t)-\qb(t)|\leq\frac{\xi}{R},\inf_{t\in[0,T]}|q_i^{\varepsilon,R}(t)-q_j^{\varepsilon,R}(t)|\le \frac{N^2}{R}\Big\}\\
&\subseteq   \Big\{\inf_{t\in[0,T]}|q_{i}(t)-q_{j}(t)|\le \frac{2\xi+N^2}{R}\Big\}\\
&= \Big\{\sup_{t\in[0,T]}|q_{i}(t)-q_{j}(t)|^{-1}\ge \frac{R}{2\xi+N^2}\Big\}\\
&= \Big\{\sup_{t\in[0,T]}-\log|q_{i}(t)-q_{j}(t)|\ge \log R -\log (2\xi+N^2)\Big\}.
\end{align*}
We employ the second inequality of \eqref{ineq:U+G>|q|-log(q)} again to infer for all $R$ large enough
\begin{align*}
B_{ij}&\subseteq \Big\{ \sup_{t\in[0,T]}-\log|q_{i}(t)-q_{j}(t)|\ge \log R -\log (2\xi+N^2)\Big\}\\
&\subseteq \Big\{ \sup_{t\in[0,T]}\Big[\max_{1\le \ell<k\le N}-\log|q_{\ell}(t)-q_{k}(t)|\Big]\ge \log R -\log (2\xi+N^2)\Big\}.
\end{align*}
As a consequence of the first inequality of \eqref{ineq:U+G>|q|-log(q)} and estimate \eqref{ineq:|q|-log|q|>0}, we get
\begin{align*}
B_{ij}&\subseteq \Big\{ \sup_{t\in[0,T]} \sum_{i=1}^N |q_i(t)|-c_G\sum_{ 1\le i<j \le N}\log|q_i(t)-q_j(t)|   \geq \frac{1}{2}c_G \log R\Big\}\\
&\subseteq \Big\{ \sup_{t\in[0,T]} \sum_{i=1}^N U(q_i(t))+\sum_{1\le i<j \le N}G(q_i(t)-q_j(t))   \geq \frac{c_G}{2C_G} \log R\Big\}.
\end{align*}
Altogether from the estimates on $B_1$ and $B_{ij}$, we obtain
\begin{align*}
&\Big\{\sup_{t\in[0,T]} |\qb^{\varepsilon,R}(t)-\qb^R(t)|+|\pb^{\varepsilon,R}(t)-\pb^R(t)| \leq\frac{\xi}{R},\sigma^R_\varepsilon < T\leq\sigma^R\Big\}\\
&\subseteq \Big\{ \sup_{t\in[0,T]} \sum_{i=1}^N U(q_i(t))+\sum_{1\le i<j \le N}G(q_i(t)-q_j(t))   \geq \frac{c_G}{2C_G} \log R\Big\}.
\end{align*}
In turn, \eqref{ineq:LE:moment-estimate:sup_[0,T]} implies that
\begin{align} \label{ineq:P(sigma<T):I_1}
I_1&=\P\Big\{\sup_{t\in[0,T]} |\qb^{\varepsilon,R}(t)-\qb^R(t)|+|\pb^{\varepsilon,R}(t)-\pb^R(t)| \leq\frac{\xi}{R},\sigma^R_\varepsilon < T\leq\sigma^R\Big\}\notag \\
&\le\frac{C(T,\qb_0,\pb_0)}{ \log R}.
\end{align}
Turning back to \eqref{ineq:P(sigma<T):I_1+I_2+I_3}, we collect~\eqref{ineq:P(sigma<T):I_3}, \eqref{ineq:P(sigma<T):I_2} and \eqref{ineq:P(sigma<T):I_1} to arrive at the bound
\begin{align} \label{ineq:P(sigma<T)}
&\P\big(\sigma^R\mi\sigma^R_\varepsilon<T\big)\le  \frac{\varepsilon}{\xi}\cdot C(T,R)+ \frac{C(T)}{\log R}.
\end{align}
In the above estimate, while the positive constants $C(T,R)$ and $C(T)$ may be very large as $R,T$ tend to infinity, we emphasize that they are independent of $\varepsilon$.

Now, putting everything together, from \eqref{ineq:P(|x-q|>xi)}, \eqref{ineq:P(|x-q|>xi,sigma>T)} and \eqref{ineq:P(sigma<T)}, we obtain the estimate
\begin{align*}
&\P\Big(\sup_{t\in[0,T]}\big( |\qb^\varepsilon(t)-\qb(t)|+ |\pb^\varepsilon(t)-\pb(t)|\big) >\xi\Big)\\
& \le \frac{\varepsilon}{\xi}\cdot C(T,R)+ \frac{C(T)}{\log R},
\end{align*}
for all $R$ large enough. By sending $R$ to infinity and then shrinking $\varepsilon$ further to zero, this produces the convergence in probability \eqref{lim:newtonian-limit:N-particle:prob}, thereby completing the proof.

\end{proof}

Lastly, in this subsection, we provide the proof of Lemma \ref{lem:LE:N-particle:moment-estimate}, which together with Proposition \eqref{prop:newtonian-limit:Lipschitz} was employed to establish Theorem \ref{thm:newtonian-limit}, part 1 for the case $N\ge 2$.

\begin{proof}[Proof of Lemma \ref{lem:LE:N-particle:moment-estimate}]
We introduce the function 
\begin{align*}
\Gamma_1(\qb,\pb)= \sum_{i=1}^N U(q_i)+ \sum_{1\le i<j\le N}G(q_i-q_j)+\frac{1}{2}|\pb|^2. 
\end{align*}
We claim that for all $n\ge 2$, the following holds
\begin{align} \label{ineq:E.int_0^T|Gamma_1|}
    \E\int_0^T \Gamma_1(t)^n\d t \le C(n,T,\qb_0,\pb_0).
\end{align}
Indeed, from \eqref{eqn:LE:original}, we apply It\^o's formula to ``(...) we apply It\^o's formula to $ \Gamma_1(t)=\Gamma_1(\qb(t),\pb(t))$ and obtain the identity
\begin{align*} 
\d \Gamma_1(t) & = -|\pb(t)|^2\d t + Nd\,\d t+ \sqrt{2}\sum_{i=1}^N \la p_i(t),\d W_i(t)\ra.
\end{align*}
More generally, for $n\ge 2$, we have
\begin{align*}
    \d \Gamma_1(t)^{n} & = n\Gamma_1(t)^{n-1}\Big(-|\pb(t)|^2\d t + Nd\,\d t + \sqrt{2}\sum_{i=1}^N \la p_i(t),\d W_i(t)\ra \Big) \notag \\
    &\quad +n(n-1)\Gamma_1(t)^{n-2}|\pb(t)|^2\d t.
\end{align*}
Recalling condition \eqref{cond:U+G>1:N>1}, we note that $\Gamma_1\ge \frac{1}{2}|\pb|^2$. Hence, we may exploit Young's inequality to infer
\begin{align*}
    \Gamma_1^{n-2}|\pb|^2\le \Gamma_1^{n-1}\le \Gamma_1^n+C,
\end{align*}
whence
\begin{align} \label{ineq:d.Gamma_1^n}
     \d \Gamma_1(t)^{n} & = n\Gamma_1(t)^{n-1}\Big(-|\pb(t)|^2\d t + Nd\,\d t + \sqrt{2}\sum_{i=1}^N \la p_i(t),\d W_i(t)\ra \Big) \notag \\
    &\quad +n(n-1)\Gamma_1(t)^{n}\d t+C\d t.
\end{align}
It follows that
\begin{align*}
    \E \Gamma_1(t)^{n} & \le \E\Gamma_1(0)^n+ CT +C\int_0^T \E\Gamma_1(s)^n\d s .
\end{align*}
By virtue of Gronwall's inequality, we deduce \eqref{ineq:E.int_0^T|Gamma_1|}, as claimed.

Turning back to \eqref{ineq:LE:moment-estimate:sup_[0,T]}, we employ \eqref{ineq:d.Gamma_1^n} once again to deduce
\begin{align*}
    \E \sup_{t\in[0,T]}\Gamma_1(t)^n      &\le \Gamma_1(0)^n + CT +C\E\int_0^T\Gamma_1(s)^n\d s\\
    &\qquad+n\sqrt{2}\E\sup_{t\in[0,T]}\Big|\int_0^t \Gamma_1(s)^{n-1}\sum_{i=1}^N\la p_i(s),\d W_i(s)\ra  \Big|.
\end{align*}
We invoke Burkholder's inequality to further estimate
\begin{align*}
    \E\sup_{t\in[0,T]}\Big|\int_0^t \Gamma_1(s)^{n-1}\sum_{i=1}^N\la p_i(s),\d W_i(s)\ra  \Big| &\le C\E\int_0^T \Gamma_1(t)^{2n-2}|\pb(t)|^2\d s +C\\
    &\le C\E\int_0^T \Gamma_1(t)^{2n-1}\d s +C.
\end{align*}
As a consequence, we deduce
\begin{align*}
     \E \sup_{t\in[0,T]}\Gamma_1(t)^n  \le  \Gamma_1(0)^n + CT +C\E\int_0^T\Gamma_1(s)^n+\Gamma_1(s)^{2n-1}\d s+C.
\end{align*}
In view of \eqref{ineq:E.int_0^T|Gamma_1|}, we obtain
\begin{align*}
    \E \sup_{t\in[0,T]}\Gamma_1(t)^n  \le  \Gamma_1(0)^n + CT +C,
\end{align*}
whence
\begin{align*}
    \E \sup_{t\in[0,T]}\bigg[\sum_{i=1}^N U(q_i(t))+\sum_{1\le i<j\le N} G(q_i(t)-q_j(t)) \bigg]^n \le  \Gamma_1(0)^n + CT +C.
\end{align*}
This establishes \eqref{ineq:LE:moment-estimate:sup_[0,T]}, thereby finishing the proof.

\end{proof}

\subsection{Proof of Theorem \ref{thm:newtonian-limit}, part 2} \label{sec:newtonian-limit:proof-of-thm:part-2}

In this subsection, we restrict our proof arguments to the single-particle case $N=1$ and improve the convergence in probability argument in Section \ref{sec:newtonian-limit:proof-of-thm:part-1}. Namely, we establish the $L^p$ convergence of \eqref{eqn:rLE:single-particle:epsilon} toward the single-particle Langevin equation, thereby concluding Theorem \ref{thm:newtonian-limit}, part 2.

For the reader's convenience, when $N=1$, the full equation \eqref{eqn:LE:original} is reduced to the following equation:
\begin{align} \label{eqn:LE:single-particle}
\d\, q(t) &= p(t)\d t, \notag\\
\d\, p(t) & = -p(t) \d t -\grad U(q(t))\d t- \grad G\big(q(t)\big) \d t +\sqrt{2} \,\d W(t).
\end{align}

The analogue of Lemma \ref{lem:LE:N-particle:moment-estimate} for \eqref{eqn:LE:single-particle} is provided below through Lemma \ref{lem:LE:single-particle:moment-estimate}, giving arbitrary moment bounds on \eqref{eqn:LE:single-particle} in any finite time window. Since the proof of this result is the same as that of Lemma \ref{lem:LE:N-particle:moment-estimate}, the explicit argument is omitted.

\begin{lemma} \label{lem:LE:single-particle:moment-estimate}
For all $(q_0,p_0)\in \X$, let $(q(t),p(t))$ be the solution of \eqref{eqn:LE:single-particle} with initial condition $(q_0,p_0)$. Then, for all $T>0$ and $n\geq 1$, the following holds
\begin{align*}
\E \sup_{t\in[0,T]}\bigg[U(q(t))+G(q(t))+\frac{1}{2}|p(t)|^2 \bigg]^n\le C,
\end{align*}
for some positive constant $C=C(n,T,q_0,p_0)$.
\end{lemma}

Next, we turn to the solution $(q^\varepsilon,p^\varepsilon)$ of \eqref{eqn:rLE:single-particle:epsilon} and establish uniform moment bounds on $(q^\varepsilon,p^\varepsilon)$ with respect to the parameter $\varepsilon$. We note that this is the crucial result on which Theorem \ref{thm:newtonian-limit}, part 2 relies.

 \begin{lemma} \label{lem:rLE:single-particle:moment-estimate:sup_[0,T]}
Suppose that Assumptions \textup{\nameref{cond:U}} and \textup{\nameref{cond:G1}} hold. Let $(q^\varepsilon(t),p^\varepsilon(t))$ be the solution of \eqref{eqn:rLE:single-particle:epsilon} with initial condition $(q_0,p_0)\in \X$. Then, for all $T>0$ and $n\ge 1$, the following holds
\begin{align}\label{ineq:rLE:single-particle:moment-estimate:sup_[0,T]}
\sup_{\varepsilon\in(0,1]}\E  \sup_{t\in[0,T]}\Big[ U(q^\varepsilon(t))+G(q^{\varepsilon}(t))+|p^\varepsilon(t)|^2\Big]^n \le C(T,n,q_0,p_0),
\end{align}
for some positive constant $C(T,n,q_0,p_0)$ independent of $\varepsilon$.
 \end{lemma}

For the sake of clarity, we defer the proof of Lemma \ref{lem:rLE:single-particle:moment-estimate:sup_[0,T]} to the end of this subsection. Assuming the result of Lemma \ref{lem:rLE:single-particle:moment-estimate:sup_[0,T]}, let us conclude the proof of Theorem \ref{thm:newtonian-limit}, part 2 giving the $L^p$ convergence in the single-particle case, while making use of the auxiliary result in Lemma \ref{lem:U+G>|q|-log(q)}. See also \cite[Theorem 2.2]{higham2002strong}

\begin{proof}[Proof of Theorem \ref{thm:newtonian-limit}, part 2]
Similar to the proof of Theorem \ref{thm:newtonian-limit}, part 1, for $R>0$, we introduce the stopping times defined as
\begin{align*} 
\tau^R = \inf_{t\geq 0}\Big\{|q(t)|+|q(t)|^{-1}\geq R\Big\},
\end{align*}
and
\begin{align*}
\tau^R_{\varepsilon} = \inf_{t\geq 0}\Big\{|q^\varepsilon(t)|+ |q^\varepsilon(t))|^{-1}\geq R\Big\}.
\end{align*}
Also, recalling the smooth cut-off function $\theta_R$ defined in \eqref{form:theta_R}, we consider the truncated systems 
\begin{align} \label{eqn:rLE:single-particle:epsilon:truncating}
\d\, q^{\varepsilon,R}(t) &= \frac{p^{\varepsilon,R} }{\sqrt{1+\varepsilon |p^{\varepsilon,R}|^2}}\d t, \notag\\
\d\, p^{\varepsilon,R}(t) & = -\frac{p^{\varepsilon,R}}{\sqrt{1+\varepsilon |p^{\varepsilon,R}|^2}} \d t-\theta_R(|q^{\varepsilon,R}(t)|) \grad \U(q^{\varepsilon,R}(t))\d t +\sqrt{2} \,\d W(t)\notag \\
&\qquad -\theta_R\big(|q^{\varepsilon,R}(t)|^{-1}\big) \grad \G\big(q^{\varepsilon,R}(t)\big) \d t,
\end{align}
and
\begin{align} \label{eqn:LE:single-particle:truncating}
\d\, q^{R}(t) &= p^{R} \d t, \notag\\
\d\, p^{R}(t) & = -p^{R}\d t -\theta_R(|q^{R}(t)|) \grad \U(q^{R}(t))\d t +\sqrt{2} \,\d W(t) \notag \\
&\qquad - \theta_R\big(|q^{R}(t)|^{-1}\big) \grad \G\big(q^{R}(t)\big) \d t.
\end{align}
Observe that
\begin{align*}
\P\big( 0\le t\le \tau^R \mi \tau_\varepsilon^R,\, (q^\varepsilon(t),p^\varepsilon(t))=(q^{\varepsilon,R}(t),p^{\varepsilon,R}(t)) , \, (q(t),p(t))  =(q^R(t),p^R(t))    \big)=1.
\end{align*}
So, 
\begin{align} \label{ineq:E.sup_[0,T].|q^epsilon-q|:single-particle}
&\E \sup_{t\in[0,T]}\big[ |q^\varepsilon(t)-q(t)|^n+|p^\varepsilon(t)-p(t)|^n\big] \notag \\
&= \E \Big[\sup_{t\in[0,T]}\big[ |q^\varepsilon(t)-q(t)|^n+|p^\varepsilon(t)-p(t)|^n\big]\cdot  \mathbf{1}\{T\le \tau^R\mi\tau^R_\varepsilon\}  \Big] \notag \\
&\qquad + \E \Big[\sup_{t\in[0,T]}\big[ |q^\varepsilon(t)-q(t)|^n+|p^\varepsilon(t)-p(t)|^n\big]\cdot  \mathbf{1}\{T\ge \tau^R\mi\tau^R_\varepsilon\}  \Big] \notag\\
&\le \E \sup_{t\in[0,T]}\big[ |q^{\varepsilon,R}(t)-q^R(t)|^n+|p^{\varepsilon,R}(t)-p^R(t)|^n\big] \notag\\
&\qquad + \E \Big[\sup_{t\in[0,T]}\big[ |q^\varepsilon(t)-q(t)|^n+|p^\varepsilon(t)-p(t)|^n\big]\cdot  \mathbf{1}\{\tau^R \le T\}  \Big] \notag \\
&\qquad +\E \Big[\sup_{t\in[0,T]}\big[ |q^\varepsilon(t)-q(t)|^n+|p^\varepsilon(t)-p(t)|^n\big]\cdot  \mathbf{1}\{\tau^R_\varepsilon \le T\}  \Big] \notag \\
&=:I_1+I_2+I_3.
\end{align}
By truncating the nonlinear potentials, observe that \eqref{eqn:rLE:single-particle:epsilon:truncating}-\eqref{eqn:LE:single-particle:truncating} are Lipschitz systems. In light of Proposition \ref{prop:newtonian-limit:Lipschitz}, we readily obtain
\begin{align*}
I_1 =  \E \sup_{t\in[0,T]}\big[ |q^{\varepsilon,R}(t)-q^R(t)|^n+|p^{\varepsilon,R}(t)-p^R(t)|^n\big] \le \varepsilon^n C(T,R).
\end{align*}
With regard to $I_2$, we invoke Holder's and Young's inequalities to infer
\begin{align*}
I_2 & = \E \Big[\sup_{t\in[0,T]}\big[ |q^\varepsilon(t)-q(t)|^n+|p^\varepsilon(t)-p(t)|^n\big]\cdot  \mathbf{1}\{\tau^R \le T\}  \Big]\\
&\le C \Big|\E \sup_{t\in[0,T]}\big[ |q^\varepsilon(t)|^{2n}+|p^\varepsilon(t)|^{2n}+|q(t)|^{2n}+| p(t)|^{2n} \big] \Big|^{1/2}\\
&\qquad\times \Big|\P\Big( \sup_{t\in[0,T]}\big[ |q(t)|+|q(t)|^{-1}\big]\geq R  \Big)\Big|^{1/2}.
\end{align*}
In view of Lemma \ref{lem:LE:single-particle:moment-estimate} and Lemma \ref{lem:rLE:single-particle:moment-estimate:sup_[0,T]}, the expectation on the above right-hand side is bounded by
\begin{align*}
\E \sup_{t\in[0,T]}\big[ |q^\varepsilon(t)|^{2n}+|p^\varepsilon(t)|^{2n}+|q(t)|^{2n}+| p(t)|^{2n} \big] \le C(T).
\end{align*}
Also, reasoning as in the proof of Theorem \ref{thm:newtonian-limit}, part 1, let $c_G$ and $C_G$ be the constants as in Lemma \ref{lem:U+G>|q|-log(q)}. For all $R$ large enough, we have
\begin{align*}
& \Big\{ \sup_{t\in[0,T]}\big[ |q(t)|+|q(t)|^{-1}\big]\geq R  \Big\} \\
&\subseteq  \Big\{ \sup_{t\in[0,T]} |q(t)|\geq \frac{R}{2}  \Big\} \cup \Big\{ \sup_{t\in[0,T]}\big[-\log|q(t)|\big]\geq \log R-\log 2  \Big\} \\
&\subseteq  \Big\{ \sup_{t\in[0,T]} U(q(t))+G(q(t)) \ge \frac{c_G}{2C_G}\log R \Big\}.
\end{align*}
So, we invoke Markov's inequality together with Lemma \ref{lem:LE:single-particle:moment-estimate} to deduce
\begin{align*}
\P\Big( \sup_{t\in[0,T]}\big[ |q(t)|+|q(t)|^{-1}\big]\geq R  \Big) \le \frac{2C_G}{c_G\log R} \E \sup_{t\in[0,T]}\big[ U(q(t))+G(q(t))\big]\le \frac{C(T)}{\log R},
\end{align*}
whence
\begin{align*}
I_2 \le \frac{C(T)}{\sqrt{\log R}}.
\end{align*}
Likewise, 
\begin{align*}
I_3 & = \E \Big[\sup_{t\in[0,T]}\big[ |q^\varepsilon(t)-q(t)|^n+|p^\varepsilon(t)-p(t)|^n\big]\cdot  \mathbf{1}\{\tau^R_\varepsilon \le T\}  \Big]\\
&\le C \Big|\E \sup_{t\in[0,T]}\big[ |q^\varepsilon(t)|^{2n}+|p^\varepsilon(t)|^{2n}+|q(t)|^{2n}+| p(t)|^{2n} \big] \Big|^{1/2}\\
&\qquad\times \Big|\P\Big( \sup_{t\in[0,T]}\big[ |q^\varepsilon(t)|+|q^\varepsilon(t)|^{-1}\big]\geq R  \Big)\Big|^{1/2}\\
& \le \frac{C}{\sqrt{\log R}} \Big|\E \sup_{t\in[0,T]}\big[ |q^\varepsilon(t)|^{2n}+|p^\varepsilon(t)|^{2n}+|q(t)|^{2n}+| p(t)|^{2n} \big] \Big|^{1/2}\\
&\qquad \times \Big|\E \sup_{t\in[0,T]}\big[ U(q^\varepsilon(t))+G(q^\varepsilon(t))\big]\Big|^{1/2}\\
&\le \frac{C(T)}{\sqrt{\log R}}.
\end{align*}
Altogether, we collect the estimates on $I_1,I_2$ and $I_3$ to obtain the bound
\begin{align*}
&\E \sup_{t\in[0,T]}\big[ |q^\varepsilon(t)-q(t)|^n+|p^\varepsilon(t)-p(t)|^n\big] \notag \\
&\le I_1+I_2+I_3\le \varepsilon^n C(T,R)+ \frac{C(T)}{\sqrt{\log R}}.
\end{align*}
In turn, by first taking $R$ sufficiently large and then shrinking $\varepsilon$ to zero, we establish the desired limit \eqref{lim:newtonian-limit:single-particle:L^p}. The proof is thus complete.

\end{proof}

Finally, we turn to the auxiliary estimates in Lemma \ref{lem:LE:single-particle:moment-estimate} and Lemma \ref{lem:rLE:single-particle:moment-estimate:sup_[0,T]}. On the one hand, the argument of the former result is essentially the same as in the proof of Lemma \ref{lem:LE:N-particle:moment-estimate} for $N$-particle case, and thus is omitted. On the other hand, the proof of the latter result employs uniform Lyapunov functionals with respect to the parameter $\varepsilon$ and is provided below.

\begin{proof}[Proof of Lemma \ref{lem:rLE:single-particle:moment-estimate:sup_[0,T]}] In order to establish \eqref{ineq:rLE:single-particle:moment-estimate:sup_[0,T]}, we introduce the functional 
\begin{align*}
\Gamma_2(q,p) = \frac{1}{2}\varepsilon\big(U(q)+G(q)\big)^2+ \big(U(q)+G(q)\big)\sqrt{1+\varepsilon|p|^2}+\frac{1}{2}|p|^2.
\end{align*}
From \eqref{eqn:rLE:single-particle:epsilon}, applying It\^o's formula to $\Gamma_2$ gives
\begin{align} \label{eqn:d.Gamma_2(t)}
\d \Gamma_2(t)  & =  -\frac{|p^\varepsilon(t)|^2}{\sqrt{1+\varepsilon|p^\varepsilon(t)|^2}}\d t -\big[U(q^\varepsilon(t))+G(q^\varepsilon(t))\big]\frac{\varepsilon|p^\varepsilon(t)|^2}{1+\varepsilon|p^\varepsilon(t)|^2}\d t \notag \\
&\qquad+ \big[U(q^\varepsilon(t))+G(q^\varepsilon(t))\big]\frac{d\varepsilon+(d-1)\varepsilon^2|p^\varepsilon(t)|^2}{(1+\varepsilon|p^\varepsilon(t)|^2)^{3/2}}\d t+d\d t+\d M_1(t),
\end{align}
where $M_1(t)$ is the semi martingale process given by
\begin{align*}
\d M_1(t)&= \Big\la \big[U(q^\varepsilon(t))+G(q^\varepsilon(t))\big]\frac{\varepsilon p^\varepsilon(t)}{\sqrt{1+\varepsilon|p^\varepsilon(t)|^2}}+p^\varepsilon(t),\sqrt{2}\d W(t)\Big\ra,
\end{align*}
and whose quadratic variation is defined as
\begin{align*}
\d \la M_1\ra(t) = 2\Big| \big[U(q^\varepsilon(t))+G(q^\varepsilon(t))\big]\frac{\varepsilon p^\varepsilon(t)}{\sqrt{1+\varepsilon|p^\varepsilon(t)|^2}}+p^\varepsilon(t)\Big|^2\d t.
\end{align*}
Recalling \eqref{ineq:d.epsilon}, we infer from the expression \eqref{eqn:d.Gamma_2(t)} that
\begin{align*}
&\d\Gamma_2(t)  \le (2d-1)\varepsilon\big[U(q^\varepsilon(t))+G(q^\varepsilon(t))\big]\d t+d\d t+ \d M_1(t),
\end{align*}
and that
\begin{align*}
    \d \la M_1\ra(t)  \le  4\varepsilon \big[U(q^\varepsilon(t))+G(q^\varepsilon(t))\big]^2+4|p^\varepsilon(t)|^2\d t\le 8\Gamma_2(t)\d t.
\end{align*}
As a consequence, for $n\ge 2$, a routine calculation making use of the above estimates gives
\begin{align*}
    \d \Gamma_2(t)^n & = n\Gamma_2(t)^n \d \Gamma_2(t) + \frac{1}{2}n(n-1)\Gamma_2(t)^{n-2}\la \d \Gamma_2(t) ,\d \Gamma_2(t) \ra\\
    &\le n(2d-1)\varepsilon\Gamma_2(t)^{n-1}\big[U(q^\varepsilon(t))+G(q^\varepsilon(t))\big]\d t+nd\Gamma_2(t)^{n-1}\d t\\
    &\qquad + 4n(n-1)\Gamma_2(t)^{n-1} \d t+ n\Gamma_2(t)^{n-1}\d M_1(t).
\end{align*}
In particular, since $U+G\le \Gamma_2$, we have
\begin{align*}
    n(2d-1)\varepsilon\Gamma_2(t)^{n-1}\big[U(q^\varepsilon(t))+G(q^\varepsilon(t))\big] \le n(2d-1)\Gamma_2(t)^{n} ,
\end{align*}
whence
\begin{align} \label{ineq:d.Gamm_2^n}
    \d \Gamma_2(t)^n \le n(2d-1)\Gamma_2(t)^{n} \d t + (nd+4n(n-1))\Gamma_2(t)^{n-1}\d t + n\Gamma_2(t)^{n-1}\d M_1(t).
\end{align}
It follows that
\begin{align*}
    \E\Gamma_2(t)^n \d t\le \Gamma_2(0)^n+ \E\int_0^T n(2d-1)\Gamma_2(t)^{n} + (nd+4n(n-1))\Gamma_2(t)^{n-1}\d t.
\end{align*}
In particular, this implies by virtue of Gronwall's inequality that
\begin{align} \label{ineq:E.int_0^t.Gamma_2^n}
    \E\int_0^T\Gamma_2(t)^{n}\d t\le C,
\end{align}
where $C=C(n,T,q_0,p_0)$ is independent of $n$. 

Turning back to \eqref{ineq:rLE:single-particle:moment-estimate:sup_[0,T]}, we employ \eqref{ineq:d.Gamm_2^n} once again with Burkholder's inequality to deduce
\begin{align*}
    \E\sup_{t\in[0,T]}\Gamma_2(t)^n& \le \Gamma_2(0)^n+ C\int_0^T\Gamma_2(t)^{n}\d t + n\E\sup_{t\in[0,T]}\Big|\int_0^T\Gamma_2(t)^{n-1}\d M_1(t)\Big|\\
    &\le  \Gamma_2(0)^n+ C\E\int_0^T\Gamma_2(t)^{n}\d t+ C\E\int_0^T \Gamma_2(t)^{2n-1}\d t +C.
\end{align*}
This together with \eqref{ineq:E.int_0^t.Gamma_2^n} implies that
\begin{align*}
    \E\sup_{t\in[0,T]}\Gamma_2(t)^n \le C(n,T,q_0,p_0).
\end{align*}
Since $\Gamma_2\ge U+G$, we immediately establish \eqref{ineq:rLE:single-particle:moment-estimate:sup_[0,T]}, as claimed.

\begin{comment}

In view of the exponential martingale inequality \eqref{ineq:exponential-Martingale} applying to $M_1(t)$, we have
\begin{align*}
\E \exp\Big\{\sup_{t\ge 0} \frac{1}{2} M_1(t)-\frac{1}{4} \la M_1\ra(t) \Big\} \le 2,
\end{align*}
implying
\begin{align} \label{ineq:E.exp.sup_[0,T].Gamma_2(t)}
&\E\exp\Big\{ \sup_{t\in[0,T]} \Big[\frac{1}{2} \Gamma_2(t)  -\frac{1}{2}(2d-1)\int_0^t\varepsilon\big[U(q^\varepsilon(s))+G(q^\varepsilon(s))\big]\d s-\frac{1}{4} \la M_1\ra(t) \Big]\Big\} \notag \\
&\qquad\le 2\exp \big\{\Gamma_2(q_0,p_0)+dT\big\}.
\end{align}
Note that
\begin{align*}
\d \la M_1\ra(t)\le 4\varepsilon\big[U(q^\varepsilon(t))+G(q^\varepsilon(t))\big]^2\d t+4|p^\varepsilon(t)|^2\d t\le 8\Gamma_2(t)\d t.
\end{align*}
Combining with \eqref{ineq:E.exp.sup_[0,T].Gamma_2(t)} and Holder's inequality, for all $n\ge 1$, we get
\begin{align*}
\E \sup_{t\in[0,T]}\Gamma_2(t)^n \le C e^{cT}+ C T^{n-1}\int_0^T\E\sup_{s\in[0,t]}\Gamma_2(s)^n \d t,
\end{align*}
for some positive constants $c,C$ independent of $\varepsilon$. Since $\Gamma_2(q,p)\ge U(q)+G(q)+\frac{1}{2}|p|^2$, the above estimate produces \eqref{ineq:rLE:single-particle:moment-estimate:sup_[0,T]} by virtue of Gronwall's inequality. The proof is thus finished.

\end{comment}

\end{proof}

%\section*{Acknowledgments}
%The research of M. H. Duong was supported by EPSRC Grants EP/V038516/1 and  EP/W008041/1. 

%M. H. Duong gratefully acknowledges Yulong Lu for useful discussions in an early stage of the topics of this work. The authors also would like to thank the anonymous reviewers for their valuable comments and suggestions.
\appendix

\section{Auxiliary results}
\label{sec:appendix}

In this section, we collect useful estimates on $\qb\in \D$, including Lemma \ref{lem:|q_i-q_i|} and Lemma \ref{lem:U+G>|q|-log(q)}. In particular, the result of Lemma \ref{lem:|q_i-q_i|} was directly employed to construct 
Lyapunov functions in Section \ref{sec:poly-mixing:Lyapunov:N-particle} 
whereas the result of Lemma \ref{lem:U+G>|q|-log(q)} appeared in the proof of Theorem \ref{thm:newtonian-limit}, presented in Sections \ref{sec:newtonian-limit:proof-of-thm:part-1}-\ref{sec:newtonian-limit:proof-of-thm:part-2}.

\begin{lemma} \label{lem:|q_i-q_i|} For all $\qb \in \D$, $\gamma\in(0,1] $ and $s\ge 0$, the following holds
\begin{align} \label{ineq:<(q_i-q_j)/|q_i-q_j|,q_i-q_k/|q_i-q_k|>}
\sum_{i=1}^N \bigg\la  \sum_{j\neq i}\frac{q_i-q_j}{|q_i-q_j|^{\gamma}},\sum_{k\neq i}\frac{q_i-q_k}{|q_i-q_k|^{s+1}} \bigg\ra \ge 2\sum_{1\le i<j\le N}\frac{1}{|q_i-q_j|^{s+\gamma-1}}.
\end{align}

\end{lemma}

\begin{remark}
The case $\gamma=1$ in Lemma \ref{lem:|q_i-q_i|} was actually proven in \cite[Lemma 4.2]{lu2019geometric}.
\end{remark}

\begin{proof}[Proof of Lemma \ref{lem:|q_i-q_i|}]
We shall follow closely the proof of \cite[Lemma A.2]{duong2024asymptotic} tailored to our settings. By a routine calculation, we have 
\begin{align*}
&\sum_{i=1}^N \bigg\la  \sum_{j\neq i}\frac{q_i-q_j}{|q_i-q_j|^\gamma},\sum_{k\neq i}\frac{q_i-q_k}{|q_i-q_k|^{s+1}} \bigg\ra 
=  2\sum_{1\le i<j\le N}\frac{1}{|q_i-q_j|^{s+\gamma-1}}+\sum_{1\le i<j<k\le N} A_{ijk},
\end{align*}
where
\begin{align*}
A_{ijk}&=  \bigg\la  \frac{q_i-q_j}{|q_i-q_j|^\gamma},\frac{q_i-q_k}{|q_i-q_k|^{s+1}} \bigg\ra + \bigg\la  \frac{q_i-q_k}{|q_i-q_k|^\gamma},\frac{q_i-q_j}{|q_i-q_j|^{s+1}} \bigg\ra \\
&\qquad +\bigg\la  \frac{q_j-q_i}{|q_j-q_i|^\gamma},\frac{q_j-q_k}{|q_j-q_k|^{s+1}} \bigg\ra + \bigg\la  \frac{q_j-q_k}{|q_j-q_k|^\gamma},\frac{q_j-q_i}{|q_j-q_i|^{s+1}} \bigg\ra\\
&\qquad +\bigg\la  \frac{q_k-q_i}{|q_k-q_i|^\gamma},\frac{q_k-q_j}{|q_k-q_j|^{s+1}} \bigg\ra + \bigg\la  \frac{q_k-q_j}{|q_k-q_j|^\gamma},\frac{q_k-q_i}{|q_k-q_i|^{s+1}} \bigg\ra.
\end{align*}
It therefore suffices to prove that $A_{ijk}$ is always non-negative. To this end, denote by $\theta_i,\theta_j,\theta_k$ the angles formed by these points and whose vertices are $q_i,q_j,q_k$, respectively. Observe that $A_{ijk}$ can be recast as follows.
\begin{align*}
A_{ijk} &= \cos \theta_i \bigg[  \frac{|q_i-q_j|^{1-\gamma}}{|q_i-q_k|^s}+  \frac{|q_i-q_k|^{1-\gamma}}{|q_i-q_j|^s} \bigg] +\cos \theta_j \bigg[  \frac{|q_j-q_i|^{1-\gamma}}{|q_j-q_k|^s}+  \frac{|q_j-q_k|^{1-\gamma}}{|q_j-q_i|^s} \bigg] \\
&\qquad +\cos \theta_k \bigg[  \frac{|q_k-q_i|^{1-\gamma}}{|q_k-q_j|^s}+  \frac{|q_k-q_j|^{1-\gamma}}{|q_k-q_i|^s} \bigg].
\end{align*}
Now, there are two cases to be considered depending on the triangle formed by $q_i,q_j,q_k$ in $\rbb^d$.

Case 1: the triangle is either an acute or right triangle. In this case, since $\theta_k,\theta_j\in[0,\pi/2)$ and $\cos$
is decreasing on $[0,\pi/2)$, it is clear that $\cos \theta_i$, $\cos \theta_i$ and $\cos \theta_i$ are both non-negative. This immediately implies that so is $A_{ijk}$.   

Case 2: the triangle is an obtuse triangle.  Without loss of generality, suppose that the angle $\theta_i\in (\pi/2,\pi]$, and thus $\theta_j+\theta_k\in[0,\pi/2)$. In this case, it holds that
\begin{align*}
\min\{\cos(\theta_j),\cos(\theta_k)\}\ge \cos(\theta_j+\theta_k)=-\cos(\theta_i)\ge 0.
\end{align*} 
Also, 
\begin{align*}
|q_j-q_k|\ge \max\{|q_j-q_i|,|q_k-q_i| \}.
\end{align*}
As a consequence, since $\gamma\in (0,1)$, observe that
\begin{align*}
A_{ijk}
 &\ge -| \cos \theta_i| \bigg[  \frac{|q_i-q_j|^{1-\gamma}}{|q_i-q_k|^s}+  \frac{|q_i-q_k|^{1-\gamma}}{|q_i-q_j|^s} \bigg] +\cos \theta_j  \frac{|q_j-q_k|^{1-\gamma}}{|q_j-q_i|^s} +\cos \theta_k   \frac{|q_k-q_j|^{1-\gamma}}{|q_k-q_i|^s}\\
&\ge -| \cos \theta_i| \bigg[  \frac{|q_i-q_j|^{1-\gamma}}{|q_i-q_k|^s}+  \frac{|q_i-q_k|^{1-\gamma}}{|q_i-q_j|^s} \bigg] +\cos \theta_j  \frac{|q_i-q_k|^{1-\gamma}}{|q_j-q_i|^s} +\cos \theta_k   \frac{|q_i-q_j|^{1-\gamma}}{|q_k-q_i|^s}\\
&\ge 0.
\end{align*}
In turn, this produces \eqref{ineq:<(q_i-q_j)/|q_i-q_j|,q_i-q_k/|q_i-q_k|>}, as claimed.

\end{proof}

\begin{lemma} \label{lem:U+G>|q|-log(q)} Under Assumptions \nameref{cond:U} and \nameref{cond:G1}, for all $\qb\in\D$, the following holds
\begin{align}\label{ineq:U+G>|q|-log(q)}
&C_G\bigg[\sum_{i=1}^NU(q_i)+\sum_{1\le i<j\le N}G(q_i-q_j)\bigg] \notag \\
& \ge \sum_{i=1}^N|q_i|-c_G \sum_{1\le i<j\le N}\log|q_i-q_j| \notag \\
&\ge c_G \sum_{i=1}^N|q_i|+c_G\max\{-\log|q_i-q_j|:1\le i<j\le N\},
\end{align}
for some positive constants $c_G$ small and $C_G$ large enough independent of $\qb$. Furthermore, 
\begin{align} \label{ineq:|q|-log|q|>0}
    \frac{1}{2}\sum_{i=1}^N|q_i|-c_G \sum_{1\le i<j\le N}\log|q_i-q_j| \ge 0.
\end{align} 
\end{lemma}
\begin{proof} By taking $c_G$ small enough, the second inequality of \eqref{ineq:U+G>|q|-log(q)} and the lower bound \eqref{ineq:|q|-log|q|>0} are relatively not difficult to verify. In what follows, we shall present the argument for the left-hand side inequality involving $G$.

First of all, we claim that 
\begin{align}\label{ineq:G+|q|>-log|q|}
G(q)+\tilde{C}_G(|q|+1)\ge - \tilde{c}_G\log|q|,\quad q\in\rbb^d\setminus\{0\},
\end{align}
for some positive constants $\tilde{c}_G,\tilde{C}_G$. To see this, there are two cases to be considered depending on the dimension $d$. 

Case 1: $d=1$. In this case, since $G$ is even, we note that 
\begin{align*}
G(q)-G(1)=G(|q|)-G(1) =\int_1^{|q|}G'(s)\d s.
\end{align*}
In view of condition \eqref{cond:G:|grad.G(x)+q/|x|^beta_1|<1/|x|^beta_2} together with the fact that $\beta_2<\beta_1$, we infer
\begin{align*}
\Big|\int_1^{|q|}G'(s)\d s  +a_4\int_1^{|q|}\frac{1}{s^{\beta_1}}\d s\Big|&\le \Big|\int_1^{|q|}G'(s) +\frac{a_4}{s^{\beta_1}}\d s\Big|\\
&\le \Big|\int_1^{|q|}\Big[\frac{|a_5|}{s^{\beta_2}}+a_6\Big]\d s\Big|\\
&\le \frac{1}{2}a_4\Big|\int_1^{|q|}\frac{1}{s^{\beta_1}}\d s\Big|+C|q|+C.
\end{align*}
It follows that
\begin{align*}
G(q)+C|q|+C\ge -\frac{1}{2}a_4\int_1^{|q|}\frac{1}{s^{\beta_1}}\d s.
\end{align*}
Since $\beta_1\ge 1$, this immediately produces \eqref{ineq:G+|q|>-log|q|} and completes case 1.

Case 2: $d\ge 2$. In this case, let $\{e_j\}_{j=1,\dots,d}$ denote the Euclidean basis in $\rbb^d$. Observe that there must exist at least some $e_j$ such that $q$ is not a multiple of $e_j$. Suppose without loss of generality, $q\notin\{c\,e_1:c\in\rbb\}$. In particular, this implies that the linear line $\gamma(s)=s\cdot q+(1-s)e_1$, $s\in\rbb$, does not cross the origin.  As a consequence, we may employ the fundamental theorem of calculus to obtain the identity
\begin{align*}
G(q)-G(e_1) & = \int_0^1 \la \grad G(\gamma(s)),\gamma'(s)\ra\d s \\
& =-a_4\int_0^1 \frac{\la \gamma(s),\gamma'(s)\ra}{|\gamma(s)|^{\beta_1+1}}\d s-a_5\int_0^1 \frac{\la \gamma(s),\gamma'(s)\ra}{|\gamma(s)|^{\beta_2+1}}\d s\\
&\qquad + \int_0^1  \Big\la \grad G(\gamma(s))+\frac{a_4\gamma(s)}{|\gamma(s)|^{\beta_1+1}}+\frac{a_5\gamma(s)}{|\gamma(s)|^{\beta_2+1}},\gamma'(s)\Big\ra\d s. 
\end{align*}
On the one hand, since $\beta_2<\beta_1$ and $\beta_1\ge 1$, we infer that
\begin{align*}
& -a_4\int_0^1 \frac{\la \gamma(s),\gamma'(s)\ra}{|\gamma(s)|^{\beta_1+1}}\d s-a_5\int_0^1 \frac{\la \gamma(s),\gamma'(s)\ra}{|\gamma(s)|^{\beta_2+1}}\d s \\
&\qquad= -\frac{1}{2} a_4\int_0^1 \frac{1}{|\gamma(s)|^{\beta_1+1}}\d (|\gamma(s)|^2)-\frac{1}{2} a_5\int_0^1 \frac{1}{|\gamma(s)|^{\beta_2+1}}\d (|\gamma(s)|^2)\\
&\qquad\ge -c  \int_0^1 \frac{1}{|\gamma(s)|^{\beta_1+1}}\d (|\gamma(s)|^2)+C\\
&\qquad\ge -c\log|q|+C.
\end{align*}
On the other hand, we invoke condition \eqref{cond:G:|grad.G(x)+q/|x|^beta_1|<1/|x|^beta_2} to infer
\begin{align*}
\Big|\int_0^1  \Big\la \grad G(\gamma(s))+\frac{a_4\gamma(s)}{|\gamma(s)|^{\beta_1+1}}+\frac{a_5\gamma(s)}{|\gamma(s)|^{\beta_2+1}},\gamma'(s)\Big\ra\d s\Big| \le a_6 \int_0^11 |\gamma'(s)|\d s= a_6|q-e_1|\le a_6(|q|+1).
\end{align*}
It follows that
\begin{align*}
G(q)-G(e_1) \ge -c\log|q|+C-a_6(|q|+1).
\end{align*}
This implies \eqref{ineq:G+|q|>-log|q|}, which completes case 2 for all dimensions $d\ge 2$. 

Turning back to \eqref{ineq:U+G>|q|-log(q)}, thanks to condition \eqref{cond:U:U(x)=O(x^lambda+1)}, we observe that
\begin{align*}
\frac{c_G}{\tilde{c}_G}\sum_{i=1}^N U(q_i)\ge \sum_{i=1}^N\Big[ |q_i|+\frac{c_G\tilde{C}_G}{\tilde{c}_G}(|q_i|+1)\Big] -C.
\end{align*}
Together with \eqref{ineq:G+|q|>-log|q|}, we get
\begin{align*}
&\frac{c_G}{\tilde{c}_G}\bigg[\sum_{i=1}^N U(q_i)+\sum_{1\le i<j\le N}G(q_i-q_j)\bigg]\\
&\qquad\ge \sum_{i=1}^N |q_i|+ \frac{c_G}{\tilde{c}_G}\sum_{1\le i<j\le N}\Big[ G(q_i-q_j)+\tilde{C}_G(|q_i-q_j|+1)\Big]-C\\
&\qquad\ge \sum_{i=1}^N |q_i|-c_g \sum_{1\le i<j\le N} \log|q_i-q_j|-C.
\end{align*}
Since it is assumed that $\sum_{i=1}^N U(q_i)+\sum_{1\le i<j\le N}G(q_i-q_j)\ge 1$, cf. \eqref{cond:U+G>1:N>1}, the above estimate immediately implies \eqref{ineq:U+G>|q|-log(q)}, thereby finishing the proof. 

\end{proof}

\section{Well-posedness of \eqref{eqn:rLE:N-particle:epsilon} } \label{sec:well-posed}

In this section, we present the proof of Proposition \ref{prop:rLE:N-particle:well-posed} giving the global well-posedness of \eqref{eqn:rLE:N-particle:epsilon}. Following Lyapunov stability methods for establishing strong wellposedness of SDEs
with singular coefficients (see e.g., \cite{veretennikov2024lyapunov} for a general presentation), the argument essentially consists of two main steps as follows.

Step 1: we truncate the nonlinearities in \eqref{eqn:rLE:N-particle:epsilon} using suitable smooth cut-off functions. This results in a Lipschitz system, which is actually given by \eqref{eqn:rLE:N-particle:epsilon:truncating}. By a standard argument, it can be shown that this truncated system has a strong solution, thereby giving the local well-posedness of\eqref{eqn:rLE:N-particle:epsilon}.

Step 2: we remove the Lipschitz constraint by exploiting the Hamiltonian structures, i.e., the function $H$ and $H_N$ respectively defined in \eqref{form:H} and \eqref{form:H_N} corresponding to $N= 1$ and $N\ge 2$. This ultimately establishes the well-posedness of \eqref{eqn:rLE:N-particle:epsilon}.

\begin{proof}[Proof of Proposition \ref{prop:rLE:N-particle:well-posed}] In what follows, we proceed to prove the well-posedness of the multi-particle case $N\ge 2$. The case $N=1$ can be established using a similar argument.

Letting $R>0$, recall the smooth function $\theta_R$ defined in \eqref{form:theta_R} and the truncated system \eqref{eqn:rLE:N-particle:epsilon:truncating}. Observe that the drift terms of this system are Lipschitz and bounded functions. Hence, we may employ the argument of \cite{veretennikov1981strong,veretennikov2024lyapunov} to obtain a global solution $(\qb^{\varepsilon,R},\pb^{\varepsilon,R})$ of \eqref{eqn:rLE:N-particle:epsilon:truncating}. As a consequence, there exists unique strong solutions of \eqref{eqn:rLE:N-particle:epsilon} until the stopping time $\sigma^R_\varepsilon$ defined in \eqref{form:stopping-time:sigma^R_varepsilon}, i.e.,
\begin{align*}
    \sigma^R_\varepsilon = \inf\Big\{t\ge 0: \sum_{i=1}^N |q_i^\varepsilon(t)|+\sum_{1\le i<j\le N} |q_i^\varepsilon(t)-q_j^\varepsilon(t)|^{-1} \ge R \Big\}.
\end{align*}
Since $\sigma^R_\varepsilon$ is non-decreasing, we may define $\lim_{R\to\infty}\sigma^R_\varepsilon=\sigma_\varepsilon$. It remains to prove that $\sigma_\varepsilon= \infty$ a.s. To this end, we invoke \eqref{eqn:L_N.H_N} and the first inequality of \eqref{ineq:d.epsilon} to infer
\begin{align*}
    \d H_N(t\mi\sigma^R_\varepsilon ) \le C\d t+\sqrt{2}\sum_{i=1}^N \frac{\varepsilon p_i(t\mi \sigma^R_\varepsilon)}{\sqrt{1+\varepsilon|p_i^\varepsilon(t\mi\sigma_\varepsilon^R)|^2}}\d W_i,
\end{align*}
where $C$ is a positive constant independent of $R$. In view of Burkh\"older-Davis-Gundy's inequality and the second inequality of \eqref{ineq:d.epsilon}, we deduce
\begin{align*}
\E\sup_{t\in[0,T]}H_N(t\mi\sigma^R_\varepsilon ) \le H_N(\qb_0,\pb_0)+CT.
\end{align*}
In particular, this implies that
\begin{align*}
    \E\Big[\mathbf{1}\{\sigma^R_\varepsilon\le  T\} \sup_{t\in[0,T]}H_N(t\mi\sigma^R_\varepsilon )\Big] \le H_N(\qb_0,\pb_0)+CT.
\end{align*}
Using an argument similar to that of \eqref{ineq:sigma^R<T}, we have for all $R$ large enough
\begin{align*} 
    \big\{  \sigma^R_\varepsilon\le T\big\} & = \Big\{\sup_{t\in[0,T]}\sum_{i=1}^N |q_i^\varepsilon(t\mi \sigma^R_\varepsilon)|+\sum_{1\le i<j\le N} |q_i^\varepsilon(t\mi \sigma^R_\varepsilon)-q_j^\varepsilon(t\mi \sigma^R_\varepsilon)|^{-1} \ge R\Big\}\\
    &\subseteq\Big\{ \sup_{t\in[0,T]}\sum_{i=1}^N U(q_i(t\mi \sigma^R_\varepsilon))+\sum_{1\le i<j\le N}G(q_i(t\mi \sigma^R_\varepsilon)-q_i(t\mi \sigma^R_\varepsilon)) \geq \frac{c_G}{C_G}\log\Big(\frac{R}{N^2}\Big)\Big\}.
\end{align*}
In the above, $c_G$ and $C_G$ are the constants as in Lemma \ref{lem:U+G>|q|-log(q)}. From the expression of $H_N$ in \eqref{form:H_N}, we get
\begin{align*} 
    \big\{  \sigma^R_\varepsilon<T\big\} 
    &\subseteq\Big\{ \sup_{t\in[0,T]}H_N(t\mi \sigma^R_\varepsilon) \geq \varepsilon\frac{c_G}{C_G}\log\Big(\frac{R}{N^2}\Big)\Big\}.
\end{align*}
It follows that
\begin{align*}
    \P(\sigma^R_\varepsilon\le T) \le \frac{H_N(\qb_0,\pb_0)+CT}{\log R}.
\end{align*}
Sending $R$ to infinity yields
\begin{align*}
    \P(\sigma_\varepsilon\le T)  =0.
\end{align*}
Since this holds for all $T>0$, we deduce $\sigma_\varepsilon = \infty$ a.s. The proof is thus finished.

\end{proof}
\section*{Acknowledgment} The research of HD was supported by EPSRC Grant EP/Y008561/1. We would like to thank the anonymous referees for the constructive comments and suggestions, which have helped us to improve the presentation of the paper.

\bibliographystyle{abbrv}
{\footnotesize\bibliography{GLE-bib}}

\end{document}